\documentclass{article}
\usepackage{url}
\usepackage{graphicx, float}
\usepackage{amsmath,amssymb,amsfonts}%
\usepackage{amsthm}%
\usepackage{mathtools}
\usepackage{authblk}
\usepackage{xcolor}

\mathtoolsset{showonlyrefs=true}
\usepackage{hyperref}
\usepackage{amsrefs}
\usepackage[margin=1.0in]{geometry}

\newcommand{\R}{\mathbb{R}}

\newcommand{\Z}{\mathbb{Z}}

\newcommand{\T}{\mathbb{T}}
\newcommand{\D}{\mathcal{D}}
\newcommand{\Hh}{\mathcal{H}}
\newcommand{\E}{\mathbb{E}}
\newcommand{\Ee}{\mathcal{E}}
\newcommand{\Pb}{\mathbb{P}}

\newtheorem{theorem}{Theorem}[section]
\newtheorem{lemma}[theorem]{Lemma}
\newtheorem{proposition}[theorem]{Proposition}
\newtheorem{corollary}[theorem]{Corollary}

\theoremstyle{definition}
\newtheorem{definition}[theorem]{Definition}

\theoremstyle{remark}
\newtheorem{remark}[theorem]{Remark}

\numberwithin{equation}{section}

\title{Stability of Stochastically Driven Couette Flow in 2D with Navier Boundary Conditions at high Reynolds number via Averaging Principle}
\author[1]{Ryan Arbon}
\author[1]{Jacob Bedrossian}
\affil[1]{Department of Mathematics, University of California, Los Angeles, 90095, CA. Email: rarbon@math.ucla.edu, jacob@math.ucla.edu}

\begin{document}
\maketitle




\begin{abstract}

We characterize the behavior of stochastic Navier-Stokes on $\T \times [-1,1]$ with Navier boundary conditions at high Reynolds number when initialized near Couette flow subject to small additive stochastic forcing. We take additive noise of strength $\nu^{5/6} \Phi dV_t + \nu^{2/3+\alpha} \Psi dW_t$, where $\Phi dV_t$ has spatial correlation in $H_0^3$ and acts only on $x$-independent modes of the vorticity, while $\Psi dW_t$ has spatial correlation in a lower order, anisotropic, Sobolev space $\mathcal{H}$ and acts on $x$-dependent-modes. We take the initial $x$-independent modes in the perturbation to be small in $H_0^3$ in a $\nu$-independent sense, while the non-zero $x$-modes are taken to be $O(\nu^{1/2 + \alpha})$ in $\Hh$. The parameter $\alpha$ is taken to be $\alpha > 1/12$. Letting $\omega$ solve the resulting perturbation equation, we split $\omega$ into the zero $x$-modes $\omega_0$ and the non-zero $x$-modes $\omega_{\neq}$. 
We demonstrate an averaging principle holds wherein $\omega_{\neq}$ is the fast variable and $\omega_0$ is the slow variable, deriving a closed nonlinear evolution equation on $\omega_0$ that holds over long time-scales (while the fast $\omega_{\neq}$ modes solve a `pseudo-linearized' equation to leading order with dynamics dominated by inviscid damping and enhanced dissipation).
This work can also be considered the stochastic analogue of the stability threshold problem for shear flows. Furthermore, we explain the connections to the Stochastic Structural Stability Theory (S3T) in the physics literature.
\end{abstract}

\maketitle

\pagestyle{plain}


\setcounter{tocdepth}{2}
\tableofcontents

\section{Introduction}\label{introduction}
Consider the incompressible stochastic Navier-Stokes system on $\T \times [-1,1]$ with Navier boundary conditions, expressed in vorticity formulation as
\begin{equation}\label{base_navier_stokes}
    \begin{cases}
                d w = \nu \Delta w dt - v \cdot \nabla wdt + \nu^{5/6} \Phi dV_t + \nu^{2/3 + \alpha}\Psi dW_t,\\
        v =  \nabla^{\perp} \Delta^{-1} w = (\partial_y \Delta^{-1} w, -\partial_x \Delta^{-1} w), \; \; w(t,x,y=\pm1) = 1,\\
        w(t = 0,x,y) = w_{in}(x,y).
    \end{cases}    
\end{equation}
Here, $v = (v^1, v^2)$ is the velocity field, $w$ is the scalar vorticity, $\nu >0$ is the viscosity. We take the torus to be length $[0,2\pi]$ with periodic boundary conditions. The noise terms are white-in-time and colored-in-space, and $V_t$ is taken independent of $W_t$. The noise $\nu^{5/6} \Phi dV_t$ acts only on the $x$-independent component of $w$, while $\nu^{2/3 + \alpha}\Psi dW_t$ acts only on the $x$-dependent modes of $w$. The parameter $\alpha > 1/12$ is related to a rescaling defined in \eqref{eqn: rescaling}. We consider $\Phi dV_t$ to have spatial correlation in $H_0^3([-1,1])$, while $\Psi dW_t$ has spatial correlation in a lower order anisotropic Sobolev space. We will make this precise in Subsection \ref{section: definitions}.  The well-posedness of \eqref{base_navier_stokes} under suitable assumptions was considered in \cite{cipriano2015inviscid}. In the deterministic setting of $\Phi \equiv 0,$  $\Psi \equiv 0$, the unique steady state of \eqref{base_navier_stokes} is the Couette flow
\begin{equation}
    w_E \coloneqq 1, \;\; v_E \coloneqq (y,0).
\end{equation}
We consider the stability of $w_E$ under a small initial perturbation, in the manner of \cite{bedrossian2023stability}. Specifically, we write $w_{in}(x,y) = w_E + \mathcal{W}_{in}(y)+ \omega_{in}(x,y)$. We view $\mathcal{W}_{in}$ as being more regular than $\omega_{in}$, and small in a $\nu$-independent manner, while $\omega_{in}$ will be small relative to $\nu$.  We let $\mathcal{W}$ be the unique solution to the following stochastic heat equation:
\begin{equation}
\begin{cases}
    d\mathcal{W}(t) = \nu \partial_y^2 \mathcal{W}(t) dt + \nu^{5/6} \Phi dV_t,\\
    \mathcal{W}(0) = \mathcal{W}_{in} \in H_0^3([-1,1]).
\end{cases}
\end{equation}
The shear flow $\mathcal{U}$ corresponding to $1 +\mathcal{W}$ is given by the Biot-Savart law:
\begin{equation}
    \mathcal{U}(t,y) \coloneqq y + \partial_y \int_{-1}^1 G(y,y') \mathcal{W}(t,y') dy',
\end{equation}
where $G$ is the Green's function for $\partial_{yy}^{-1}$ on $[-1,1]$ with homogeneous Dirichlet boundary conditions. We write $w$ as $w(t,x,y) = 1 + \mathcal{W}(t,y) + \omega(t,x,y)$ and arrive  at the perturbation equation for $\omega$:
\begin{equation}\label{original_couette}
    \begin{cases}
        d \omega = \left(\nu \Delta \omega  - \mathcal{U} \partial_x \omega + \mathcal{U}'' \partial_x \Delta^{-1} \omega- u \cdot \nabla \omega \right)dt+ \nu^{2/3 + \alpha} \Psi dW_t,\\
        \nabla \cdot u =0, u = \nabla^\perp \Delta^{-1} \omega,\\
        \omega(t,x, y = \pm 1) = 0.
    \end{cases}
\end{equation}
 This is a stochastic version of the nearly-Couette system in \cite{bedrossian2023stability}, with the randomness entering through the forcing $\Psi dW_t$ as well as the background flow $\mathcal{U}$. We extract the zero and non-zero $x$-modes of $\omega$ as $\omega_0(t,y) \coloneqq \int_\T \omega(t,x,y) dx$ and $\omega_{\neq} \coloneqq \omega - \omega_0$. In general, we set $P_{\neq} f \coloneqq f - \int_\T f(x) dx$. We wish to characterize the behavior of
 $$w_0 = 1 + \mathcal{W} + \omega_0, \; \; v_0 = (\mathcal{U},0) + u_0$$
as $\nu \to 0$. More precisely, we seek to establish stability criteria on long time scales of order $O(\nu^{\gamma-1})$.

The stability of the Couette flow in the deterministic setting, i.e $\Phi \equiv 0,$ $\Psi \equiv 0$, has been well studied, with papers examining the stability properties of the Couette flow with various domains, boundary conditions, and regularities; see e.g. \cites{arbon_bedrossian, bedrossian2016enhanced, bedrossian2023stability, bedrossian2018sobolev, bian2024boundary, deng2023long, li2022asymptotic, li2025stability, li2024dynamical, masmoudi2020enhanced, masmoudi2022stability, wei2023nonlinear} and the references therein. One of the key findings of these works is the \textit{enhanced dissipation} of $\omega_{\neq}$. For example, \cite{bedrossian2023stability} found that for the nearly-Couette system \eqref{original_couette} with $\Phi \equiv 0$, $\Psi \equiv 0$, that if $\mathcal{W}_{in}$ and $\omega_{in}$ are sufficiently small in suitable topologies, then the non-zero frequencies $\omega_{\neq}$ decay at the enhanced rate of $\exp(-\nu^{1/3} t)$. This contrasts with $\omega_0$, which decays in line with the heat equation like $\exp(-\nu t)$.  Hence we may expect \eqref{original_couette} to behave as a \textit{fast-slow} system in the inviscid limit of $\nu \to 0$. The object of this paper is to make this rigorous with an appropriate averaging principle, showing convergence of $\omega_0$ to an averaged process $\bar{\omega}_0$ as $\nu \to 0$, and similarly
$$w_0 \to \bar{w}_0 = 1 + \mathcal{W} + \bar{\omega}_0, \;\;\; v_0 \to \bar{v}_0 = (\mathcal{U},0) + \bar{u}_0.$$
First, we split $\omega_0$ and $\omega_{\neq}$ and re-write the system \eqref{original_couette} as
\begin{equation}\label{modified_couette}
    \begin{cases}
        d \omega_0 = \left(\nu \partial_y^2 \omega_0 - \int_{\mathbb{T}} u_{\neq} \cdot \nabla\omega_{\neq}dx\right) dt\\
        d\omega_{\neq} =\biggl(\nu \Delta \omega_{\neq}- \mathcal{U} \partial_x \omega_{\neq} + \mathcal{U}'' \partial_x \Delta^{-1} \omega_{\neq}- P_{\neq}(u \cdot \nabla \omega) \biggr)dt+ \nu^{2/3 + \alpha} \Psi dW_t,\\
        \nabla \cdot u =0, u = \nabla^\perp \Delta^{-1} \omega,\\
        \omega(t,x, y = \pm 1) = 0.
    \end{cases}
\end{equation}
The nonlinear terms in \eqref{modified_couette} correspond to breaking up $u \cdot \nabla \omega$ into its $x$-average and the $x$-dependent fluctuations. Since the enhanced dissipation is only observed for suitably small $\omega$ (in a $\nu$-dependent manner), we rescale both in amplitude and in time. We fix $\alpha > 1/12$, $\beta> 0$, and $\gamma \in [0, 1/3)$ such that 
\begin{equation}\label{scaling_constraint}
    2\alpha - \beta + \gamma/2 = 1/3.
\end{equation} 
We then define the rescaled processes
\begin{equation}\label{eqn: rescaling}
    X_t \coloneqq \nu^{-1/2 - \beta}\omega_0(t/\nu^{1-\gamma},x,y), \;\; Y_t \coloneqq \nu^{-1/2 - \alpha} \omega_{\neq}(t/\nu^{1-\gamma}, x, y), \;\; U_t \coloneqq \mathcal{U}(t/\nu^{1-\gamma},y),
\end{equation}
which yields the system
\begin{equation}\label{fast_slow_syst}
    \begin{cases}
        d X^\nu_t = \left(\nu^\gamma \partial_y^2 X^\nu_t - \nu^{\gamma/2-1/6}b_0(Y_t^\nu)\right) dt,\\
        d Y_t^\nu = \nu^\gamma(\Delta Y_t^\nu - \nu^{-1} U_t \partial_x Y_t^\nu + \nu^{-1} U_t''\partial_x \Delta^{-1} Y_t^\nu)dt\\
        \;\;\;\;\;\;\;\;\;\;\;- \nu^{\beta + \gamma -1/2} b_m(X_t^\nu, Y_t^\nu)dt - \nu^{\alpha + \gamma -1/2}b_{\neq}(Y_t^\nu)dt + \nu^{-1/3 + \gamma/2} \Psi dW_t,\\
        Y_t(x,y = \pm 1) = X_t(x,y = \pm 1) = 0,\\
        (x,y) \in \T \times [-1,1].
    \end{cases}
\end{equation}
Here $b_0$, $b_m$, $b_{\neq}$ are the projections of the nonlinearity $u \cdot \nabla \omega$ and are defined via
\begin{equation}\label{def_of_b_nonlins}
\begin{split}
        b_0(Y_t^\nu) &\coloneqq - \partial_y \int_\T (\partial_x \Delta^{-1} Y_t^\nu)(x) Y_t^\nu(x) dx,\\
        b_m(X_t^\nu ,Y_t^\nu) &\coloneqq (\partial_y \Delta^{-1} X_t^\nu) (\partial_x Y_t^\nu) - (\partial_x \Delta^{-1} Y_t^\nu)(\partial_y X_t^\nu),\\
        b_{\neq}(Y_t^\nu) &\coloneqq P_{\neq} ((\nabla^\perp \Delta^{-1}Y_t^\nu) \cdot \nabla Y_t^\nu).
\end{split}
\end{equation}
The subscripts indicate that $b_0$ considers the zero-frequencies of the self-interaction of $Y$, $b_m$ considers the mixed interaction between $X$ and $Y$, and $b_{\neq}$ considers pure non-zero interactions. Note that $b_0$ is simply a re-writing of $\int_{\T} u_{\neq} \cdot \nabla \omega_{\neq} dx$ in the re-scaled variables, while $b_m$ and $b_{\neq}$ come from expanding $P_{\neq}(u \cdot \nabla \omega)$ to get the exact dependence on zero and non-zero modes. For motivation behind the re-scaling, see Section \ref{subsection: quant_stability}. The main theorem of this paper is that under suitable assumptions, $X_t^\nu$ converges in the $\mathbb{E}||\cdot||_{L^\infty([0,T])H^{-\theta}([-1,1])}$ norm, $\theta > 0$, to the solution of the averaged equation
\begin{equation}
    d\bar{X}_t = \nu^{\gamma} \partial_y^2 \bar{X}_t^\nu dt + \nu^{\gamma/2 - 1/6} \bar{b}_0(U_t,\bar{X}_t^\nu)dt,
\end{equation}
where $\bar{b}_0(U,X) = \int_{\mathcal{H}_{\neq}} b_0(Y) \mu_{\nu}^{U,X}(dY)$, for $\mathcal{H}_{\neq}$ a suitable Hilbert space to be defined later, and $\mu_{\nu}^{U,X}$ is the invariant measure of the linear stochastic PDE
\begin{equation}
    \begin{split}
        dZ_t &= \nu^{2/3}(\Delta Z_t - \nu^{-1} U \partial_x Z_t + \nu^{-1} U'' \partial_x \Delta^{-1} Z_t)dt- \nu^{\beta + 1/6}b_m(X,Z_t) dt + \Psi d\bar{W}_t.
    \end{split}
\end{equation}
The term $\Psi d\bar{W}_t$ is i.i.d with $\Psi dW_t$. See Subsection \ref{section: definitions} for precise definitions. Subsection \ref{subsection: theorems}  will contain the theorem statement. We also will state Corollary \ref{main_corollary}, which will describe the behavior of $\omega_0$ in the original variables, and we will discuss the implications for the vorticity $w$ of \eqref{base_navier_stokes}.

This paper lies at the intersection of quantitative stability theory of shear flows and averaging principles of multiscale systems, and to the authors' knowledge is the first result of this type. As such we will review relevant aspects from both areas in the literature in Section \ref{section: literature_review}. This will give further context and motivation to our main result.

\subsection{Assumptions and Problem Set-Up}\label{section: definitions}

We will pursue the averaging principle by adapting the hypocoercive energy structure from \cite{bedrossian2023stability} to the stochastic setting and employing a Khasminskii discretization. Hence we define the Hilbert space $\mathcal{H}$ via
\begin{equation}
    \mathcal{H} = \{ f \in L^2(\T \times[-1,1] \;\vert \; f(x,y= \pm1) = 0, \;||f||_{\mathcal{H}}^2 \coloneqq \sum_{0 \leq j \leq 1} \big|\big| (\nu^{1/3} \partial_y)^j \langle \partial_x\rangle^{m - j/3} f \big| \big|_{L^2(\T \times [-1,1])}^2 < \infty \},
\end{equation}
where the vanishing on the boundary is understood in the trace sense. The regularity parameter $m$ is taken to be $m \in (2/3,1)$. We will distinguish between the zero and non-zero $x$-Fourier components of elements of $\mathcal{H}$ and define
\begin{equation}
    \mathcal{H}_{\neq} \coloneqq \{ f \in \mathcal{H} \; \vert \; P_{\neq} f = f\}, \;\; \mathcal{H}_0 \coloneqq \{f \in \mathcal{H} \; \vert \; P_{\neq} f = 0\},
\end{equation}
so that $\mathcal{H} = \mathcal{H}_0 \oplus \mathcal{H}_{\neq}$. In this paper, we will also be working with various isotropic (fractional) Sobolev norms. For $s \in \R$, we define $|\partial_y|^s$ via the discrete sine transform as
\begin{equation}\label{def: spectral_derivative}
    |\partial_y|^s f \coloneqq (-\Delta)^{s/2} f = \sum_{j = 1}^\infty \left(\frac{j \pi}{2}\right)^s \langle f, e_{0,j}\rangle_{L^2} e_{0,j},
\end{equation}
where $e_{0,j} = \sin\left( \frac{j \pi}{2} ( y+1)\right)$ for $j \in \Z_+$, the set of positive integers. 
For $s > 0$, we let $H_0^s$ denote the closure of $C_c^\infty((-1,1))$ in the norm 
$$||f||_{H^s}^2 \coloneqq \sum_{j=1}^\infty \left(\frac{j \pi}{2}\right)^{2s} |\langle f, e_{0,j}\rangle_{L^2}|^2 .$$
For negative $s \in \R$, we let $H^s$ be the space of distributions admitting a convergent expansion in the $e_{0,j}$ basis with respect to the $||\cdot||_{H^s}$ norm. Equivalently, $H^s$ is the dual space (in the $L^2$ pairing sense) of $H_0^{-s}$ for negative $s$. We let $L^2_0$ be the closure of the span of the $e_{0,j}$ in $L^2$.

We briefly recall the following properties of $\partial_y^2$ with Dirichlet boundary conditions. For any $t >0$, the following holds:
\begin{equation}\label{heat_propogator}
\begin{split}
        \begin{cases}|| e^{t\partial_y^2} f||_{H^s} \leq C(s)|t|^{-s/2} ||f||_{L^2}, & s \geq 0, \;\;f \in L^2_0,\\
        || (e^{t\partial_y^2}  -I)f||_{L^2} \leq C(s)|t|^{s/2}||f||_{H^s}, & s \geq 0, \;\; f \in H_0^s.
        \end{cases}
\end{split}
\end{equation}
We will also have occasion to consider not just functions $f \in H^s_0$, but also their derivatives $\partial_y$. The derivative of mixed type $\partial_y |\partial_y|^s f$ is well-defined so long as $f \in H_0^s$, $s \geq 0$ (or $f \in H^s$, $s < 0$), and $|\partial_y|^s f \in H^1$. We will be careful to avoid writing $|\partial_y|^s \partial_y$, since the derivative of a function vanishing on the boundary does not necessarily vanish on the boundary itself. We have the following estimates:
\begin{equation}\label{heat_propogator_2}
    \begin{cases}
        |||\partial_y|^s e^{t \partial_y^2} f||_{L^2} \leq C(s) |t|^{-s/2 - 1/2}|| |\partial_y|^{-1} f||_{L^2}, & s \in [0,1), \;\;f \in L^2_0,\\
        || \partial_y |\partial_y|^s e^{t \partial_y^2}  f||_{L^2} \leq C(s) |t|^{-s/2 - 1/2}|| f||_{L^2}, & s \in [0,1), \;\; f \in L^2_0,\\
        || \partial_y (e^{t\partial_y^2}  -I)f||_{L^2} \leq C(s)|t|^{s/2}|| \partial_y |\partial_y|^s f||_{L^2}, & s \geq 0, \;\;f \in H_0^s, \;\; |\partial_y|^s f \in H^1.
    \end{cases}
\end{equation}
The proofs of all of the above statements follow from the spectral calculus.

We now define the probabilistic structure. Let $(\Omega, \mathcal{F}, \{\mathcal{F}_t\}_{t \geq 0}, \Pb)$ be a stochastic basis. Define for $(k,j) \in \Z  \times \Z_+$ the function $e_{k,j} \coloneqq e^{ik x} \sin\left( \frac{j \pi}{2} ( y+1)\right)$. Let $\{W_t^{k,j}\}_{(k,j) \in \Z \setminus \{0\} \times \Z_+}$ and $\{V_t^{j}\}_{j \in \Z_+}$ be independent collections of independent one-dimensional Brownian motions. Then
\begin{equation}
    dW_t \coloneqq \sum_{ (k, j) \in \Z \setminus\{0\} \, \times \, \Z_+} e_{k,j} dW_t^{k,j}, \; \; \; dV_t \coloneqq \sum_{j \in \Z_+} e_{0,j} dV_t^j
\end{equation}
define cylindrical Wiener processes on $L^2(\T \times [-1,1])$ and $L^2([-1,1])$, respectively. We let $\Phi : L^2([-1,1]) \to H_0^3([-1,1])$ and $\Psi : L^2(\T \times [-1,1]) \to \mathcal{H}_{\neq}$ be of trace class. For simplicity, we write $\Phi$ and $\Psi$ as diagonal operators acting on the Fourier basis, so that our key assumptions on the noise become:
\begin{equation}\label{noise_assumptions}
    \begin{cases}
        ||\Psi||^2 \coloneqq \sum_{(k,j) \in \Z \setminus\{0\} \, \times \, \Z_+} |k|^{2m}|\psi_{k,j}|^2(1 + \frac{\pi^2}{4}\nu^{2/3}|k|^{-2/3} j^2) < \infty,\\
        ||\Phi||^2 \coloneqq \sum_{j \in \Z_+} j^6 \frac{\pi^6}{64}|\phi_j|^2 < \infty.
    \end{cases}
\end{equation}
Under these assumption, $\Phi dV_t$ has sample paths almost surely in $H_0^3([-1,1])$ and $\Psi dW_t$ has sample paths almost surely in $\mathcal{H}_{\neq}$. Related to the background vorticity $\mathcal{W}$, we define the stopping time $\sigma$:
\begin{equation}
    \sigma \coloneqq \inf\{ t\in [0,T] \; | \; ||\mathcal{W}(t/\nu^{1-\gamma})||_{H^3} \geq 2 c_0\},
\end{equation}
where $c_0$ is a small constant given in Theorem \ref{main_theorem}, which depends on $m$ and on no other parameters. We include this stopping time since we only have good control over the various processes when $U-y$ is sufficiently small. 

The system \eqref{original_couette} is well-posed with a solution in the following sense.
\begin{definition}\label{sense_of_soln}
    Let $\omega(0) \in L^2(\T \times [-1,1])$. We say system \eqref{original_couette} has an analytically weak, probabilistically strong solution if there exists $\omega\in C([0,T];L^2(\T \times [-1,1])) \cap L^2([0,T];H^1(\T \times [-1,1]))$ such that for any $t \in [0,T]$ and $\phi \in D(-\Delta)$, the following holds $\mathbb{P}$-a.s:
    \begin{equation}\begin{split}
                \langle \omega, \phi\rangle_{L^2} &= \langle \omega(0), \phi\rangle_{L^2} + \int_0^t \langle \omega(s), (\nu\Delta-\mathcal{U}\partial_x + \partial_x \Delta^{-1} \mathcal{U}'')\phi\rangle_{L^2} ds\\ &\quad- \int_0^t \langle  \omega(s), (\nabla^\perp \Delta^{-1} \omega(s))\cdot \nabla\phi \rangle_{L^2} ds + \int_0^t \langle \nu^{2/3+\alpha} \Psi dW_s, \omega(s)\rangle_{L^2}.
    \end{split}
    \end{equation}
\end{definition}
For any $\omega(0) \in \mathcal{H}$, the assumptions \eqref{noise_assumptions} on $\Psi$ ensure that a unique solution $\omega$ to \eqref{original_couette} in the sense of Definition \ref{sense_of_soln} exists and is unique (see Theorem \ref{well_posed_full} in Appendix \ref{appendix: a}). The method of proof is similar to the proof of well-posedness for the stochastic Navier-Stokes system with Navier boundary conditions \eqref{base_navier_stokes}, which is essentially contained in \cite{cipriano2015inviscid}. In fact the Dirichlet, or Lions, boundary conditions on $\omega$ enable one to mimic the proof of well-posedness on $\mathbb{T}^2$ (see \cite{kuksin2012mathematics}). Not only does Theorem \ref{well_posed_full} show that $(X_t^\nu, Y_t^\nu) \in L^2$, it shows that for small initial data, $(X_t^\nu, Y_t^\nu) \in \mathcal{H}$ at least locally up to a positive stopping time, and it establishes the use of the infinite dimensional It\^o formula on the relevant norms. Section \ref{energy_section} will give tighter bounds on the size of the solution in $\mathcal{H}$, and show that $(X_t^\nu, Y_t^\nu) \in \mathcal{H}$ on all of $[0,T]$ with high probability. Next, we define the fast system with frozen slow component: for fixed $U \in H^4$, $X \in \mathcal{H}_0$, and $Y \in \Hh_{\neq}$, we let $Z^{U,X,Y}$ solve
\begin{equation}\label{frozen_system}
    \begin{cases}
        &d Z_t^{U,X,Y} = \nu^{2/3}\Delta Z_t^{U,X,Y} dt - \nu^{-1/3} U\partial_x Z_t^{U,X,Y} dt + \nu^{-1/3} U'' \partial_x \Delta^{-1}Z_t^{U,X,Y}dt\\
        &\quad \quad\quad\quad\quad- \nu^{\beta + 1/6} b_m(X, Z_t^{U,X,Y}) dt + \Psi d\bar{W}_t,\\
        &Z_0^{U,X,Y} = Y \in \mathcal{H}_{\neq}, \;\; X\in \mathcal{H}_0, \;\; U\in H^4.
    \end{cases}
\end{equation}
We will see from estimates in Section \ref{fast_with_frozen_section} that for $X$ sufficiently small in $\mathcal{H}_0$ and $||U-y||_{H^4}$ sufficiently small, \eqref{frozen_system} is globally well-posed in $\mathcal{H}_{\neq}$, and has a unique stationary measure $\mu_\nu^{U, X}(dY)$. We let $P_t^{U,X,Y}$ be the transition semi-group of the frozen system, which acts as $P_t^{U,X,Y}(f) = \E[f(Z_t^{U, X,Y})]$. We will discuss the existence, uniqueness, and other properties of $\mu_\nu^{U,X}$ in Section \ref{fast_with_frozen_section}. The final system to define before we can state the theorem is the averaged $0$-mode system:
\begin{equation}\label{averaged_syst}
\begin{cases}
    d\bar{X}_t^\nu = \nu^{\gamma}\partial_y^2 \bar{X}_t^\nu dt - \nu^{\gamma/2-1/6} \int_{\mathcal{H}_{\neq}} b_0(Y) \mu_{\nu}^{U_s, \bar{X}_s^\nu}(dY) dt,\\
    \bar{X}_t(x,y=\pm 1) = 0,\\
    \bar{X}_0^\nu = X_0^\nu \in H_0.
\end{cases}
\end{equation}
In general, we write
\begin{equation}
    \bar{b}_0(U,X) \coloneqq \int_{\mathcal{H}_{\neq}} b_0(Y) \mu_{\nu}^{U, X}(dY),
\end{equation}
whenever $U$ and $X$ are such that $\mu_{\nu}^{U,X}$ exists (and is unique). The nature of the averaged system, including its well-posedness in $C([0,T \wedge \sigma] ; \mathcal{H}_0)$, will be addressed in Section \ref{averaged_section}.  We note that $\bar{X}_t^\nu$ is only defined for $t < \sigma$, but $\{\sigma <T\}$ is an event with exponentially small probability (see \eqref{eqn: tail_est_on_sigma} for a precise bound). We extend the definition of $\bar{X}_t^\nu$ to the whole interval $[0,T]$ by $\bar{X}_t^\nu = \bar{X}_\sigma^\nu$ for $t > \sigma$. We are now prepared to state and discuss the main theorem.

\subsection{The Main Result}\label{subsection: theorems}

We will state two versions of the main theorem. The first, Theorem \ref{main_theorem_2}, will be a qualitative averaging principle, and therefore simpler to state. The second, Theorem \ref{main_theorem}, will be quantitative, but we will state it for a particular choice of $\gamma = 0$. The same techniques are used in proving both theorems, but the presence of a large number of parameters can make quantitative arguments more unwieldy.
\begin{theorem}\label{main_theorem_2}
Suppose that $\alpha > 1/12$, $\beta > 0$, and $\gamma \in [0, 1/3)$ satisfy \eqref{scaling_constraint}, the forcing satisfies \eqref{noise_assumptions}, and the compatibility condition $\mathcal{W}_{in}(y=\pm 1) = 0$ holds. Fix the regularity parameters $m \in (2/3,1)$, $\theta \in (0,1/2]$, and $a \in (0,1)$, and the terminal time $T > 0$. There exists a constant $c_0 = c_0(m) > 0$, such that if $$||\mathcal{W}_{in}||_{H^3} \leq c_0, \; \; X_0^\nu \in \mathcal{H}_0, \;\; |\partial_y|^a X_0^\nu \in\mathcal{H}_0, \;\; Y_0^\nu \in \mathcal{H}_{\neq},$$
then
    \begin{equation}
             \forall p\geq 1,\;\lim_{\nu \to 0}\E(\sup_{t \in [0,T]}|| X_t^\nu - \bar{X}_t^\nu||_{H^{-\theta}}^p) = 0
    \end{equation}
\end{theorem}
\begin{theorem}\label{main_theorem}
    Take $\gamma = 0$ and suppose $\alpha > 1/6$, $\beta > 0$ satisfy \eqref{scaling_constraint}, the forcing satisfies \eqref{noise_assumptions}, and the compatibility condition $\mathcal{W}_{in}(y=\pm 1) = 0$ holds. Fix the regularity parameters $m \in (2/3,1)$, $\theta \in (0,1/2]$, and $a \in (0,1)$, and the terminal time $T > 0$. Set $\alpha' \coloneqq  \frac{1}{36}\min\{ \frac{\theta}{6}\frac{a}{1+2a} , \beta, \alpha\}$. Then there exist constants $c_0 = c_0(m) > 0$, $C_0 = C_0(m) > 0$, $C_1 = C_1(||\Psi||, T, \alpha, m, \theta, a)>0$, and $\nu_0 = \nu_0(||\Phi||,||\Psi||, T, m, \theta, a) >0$ such that for all $\nu \in (0, \nu_0)$, if $X_0^\nu \in \Hh_0$, $Y_0^\nu \in \Hh_{\neq}$, and
    \begin{equation}
        ||\mathcal{W}_{in}||_{H^3} \leq c_0, \;\;||X_0^\nu||_{\mathcal{H}} \leq C_0\nu^{-\alpha'/2},\;\;\nu^{a/2}|||\partial_y|^a X_0^\nu||_{\mathcal{H}} \leq C_0 \nu^{-\alpha'/2}, \;\;||Y_0^\nu||_{\mathcal{H}} \leq C_0 \nu^{-\alpha'/4},
    \end{equation}
    then
\begin{equation}\label{explicit_error_rate}
        \E(\sup_{t \in [0,T]}|| X_t^\nu - \bar{X}_t^\nu||_{H^{-\theta}}) \leq  C_1\nu^{\frac{\theta}{12} \frac{a}{1+2a}}.
    \end{equation}
    \end{theorem}
\begin{remark}
    In the fully quantitative case, it is important to understand the dependence of the various constants on $T$, since $\nu_0$ also depends on $T$. Per the definition of $\nu_0$ via \eqref {eqn: def_of_nu_0}, the only requirement on $T$ is $T \sim\nu_0^{-1/3-}$, or in other words
\begin{equation}\label{condition_on_T}
    \nu^{1/3-} T \lesssim 1,
    \end{equation}
    where the exact relationship depends on $||\Phi||_{H^3}, ||\Psi||$, $m$, and universal constants, as described by the definition of $\nu_0$. Note that \eqref{condition_on_T} comes from the need to control the background shear flow $\mathcal{U}$. If we take $\Phi \equiv 0$, then one can decouple the explicit dependence between $\nu$ and $T$. Meanwhile, the constant $C_1$ depends on $T$ like $1+T^2$. Hence, one could potentially upgrade the statement of Theorem \ref{main_theorem} to a bound over time intervals of at most length $O(\nu^{\frac{\theta}{24} \frac{a}{1+2a} -})$, with the rate of convergence becoming arbitrarily small. Note that $1/3$ is greater than $\frac{\theta}{24} \frac{a}{1+2a}$, so the most binding restriction on taking $T \to \infty$ actually comes from needing $C_1 \nu^{\frac{\theta}{12} \frac{a}{1+2a}} \to 0$. Note that this restriction on $T$ is in the long time scale. In the original time scale, we allow times $O(\nu^{-1-\frac{\theta}{24} \frac{a}{1+2a} -})$.
\end{remark}
\begin{remark}
One may wonder at the need for a precise quantitative statement, as in Theorem \ref{main_theorem}. Precise quantitative results are important both in quantitative stability theory and in the theory of averaging principles. In short, quantitative stability theory is concerned with the precise $\nu$-dependent size of initial data for which stability results hold. Meanwhile, precise convergence results for averaging principles are of interest in the analysis of numerical methods, as well as for understanding Gaussian fluctuation about the averaged system \cites{brehier2012strong, brehier2013analysis, brehier2020orders, rockner2023asymptotic, weinan2003multiscale}.
In generic settings, the optimal strong convergence rate for a fast-slow system with time-scale parameter $\epsilon$ is $\epsilon^{1/2} $\cites{brehier2012strong, brehier2020orders}. For more information, see Section \ref{section: literature_review}. This would indicate that the best rate of convergence one can reasonably expect would be $\nu^{\frac{1}{2}(2/3-\gamma)}$. Contained in the proof of Theorem \ref{main_theorem_2} are the ingredients to obtain a convergence rate of at best $\min\{ \nu^{\frac{1}{8}(\frac{2}{3}- \gamma)},\nu^{\frac{\theta}{2}\left(\frac{1}{3} \frac{a}{1+2a}(1 - 3\gamma) + \gamma \right)- }\}$. The first component of this bound comes from an approximation we make to the true process $X_t^\nu$ by the so-called pseudo linearized process $\tilde{X}_t^\nu$ ( see \eqref{fast_slow_lin} and Proposition \ref{approx_with_lin}). The second component comes from an averaging principle proved for $\tilde{X}_t^\nu$. It is natural to ask if one can obtain the rate of $\nu^{\frac{1}{2}(2/3-\gamma)}$ and if one can obtain convergence in $L^2$ rather than $H^{-\theta}$. Additionally, one can ask if a weak averaging principle holds with rate $\nu^{2/3-\gamma}$. This would be useful in obtaining central limit theorem-type results or in establishing a large deviation principle.
\end{remark}
\begin{remark}
When $\gamma = 0$, as in Theorem \ref{main_theorem}, the noise strength parameter $\alpha$ is at best $1/6+$, meaning $\Psi dW_t$ in \eqref{original_couette} has noise intensity $\nu^{5/6+}$. This is much weaker than the classical scaling for 2D stochastic Navier-Stokes \cite{kuksin2012mathematics}. If we wish to have stronger forcing, then we are forced to consider shorter time scales by \eqref{scaling_constraint}. By taking $\gamma \to 1/3$, we enable $\alpha \to 1/12$, giving a noise intensity of $\nu^{3/4+}$.

When examining the governing equation \eqref{fast_slow_syst} for $X_t^\nu$ or \eqref{averaged_syst} for $\bar{X}_t^\nu$, we see that for $\gamma > 0$, the linear evolution is small. Thus meaningful evolution must come from the nonlinear term. In Section \ref{averaged_section}, we will show that $\nu^{\gamma/2 - 1/6} \int_0^t e^{\nu^{\gamma} (t-s)\partial_y^2}\int_{\mathcal{H}_{\neq}} b_0(Y) \mu^{U_s, \bar{X}_s^\nu}(dY) ds$ is at most $O(1)$. 
Let us briefly give heuristics to suggest that this could be sharp and we should expect this nonlinear term to be dominant in general. 
Under the stopping time hypotheses, the linear problem that $Z_t$ solves is close to the linearized Couette flow subjected to stochastic forcing $\Psi d \bar{W}_t$, i.e. 
\begin{align*}
    d Z_t \approx \nu^{2/3}\left( \Delta Z_t - \nu^{-1 } y \partial_x Z_t \right) dt + \Psi d \bar{W}_t. 
\end{align*}
If there were no boundaries, this SDE can be explicitly integrated, since one can take the Fourier transform in both $x$ and $y$ and integrate the resulting first order equation (denoting $(x,y) \mapsto (k,\eta)$ as the frequency variables), 
\begin{align*}
    d \hat{Z}_t = \nu^{2/3}\left( -(k^2 + \eta^2)Z_t + \nu^{-1 } k \partial_\eta Z_t \right) dt + \widehat{\Psi}(k,\eta) d \bar{W}_t. 
\end{align*}
As we can see, frequencies in $y$ are transferred from low to a high as $\approx \nu^{-1/3} kt$ and progressively damped by the viscosity $-\nu^{2/3}\eta^2$. 
From this we see that any information added at $O(1)$ times by the forcing will be strong damped at time $t \gtrsim k^{-2/3}$, and it will have reached frequencies $\eta \approx \nu^{-1/3} k^{1/3}$. 
For the purpose of this heuristic it suffices to consider only $k = O(1)$. 
Now, if we assume that the true dynamics are roughly approximated by the dynamics on $\mathbb T \times \mathbb R$, then we can determine the generic long-term behavior of $Z_t$ and hence describe the invariant measure $\mu_\nu$. 
Specifically, if we consider a random function $Y$  sampled from $\mu_\nu$, it will satisfy 
\begin{align*}
\mathbb E \|Y(k,\eta)\|^2 \approx \nu^{1/3}
\end{align*}  
for $\|\eta\| \lesssim \nu^{-1/3}$, yielding $\mathbb E ||Y||_2^2 \approx 1$ and $\mathbb E ||(-\Delta)^{-1} Y||_{L^\infty}^2 \lesssim \nu^{1/3}$ if one uses Gagliardo-Nirenberg (i.e. $||(-\Delta)^{-1} Y||_{L^\infty}^2 \lesssim ||(-\Delta)^{-1} Y||_{L^2} ||\partial_y (-\Delta)^{-1} Y||_{L^2}$.
Integrating the nonlinearity then would give, heuristically (using the parabolic regularization to eliminate the $\partial_y$ and neglecting the $x$-regularity, as higher $x$ frequencies would generally be better), 
\begin{align*}
\nu^{\gamma/2 - 1/6} || \int_0^t e^{\nu^{\gamma} (t-s)\partial_y^2}\int_{\mathcal{H}_{\neq}} b_0(Y) \mu^{U_s, \bar{X}_s^\nu}(dY) ds||_{L^2} & \lesssim t^{1/2} \nu^{-1/6}  \int_{\mathcal{H}_{\neq}}||(-\Delta^{-1} Y)||_{L^\infty} ||\partial_x Y||_{L^2} \mu(dY)  \\ 
& \approx t^{1/2},
\end{align*}
which would suggest we can obtain for $O(1)$ times, which can then be extended to long times via a Gr\"onwall argument. Our estimates match these heuristics. There may be additional cancellations if one can justify that the large scales and small scales are approximately statistically independent and hence obtain some additional smallness in the quadratic nonlinear interaction. 
\end{remark}

We now state a version of Theorem \ref{main_theorem} in the original scaling. A similar version of Theorem \ref{main_theorem_2} is also possible to state, but due to the scaling of the initial data, we view Theorem \ref{main_theorem} to be more relevant. Of course, one could write a version of the corollary in the $\gamma > 0$ setting, but the proliferation of parameters becomes unwieldy so for clarity we decided to focus only on the case of $\gamma=0$ below. In the statement of the corollary, we let $\bar{\omega}_0 = \nu^{1/2+ \beta} \bar{X}_{\nu^{1-\gamma} t}$.
    \begin{corollary}\label{main_corollary}
        Take $\gamma = 0$ and suppose $\alpha > 1/6$, $\beta > 0$ satisfy \eqref{scaling_constraint}, the forcing satisfies \eqref{noise_assumptions}, and the compatibility condition $\mathcal{W}_{in}(y=\pm 1) = 0$ holds. Fix the regularity parameters $m \in (2/3,1)$, $\theta \in (0,1]$, $a \in (0,1)$, and the terminal time $T >0$. Set $\alpha' \coloneqq  \frac{1}{36}\min\{ \frac{\theta}{6} \frac{a}{1+2a}, \beta, \alpha\}$. Then there exists constants $c_0 = c_0(m) > 0$, $C_0 = C_0(m) > 0$, $C_1 = C_1(||\Psi||, T, \alpha, m, \theta, a)>0$,  and $\nu_0 = \nu_0(||\Phi||, ||\Psi||, T, \alpha, m, \theta, a) >0$ such that for $\nu < \nu_0$, if
    \begin{equation}
        ||\mathcal{W}_{in}||_{H^3} \leq c_0, \;||\omega_{in,0}||_{\mathcal{H}_0}\leq C_0 \nu^{1/2 + \beta - \alpha'/2} ,\;\nu^{a/2}|| |\partial_y|^a \omega_{in, 0}||_{\mathcal{H}_0} \leq C_0 \nu^{1/2+\beta-\alpha'/2}, \; ||\omega_{in,\neq}||_{\mathcal{H}_{\neq}}\leq C_0 \nu^{1/2+\alpha - \alpha'/4},
    \end{equation}
   then
    \begin{equation}
       \mathbb{E}\left(\sup_{t \in [0,T]}||\omega_0 -\bar{\omega}_0||_{H^{-\theta}} \right) \leq C_1\nu^{\frac{1}{2} + \beta+\frac{\theta}{12} \frac{a}{1+2a}}.
    \end{equation}
    \end{corollary}
    In terms of the vorticity $w$ in \eqref{base_navier_stokes}, we see that for the $x$-independent mode $w_0$:
    $$w_0(t,y) \approx 1 + \mathcal{W}(t,y) + \bar{\omega}_0(t,y)$$
    with high probability until times at least like $\nu^{-1}$. Notably, the governing equations are closed and do not involve  $w_{\neq}$. Thus the shear component of $v$ is approximately given by
    \begin{equation}\label{description_of_velocity}
        v_0(t,y) \approx (\mathcal{U}(t,y) + \partial_y \int_{-1}^1 G(y,y') \bar{\omega}_0(t,y') dy',0).
    \end{equation}
    This expresses the shear flow as Couette flow, plus a regular $O(1)$ component evolving via the stochastic heat equation, plus a possibly much rougher component which is $O(\nu^{1/2+\beta})$ and evolving according to the (nonlinear and non-local) averaged slow system \eqref{averaged_syst}.

\subsection{Literature Review}\label{section: literature_review}

\subsubsection{Quantitative Stability}\label{subsection: quant_stability}
The study of the stability properties of shear flows dates back to at least Kelvin \cite{kelvin1887}, who first considered the enhanced dissipation of Navier-Stokes linearized around Couette flow on $\T \times \R$. Later, Orr observed in the Euler equations linearized around Couette flow the important property of \textit{inviscid damping}, where the non-zero $x$-Fourier frequencies of the perturbation velocity $u$ decay at an algebraic-in-$t$, $\nu$-independent rate \cite{orr1907stability}. It has become common to investigate the stability of shear flows as equilibrium solutions to Navier-Stokes and to prove enhanced dissipation and inviscid damping in the nonlinear setting for small initial perturbations. Quantitative stability theory at high Reynolds number asks the following: given two norms $||\cdot||_{X}$, $||\cdot||_{Y}$, an equilibrium solution (or slowly varying solution) $u_E$ to the Navier-Stokes equations, and an initial condition $u_{in}$, determine a minimal $\eta$ such that
\begin{equation}
    ||u_{in} - u_E||_{X} \ll \nu^{\eta} \implies \begin{cases} ||u(t) - u_E||_{Y} \ll \nu^\eta.\\
    \lim_{t \to 0}||u(t) - u_E||_{Y} =0,
        \end{cases}
\end{equation}
where $u(t)$ is the corresponding solution with initial condition $u_{in}$. The norms $||\cdot||_{X}$ and $||\cdot||_{Y}$ are carefully chosen to capture properties of the flow, such as enhanced dissipation and inviscid damping. See \cite{stabilityoverview} for an overview of the problem in two and three dimensions. We remark that the threshold value $\eta$ can depend greatly on the exact choices of $X$ and $Y$, boundary conditions, domain, and dimension. In particular, we highlight
\cites{bedrossian2023stability, bedrossian2024uniform, chen2020transition}, which consider the problem on the periodic channel $\T \times [-1,1]$. In the regularity class and with the boundary conditions in this paper, $\eta$ was estimated as $\eta \leq 1/2$ in \cite{bedrossian2023stability}. Hence we re-scale $\omega_{\neq}$ and $\omega_0$ into $X_t^\nu$ and $Y_t^\nu$ by multiplying by at least $\nu^{1/2}$. We wish to stay strictly below the stability threshold, hence the factors of $\alpha$ and $\beta$ in $Y_t^\nu = \nu^{1/2+\alpha}\omega_{\neq}(t/\nu^{1-\gamma})$ and $X_t^\nu = \nu^{1/2+\beta}\omega_{0}(t/\nu^{1-\gamma})$. The condition \eqref{scaling_constraint} is determined by the necessity to control the size of the non-linearity. Note that since $\alpha > 1/12$ strictly, one could define $\tilde{\alpha} = \alpha - 1/12$, so that the rescaling factor is $\nu^{7/12+ \tilde{\alpha}}$, the strength of the $\Psi dW_t$ noise term is $\nu^{3/4+\tilde{\alpha}}$, and the scaing constraint is
\begin{equation}
    2\tilde{\alpha} - \beta + \gamma/2 = 1/6.
\end{equation}
However, we wish to emphasize the role of $\alpha$ as an amplitude scaling which is in addition to the $\nu^{1/2}$ factor. That is, we wish to emphasize that $\alpha$ is a defect arising in the stochastic setting, which could potentially be reduced in future work.

In $\mathbb T \times \mathbb R$, the work \cite{bedrossian2016enhanced}, characterized the dynamics near Couette flow as $\nu \to 0$ for initial data which was very smooth but not small relative to $\nu$, with particular attention paid to different emergent time-scales. Qualitatively, it was found that for $t \ll \nu^{-1/3}$, the inviscid dynamics dominate, and hence the solution undergoes inviscid damping and the vorticity mixes essentially as a passive scalar. Then for $\nu^{-1/3} \ll t \ll \nu^{-1}$, the enhanced dissipation of $\omega_{\neq}$ dominates, while $\omega_0$ remains relatively stable. The velocity $v$ moves to a (potentially non-Couette) shear flow due to the inviscid damping. Then for $\nu^{-1} \ll t$, the final relaxation occurs and the system slowly converges back to Couette flow. The work \cite{bedrossian2024uniform} provides the analogous results \eqref{original_couette} with $\Phi \equiv 0$ and $\Psi \equiv 0$, with a similar qualitative interpretation as \cite{bedrossian2016enhanced}. These works focused on Gevrey-2 regularity where one can hope to obtain uniform-in-$\nu$ results, however, more relevant for the work here is the work \cite{bedrossian2023stability} which proves similar results with initial data $O(\nu^{1/2})$ in $\mathcal{H}$; see \cites{bedrossian2018sobolev,beekie2024transition, masmoudi2020enhanced, masmoudi2022stability, wei2023nonlinear} for works in Sobolev regularity classes on $\mathbb{T} \times \R$. 

Corollary \ref{main_corollary} provides a kind of analogue of the stability threshold problem for the stochastically forced setting. 
The qualitative description of the dynamics is as follows: for $t \ll \nu^{-1/3}$, $v_0$ does not significantly evolve, while $\omega_{\neq}$ begins to mix and undergo inviscid damping. By $t \sim \nu^{-1/3}$, $\omega_{\neq}$ has been driven by the inviscid damping and enhanced dissipation to a statistical equilibrium (in an approximate and rescaled sense) with high probability. For the longer time-scale $\nu^{-1/3} \ll t\ll \nu^{\gamma-1}$, $\omega_0$ evolves according to the (nonlinear) averaged equation \eqref{averaged_syst} while $\omega_{\neq}$ stays close to a slowly evolving statistical equilibrium that is marked by a continual transfer of enstrophy via mixing from the forcing to the small scales where it is dissipated. Analogous to stability threshold problems in Sobolev regularity, we have only proved the averaging principle  holds for small initial data and small stochastic forcing. 

\subsubsection{Fast-Slow Systems}
The marked difference in time-scale between the zero and non-zero frequencies leads us to the theory of multi-scale systems, where two components $A_t$ and $B_{t}$ evolve on different time scales, quantified by a parameter $\epsilon$. Consider an ODE system
\begin{equation}
    \begin{cases}
        \dot{A}_t = f(A_t, B_t), &A_0 = a\\
        \dot{B}_t = \epsilon^{-1} g(A_t, B_t), &B_0 = b.
    \end{cases}
\end{equation}
In certain situations (for example, if the orbits of $B_t$ are periodic), $A_t$ converges on $[0,T]$ to $\bar{A}_t$ satisfying an ``averaged" equation
$$ \dot{\bar{A}}_t  = \bar{f}(\bar{A}_t),$$
with $\bar{f}(A) = \lim_{T \to 0} \frac{1}{T}\int_0^T f(A, B_s) ds.$ This no longer depends on the fast process time-scale, which has significant potential computational benefits \cites{brehier2013analysis, brehier2020orders, weinan2003multiscale}. The mathematical justification for the averaging principle in the ODE setting was first considered rigorously by Bogolyubov \cite{bogoliubov1961asymptotic}. Khasminkskii developed the theory of averaging principles for SDEs \cite{khasminskii1968}. In this case, the average $\bar{f}$ is understood as averaging over the invariant measure of the fast process with frozen slow component.

To the authors' knowledge, the first averaging principles in the context of nonlinear SPDES were obtained by Cerrai and Friedlin for stochastic reaction-diffusion equations \cites{Cerrai2009Averaging2, cerrai2009averaging}. Since then, a variety of systems have been considered, including nonlinear Schr\"odinger \cite{gao2018averaging}, wave equations \cite{fu2018weak}, 1D Burger's equations \cite{dong2018averaging}, and McKean-Vlasov systems \cites{huang2024averaging, rockner2021strong, zhang2025averaging}. 

A strong averaging principle corresponds to the statement $\lim_{\epsilon \to 0}\E \sup_{t \in [0,T]} || A_t - \bar{A}_t|| \to 0.$ A weak averaging principle corresponds to convergence of $A_t \to \bar{A}_t$ in law as $\epsilon \to 0$.

We emphasize that the averaging principle discussed in this paper involves a fast process with random, time-dependent coefficients. The time dependence necessitates that the invariant measure $\mu_{\nu}^{U,X}$ of the frozen equation be dependent on $U$, which leads to the stopping time $\sigma$. Averaging principles with time-dependent coefficients are more rare in the literature, see for instance \cite{liu2020averaging}.

In the context of Navier-Stokes systems, rigorous averaging principles are a recent development. In \cite{li2018averaging}, a strong averaging principle was proved for the system
\begin{equation}\label{navier_process}
    \begin{cases}
        dX_t^\epsilon = ( \nu \Delta X_t^\epsilon - (X_t^\epsilon \cdot \nabla) X_t^\epsilon + f(X_t^\epsilon, Y_t^\epsilon) - \nabla p) dt + \sigma_1(X_t^\epsilon) Q_1 dW_t^1,\\
        dY_t^\epsilon = \frac{1}{\epsilon}(\Delta Y_t^\epsilon + g(X_t^\epsilon, Y_t^\epsilon)) ds + \frac{1}{\epsilon^{1/2}} \sigma(X_t^\epsilon, Y_t^\epsilon) Q_2 dW_t^2,\\
        \nabla \cdot X_t^\epsilon  = 0, \; \nabla \cdot Y_t^\epsilon = 0,
        \end{cases}
\end{equation}
posed on $\T^2$. Here, the slow process $X_t^\epsilon$ is the Navier-Stokes process, and $X_t^\epsilon$ and $Y_t^\epsilon$ are coupled through the Lipschitz nonlinearities $f$ and $g$.  The authors proved strong averaging using a Khasminskii-style discretization approach, with appropriate stopping times employed to handle the non-linear contributions. See also \cite{liu2024averaging} for a similar result for generalized Navier-Stokes in 3D, and \cite{gao2021averaging} for a 2D result driven by L\'evy noise.  Contrasting with \eqref{navier_process}, the fast and slow processes of \eqref{fast_slow_syst} both come from the same Navier-Stokes process, and are coupled with non-Lipschitz nonlinearities. Additionally, in our case, the scale parameter is a power of the viscosity, so the dissipative term is getting progressively weaker. These issues add additional complexity to our problem. By way of similarity, we will also employ a Khasminskii discretization scheme to obtain strong convergence. We will be using stopping times to control nonlinear behavior, but ours will be more directly tied to various martingales. We also mention \cite{liu2023strong}, which had applications to Navier-Stokes systems, and from which we take some ideas in the final steps of the proof.

A more similar result to ours is \cite{debussche2024second}, where a weak averaging principle was obtained for several classes of fluid equations in dimension $2$ and $3$. The general form of the systems considered in \cite{debussche2024second} is
\begin{equation}
    \begin{cases}
        du_t^\epsilon = A u_t^\epsilon dt + b(u_t^\epsilon, u_t^\epsilon) dt + b(v_t^\epsilon, u_t^\epsilon) dt,\\
        dv_t^\epsilon = \epsilon^{-1} C v_t^\epsilon dt + A v_t^\epsilon dt + b(u_t^\epsilon, v_t^\epsilon)dt+ b(v_t^\epsilon, v_t^\epsilon)dt + \epsilon^{-1/2} Q^{1/2} dW_t,
    \end{cases}
\end{equation}
where $A$ and $C$ are negative definite linear operators and $b$ is a nonlinearity, such as $b(u,v) = (u\cdot \nabla)v$. Subject to suitable assumptions, it was shown that every weak accumulation point of $u^\epsilon$ solves
\begin{equation}\label{transport_noise}
    du_t = A u_t + b(u_t, u_t) + b((-C)^{-1} Q^{1/2} \circ dW_t, u_t) + b(r,u_t) dt,
\end{equation}
where $r$ is the It\^o-Stokes drift velocity
$$r = \int(-C)^{-1} b(w,w) d\mu(w)$$
and $\mu$ is the invariant measure of the linearized equation $dv_t = C v_t dt + Q^{1/2} dW_t.$ Note the presence of the transport noise in \eqref{transport_noise}, which is not explicitly present in our averaged equation \eqref{averaged_syst}. Formally, this result is more closely related to the present paper, if one considers $A = \Delta$ and $C = - U_t \partial_x + U_t'' \partial_x \Delta^{-1}$. However, this analogy breaks apart even in the simple case $U = y$, as $-y \partial_x$ is not a negative definite operator (in fact it is skew-adjoint on $H^1_0(\T \times [-1,1])$). Furthermore, $U_t$ is both time-dependent and random. Also, the various nonlinearities $b_0$, $b_m$, and $b_{\neq}$ are composed of the projections of the full nonlinearity and hence do not individually enjoy all of the same properties. Another difference is that the resulting invariant measure in our context $\mu_\nu^{U,X}$  is $\nu$-dependent, and hence so is our averaged equation. This is not typical in the fast-slow literature. In fact, it is not clear if $\mu_\nu^{U,X}$ converges in any useful sense as $\nu \to 0$. 
One of the key techniques from \cite{debussche2024second} which we do adopt is the usage of a (pseudo) linearization of the fast process in order to construct the invariant measure which is averaged over. 

\subsubsection{Relationship with Stochastic Structural Stability Theory (S3T)}
Stochastic Structural Stability Theory (S3T) was introduced in \cite{farrell2003structural} to study the evolution of `zonal flows' (i.e. the background shear flow component $X_t$) in the presence of smaller scale, non-zonal (here $Y_t$) fluctuations, mainly in the context of atmosphere and ocean sciences where such kinds of zonally dominated, approximately two dimensional flows commonly arise in reduced models. 
The theory has been applied and expanded in various ways, see e.g. \cites{constantinou2014emergence,bakas2011structural,farrell2009stochastic,farrell2016structure,farrell2007structure}. 
Here, small-scale turbulence is modeled on the non-zonal components by subjecting $Y_t$ to stochastic forcing. 
In this theory, the nonlinear interactions between $Y_t$ and itself are dropped, while the nonlinear interactions between the $X_t$ and $Y_t$ components are retained. 
It is often further assumed that the small scales statistically equilibrate on a much faster time-scale than the large-scale zonal flows, which essentially reduces the theory to a fast-slow averaging principle. 
In this way, Theorem \ref{main_theorem_2} can be considered, at least to some degree, a mathematically rigorous justification of these approximations in certain perturbative regimes.

\subsubsection{Physical Interpretation}

We begin by commenting on the Navier boundary conditions of \eqref{base_navier_stokes}. The most common boundary condition in fluid dynamics is the no-slip condition: $v = 0$ on the boundary of the domain. Navier, or Navier-slip, boundary conditions were first considered by Navier in \cite{navier1827lois}. The work \cite{masmoudi2003boltzmann} gives a derivation of Navier-type boundary conditions from kinetic theory. In general for a domain $D$, the Navier slip boundary condition is given by
$$2 S(v) \vec{n} \cdot \vec{t} + a v \cdot\vec{t} = 0\vert_{[0,T] \times \partial D}$$
 where $S(v) = \frac{1}{2}(\nabla v +(\nabla v)^T)$ is the rate-of-strain tensor, $\vec{n}$ is the unit outward normal of $\partial D$, $\vec{t}$ is the unit tangent vector to $\partial D$, and $a$ is a parameter. In our case, after normalizing $a = 1$ and enforcing impermeability $v \cdot \vec{n} \vert_{\partial D} = 0$, the boundary conditions reduce to
$$\partial_y v^1(t,x,y = \pm 1) = 1, \;\; v^2(t,x,y = \pm 1) = 0.$$
Physically, this corresponds to a fluid slipping along the boundaries which are being moved in opposite directions by a fixed force. It is interesting to ask whether or not our results also hold for no-slip Dirichlet boundary conditions, however, one is unlikely to be able to answer that question if one cannot answer the question with Navier-type boundary conditions. 

In the classical theory of small noise stochastic Navier-Stokes on $\T^2$, the scaling $\nu^{1/2}$ on the noise is needed to ensure that the resulting invariant measures $\{\mu_\nu\}$ have a well-defined and non-zero inviscid limit \cite{kuksin2012mathematics}. This scaling can potentially be interpreted as the largest noise that will avoid true 2d turbulence as $\nu \to 0$, and any weaker noise will be essentially laminar. We have much weaker noise, with $\nu^{5/6}$ on the $x$-independent component and the much weaker $\nu^{2/3+\alpha}$ on the $x$-dependent component. Thus, we consider the equilibrium reached by $\omega_{\neq}$ to be small, and vanishing as $\nu \to 0$. Indeed, our main result concerns the long-time evolution of $\omega_0$ (and therefore $w_0$ and $v_0$).

\subsection{Outline}\label{outline}
We introduce the following norms from \cite{bedrossian2023stability}. For $f \in \mathcal{H}_{\neq}$, define
\begin{equation}
    \begin{split}
        \Ee_{\neq}(f) \coloneqq||f||_{\mathcal{H}_{\neq}}^2 &\coloneqq \sum_{k \neq 0}|k|^{2m}\mathrm{Re}\langle (I + c_{\mathfrak{t}} \mathfrak{J}_k) f_k, f_k \rangle_{L^2( [-1,1])}\\
        &\quad+ c_{\mathfrak{a}} \sum_{k \neq 0}|k|^{2m} \nu^{2/3} |k|^{-2/3} \mathrm{Re}\langle (I + c_\mathfrak{t}\mathfrak{J}_k) \partial_y f_k, \partial_y f_k \rangle_{L^2( [-1,1])} \\
        & \quad- c_{\mathfrak{b}} \sum_{k\neq 0}|k|^{2m} |k|^{-4/3} \nu^{1/3}\left(\mathrm{Re}\langle ik f_k, \partial_y f_k\rangle_{L^2([-1,1])} + \mathrm{Re}\langle \partial_y f_k, ik f_k\rangle_{L^2([-1,1])}\right),
    \end{split}
\end{equation}
where $f_k$ is the Fourier transform of $f$ in the $(x,k)$ pair evaluated at $k$. The constants $c_{\mathfrak{t}}, c_{\mathfrak{a}},$ and $c_{\mathfrak{b}}$ are small constants specified in \cite{bedrossian2023stability}. The operator $\mathfrak{J}_k$ is a self-adjoint singular integral operator defined in \cite{bedrossian2023stability}, which is bounded on $H^1_0$. Meanwhile for $f \in \mathcal{H}_{0}$, we define
\begin{equation}
    \Ee_0(f) \coloneqq||f||_{\mathcal{H}_{0}^2} \coloneqq ||f||_{L^2([-1,1])}^2 + c_{\mathfrak{a}}||\nu^{1/3} \partial_y f||_{L^2([-1,1])}^2.
\end{equation}
We will use $\langle \cdot, \cdot \rangle_{\mathcal{H}_{\neq}}$, and $\langle \cdot, \cdot \rangle_{\mathcal{H}_{0}}$ to denote the respective inner products, as defined by polarization. Under a suitable choice of constants $c_{\mathfrak{t}}, c_{\mathfrak{a}},$ $c_{\mathfrak{b}}$, we have $||f||_{\mathcal{H}_{\neq}} \approx ||f||_\mathcal{H}$ for $f \in \mathcal{H}_{\neq}$ and $||f||_{\mathcal{H}_{0}} \approx ||f||_\mathcal{H}$ for $f \in \mathcal{H}_{0}$.
 
 In Section \ref{energy_section}, we will discuss bounds on $\mathcal{E}$, as well as the associated dissipation $\mathcal{D}$. We will recall properties of the deterministic equations from \cite{bedrossian2023stability} related to controlling $\Ee$ and define relevant stopping times which will enable us to extend these concepts to the stochastic setting. In particular, we will define the key stopping time $\tau$ in \eqref{definition_of_tau}. These will lead to some \textit{a priori} estimates. We will also describe the behavior of the random, time dependent coefficients $U_t$ and $U_t''$.

We will need several additional processes in order to prove the main theorem. We begin with the \textit{pseudo-linearized process}:
\begin{equation}\label{fast_slow_lin}
    \begin{cases}
        d \tilde{X}^\nu_t = (\nu^\gamma \partial_y^2 \tilde{X}^\nu_t - \nu^{\gamma/2 - 1/6}b_0(\tilde{Y}_t^\nu))dt,\\
        d \tilde{Y}_t^\nu = \nu^\gamma(\Delta \tilde{Y}_t^\nu - \nu^{-1} U_t \partial_x \tilde{Y}_t^\nu + \nu^{-1} U_t'' \partial_x \Delta^{-1}\tilde{Y}_t^\nu) dt - \nu^{\beta + \gamma -1/2} b_m(\tilde{X}_t^\nu, \tilde{Y}_t^\nu)dt + \nu^{-1/3 + \gamma/2} \Psi dW_t,\\
        \tilde{Y}_t(x,y = \pm 1) = \tilde{X}_t(y = \pm 1) = 0,\\
        (x,y) \in \T \times [-1,1].
    \end{cases}
\end{equation}
We call this the pseudo-linearized process since the equation for $\tilde{Y}_t$ is linear in $\tilde{Y}_t$ if the slow process is fixed.

Furthermore, we define the following \textit{auxiliary processes}. These processes have been commonly employed in proving averaging principles via the Khasminskii discretization approach \cites{brehier2012strong, huang2024averaging, khasminskii1968, li2018averaging, liu2024averaging, rockner2021strong}. Let $\delta = \delta(\nu)>0$ be a step size to be determined later. For now, we simply assume that $\nu^{2/3-\gamma} < \delta < 1$. Define $K \coloneqq \lfloor T/\delta \rfloor $. Then for $k \in \{0,\hdots, K\}$ and for $t \in [k\delta, (k+1)\delta \wedge T]:$
\begin{equation}
\begin{split}
        \hat{Y}_t^\nu &= \hat{Y}_{k\delta} + \int_{k\delta}^t \nu^\gamma(\Delta \hat{Y}_s^\nu - \nu^{-1} U_{k\delta} \partial_x \hat{Y}_s^\nu +  \nu^{-1} U_{k\delta}'' \partial_x \Delta^{-1}\hat{Y}_s^\nu)ds\\
        &\quad- \nu^{\beta + \gamma -1/2} b_m(\tilde{X}_{k\delta}^\nu, \hat{Y}_t^\nu)ds+ \nu^{-1/3 + \gamma/2} \Psi dW_t.
\end{split}
\end{equation}
Alternatively, let $t(\delta) \coloneqq \lfloor t/\delta \rfloor \delta$. Then
\begin{equation}\label{fast_aux_process}
\begin{split}
        d\hat{Y}_t^\nu &= \nu^\gamma(\Delta \hat{Y}_t^\nu - \nu^{-1} U_{t(\delta)} \partial_x \hat{Y}_t^\nu +  \nu^{-1} U_{t(\delta)}'' \partial_x \Delta^{-1} \hat{Y}_t^\nu)dt\\
        &\quad- \nu^{\beta + \gamma -1/2} b_m(\tilde{X}_{t(\delta)}^\nu, \hat{Y}_t^\nu)dt + \nu^{-1/3 + \gamma/2} \Psi dW_t.
\end{split}
\end{equation}
We also define the auxiliary process for $\tilde{X}$:
\begin{equation}\label{slow_aux_process}
    d\hat{X}_t^\nu = \nu^\gamma \partial_y^2 \hat{X}^\nu_t dt - \nu^{\gamma/2 - 1/6}b_0(\hat{Y}_t^\nu) dt.
\end{equation}
Note that the equation for $\hat{Y}_t$ depends on $\tilde{X}_t$ only through the fixed values $\tilde{X}_{k\delta}$ and $U_{k\delta}$. Meanwhile, $\hat{X}_t$ depends on the full trajectory of $\hat{Y}_t$. The systems \eqref{fast_slow_lin}, \eqref{fast_aux_process}, and \eqref{slow_aux_process} are locally well-posed up to stopping time and $\mathcal{H}$-valued. Their well-posedness and the validity of an infinite dimensional It\^o formula will be addressed in Appendix \ref{appendix: a}.  More precise \textit{a priori} estimates will be considered in Section \ref{energy_section}. We remark that in particular, one can write mild formulations for $\tilde{X}_t^\nu$ and $\hat{X}_t^\nu$ as
\begin{equation}\label{mild_form_of_pseudo_lin}
    \tilde{X}_t^\nu = e^{\nu^\gamma t \partial_y^2} X_0^\nu + \nu^{\gamma/2 - 1/6}\int_0^t e^{\nu^\gamma (t-s) \partial_y^2} b_0(\tilde{Y}_s)ds
\end{equation}
and
\begin{equation}
    \hat{X}_t^\nu = e^{\nu^\gamma t \partial_y^2} X_0^\nu + \nu^{\gamma/2 - 1/6}\int_0^t e^{\nu^\gamma (t-s) \partial_y^2} b_0(\hat{Y}_s)ds.
\end{equation}
Concerning the Khasminskii discretization, we observe that for $t \in [k\delta, (k+1)\delta]$, the auxiliary process has the same distribution as a time-rescaled version of the fast process with frozen slow component with initial data $\hat{Y}_{k\delta}$, i.e $\hat{Y}_t^\nu \sim Z_{t/\nu^{2/3}}^{U_{k\delta}, X_{k\delta}^\nu, \hat{Y}_{k\delta}^\nu}$. Thus, the general idea is to break $[0,T]$ into intervals of size $\delta$. Then on each interval, we approximate by the psuedo-linearized process, which is in turn approximated by the auxiliary process, which we have observed is in distribution equal to the fast-process with frozen slow component. The size $\delta$ must be sufficiently large so that a near statistical equilibrium can be reached, while still vanishing as $\nu \to 0$. For more information, see the classic work by Khasminskii on SDEs \cite{khasminskii1968} or the modern work by Brehier on SPDEs \cite{brehier2012strong}.

The last technical objects we need to introduce are the functions $B_0$ and $\bar{B}_0$. These are simply the operators $b_0$ and $\bar{b}_0$ with the $-\partial_y$ removed, respectively. For $Y \in \Hh_{\neq}$, we set
\begin{equation}
    B_0(Y) \coloneqq \int_\T (\partial_x \Delta^{-1} Y)(x) Y(x) dx,
\end{equation}
and for $X \in \Hh_0$ suitably small and $||U - y||_{H^4}$ suitably small, we set
\begin{equation}
    \bar{B}_0(U,X) \coloneqq \int_{\Hh_{\neq}} B_0(Y) \mu_{\nu}^{U,X}(dY).
\end{equation}

Section \ref{fast_with_frozen_section} will address the fast process with frozen slow component \eqref{frozen_system}, as well as the invariant measure associated with \eqref{frozen_system}. Meanwhile, Section \ref{averaged_section} will discuss the behavior of the averaged system $\bar{X}_t$. We will show the well-posedness of \eqref{averaged_syst}, which will include showing that the nonlinearity is at most $O(1)$ as $\nu \to 0$. We have already described heuristics suggesting that the nonlinearity does not vanish in the inviscid limit, even when $\gamma > 0$.

The main technical arguments for the proof of Theorem \ref{main_theorem} will be contained in Section \ref{section_of_propositions}. Throughout Section \ref{section_of_propositions}, and in the following propositions, we will assume that $\gamma = 0$, $\alpha'$ is as in Theorem \ref{main_theorem}, and the stopping time $\tau$ is as defined in \eqref{definition_of_tau}. First, we will show that, up to stopping time, $(X_t,Y_t)$ can be approximated by the pseudo-linearized processes $(\tilde{X}_t^\nu, \tilde{Y}_t^\nu)$. In particular,  we will show the following proposition.
\begin{proposition}\label{approx_with_lin}
 There exist constants $C_0 = C_0(m) > 0$, $C = C(||\Psi||, T) > 0$, and $\nu_* = \nu_*(||\Psi||) > 0$ such that if $\nu < \nu_*$, $\Ee_{\neq}(Y_0^\nu) < C_0\nu^{-\alpha'/2}$, and $\Ee_0(X_0^\nu) < C_0\nu^{-\alpha'}$, then
    \begin{equation}
        \sup_{t \in [0, T \wedge \tau \wedge \sigma]}\Ee_0(X_t^\nu - \tilde{X}_t^\nu) \leq C \nu^{\alpha - 2\alpha'}.
    \end{equation}
\end{proposition}
Next, we will provide continuity-in-time estimates on $\tilde{X}_t^\nu$ and $\hat{X}_t^\nu$. We will use these continuity estimates to prove the following proposition bounding the difference between $\tilde{X}_t^\nu$ and $\hat{X}_t^\nu$.
\begin{proposition}\label{lin_to_aux_prop}
There exist constants $C_0 = C_0(m)> 0$, $C = C(||\Psi||, T,a) > 0$, and $\nu_* = \nu_*(||\Psi||) > 0$ such that if $\nu < \nu_*$ and
$$\Ee_{\neq}(Y_0^\nu) < C_0\nu^{-\alpha'/2}, \; \;\Ee_0(X_0^\nu) < C_0\nu^{-\alpha'}, \; \; \nu^{a/3}\Ee_0(|\partial_y|^a X_0^\nu) < C_0\nu^{-\alpha'},$$
then so long as $\nu^{2/3-\alpha'} \leq \delta < \nu^{1/3}$,
    \begin{equation}
        \E\left(\sup_{t \in [0, T \wedge \tau \wedge \sigma]} \Ee_0(\tilde{X}_t^\nu - \hat{X}_t^\nu) \right)\leq  C\delta^{a} \nu^{-a/3}.
    \end{equation}
\end{proposition}
The last key proposition approximates $\hat{X}_t^\nu$ with $\bar{X}_t^\nu$. This is the most in depth of the propositions, as it is where we need to show a sort of convergence between $b_0$ and $\bar{b}_0$. To that end, we introduce the norms $||\cdot||_{\mathcal{H}_0^{1/2}}$ and $||\cdot||_{\mathcal{H}_0^{-1/2}}$ as
\begin{equation}\label{def_of_fractional_energy}
    ||f||_{\mathcal{H}_0^{1/2}} \coloneqq |||\partial_y|^{1/2} f||_{\mathcal{H}_0}, \;\;||f||_{\mathcal{H}_0^{-1/2}} \coloneqq |||\partial_y|^{-1/2} f||_{\mathcal{H}_0},
\end{equation}
where we take the spectral definition of the fractional derivative. Note that based on the spectral definition, the fractional derivative preserves the boundary conditions.
\begin{proposition}\label{key_prop}
There exist constants $C_0 = C_0(m) > 0$, $C = C(||\Psi||, T,a) > 0$, and $\nu_* = \nu_*(||\Psi||) > 0$ such that if $\nu < \nu_*$ and
$$\Ee_{\neq}(Y_0^\nu) < C_0\nu^{-\alpha'/2}, \; \; \Ee_0(X_0^\nu) < C_0\nu^{-\alpha'}, \; \; \nu^{a/3}\Ee_0(|\partial_y|^a X_0^\nu) < C_0\nu^{-\alpha'},$$
then so long as $\nu^{2/3-\alpha'} \leq \delta < \nu^{1/3}$,
    \begin{equation}
        \E\left( \sup_{t \in [0, T \wedge \tau \wedge \sigma]} ||\hat{X}_t - \bar{X}_t||_{\mathcal{H}_0^{-1/2}}^2\right)\leq C\left( \delta^{a/2} \nu^{-a/6-3\alpha'}  + \delta^{-1/4} \nu^{1/6 - 3\alpha'}\right).
    \end{equation}
\end{proposition}
We will conclude Section \ref{section_of_propositions} by combining the relevant propositions into a complete proof of Theorem \ref{main_theorem}.

\section{Control of the Energy Structure and \textit{a priori} Estimates}\label{energy_section}
In this section, we introduce the energy and review relevant theorems from \cite{bedrossian2023stability}. We also include important modifications and clarifications for our purposes. We introduce a collection of stopping times which will be used throughout this paper as a partial replacement for the bootstrap estimates in the deterministic setting. We show how with these stopping times, we can obtain important bounds on $\Ee(X_t^\nu,Y_t^\nu)$, $\Ee(\tilde{X}_t^\nu, \tilde{Y}_t^\nu)$, $\Ee_{\neq}(\hat{Y}_t^\nu)$, and $\Ee_0(\hat{X}_t^\nu)$. To begin with, we review concepts related to the random coefficient $U$.

\subsection{The Random Shear Flow}\label{subsection: random_shear_flow}
Recall that we have defined $\mathcal{W}$ as the unique solution to the stochastic heat equation:
\begin{equation}
    \begin{cases}
        d \mathcal{W} = \nu \Delta \mathcal{W} + \nu^{5/6} \Phi dV_t,\\
        \mathcal{W}(0) = \mathcal{W}_{in}.
    \end{cases}
\end{equation}
The random coefficients $U$ and $U''$ appearing in \eqref{fast_slow_syst} will eventually be estimated in terms of $||U'||_{L^\infty}, ||U''||_{L^\infty}$, and $||U'''||_{L^\infty}$, per the linear theory for nearly-Couette flows developed in \cite{bedrossian2023stability}. By Gagliardo-Nirenberg-Sobolev, these can be bounded by $||\mathcal{W}(t/\nu^{1-\gamma})||_{H^3}$. To get the linear estimates to close, the authors of \cite{bedrossian2023stability} assumed that $||\mathcal{W}||_{H^4} \leq \delta_0$ for some small constant $\delta_0$, independent of $\nu$ and depending only on $m$. This smallness condition was propagated forward for all time, since when $\Phi \equiv 0$, $\mathcal{W}$ simply solves a heat equation. In fact, the work of \cite{bedrossian2023stability} actually only requires control over the $H^3$ norm, so we will work in this paper at the $H^3$ rather than $H^4$ regularity level.

In the stochastic setting and on the long time scale of $t/\nu^{1-\gamma}$, rare events may cause $\mathcal{W}$ to suddenly increase in size and exceed these bounds in $H^3$. To that end, we set $c_0 = c_0(m) \coloneqq \frac{1}{2}\delta_0(m)$, which depends only on the regularity parameter $m$ and not on the viscosity $\nu$. We recall the stopping time:
\begin{equation}
    \sigma = \inf\{ t\in [0,T] \; | \; ||\mathcal{W}(t/\nu^{1-\gamma})||_{H^3} \geq 2c_0 \}.
\end{equation}
By standard Gaussian tail bounds, we have
\begin{equation}\label{eqn: tail_est_on_sigma}
    \Pb( \sigma < T) \leq c_1 \exp(-c_2 c_0^2 ||\Phi||^{-2} \nu^{-2/3 - \gamma} T^{-2}),
\end{equation}
for some constants $c_2$ and $c_1$ which are independent of $\nu$,  $||\Phi||$, and $T$.

We will later see that it is not enough to merely control $||\mathcal{W}(t/\nu^{1-\gamma})||_{H^3}$. The coefficients in \eqref{fast_slow_syst} and \eqref{fast_slow_lin} are time-dependent, but in the auxiliary process \eqref{fast_aux_process} they are considered stopped at the times $k\delta, k \in \{0,\hdots, K\}$. This will lead to the need to control terms like $\E\sup_{t \in [0,T]}||\mathcal{W}(t/\nu^{1-\gamma}) - \mathcal{W}(t(\delta)/\nu^{1-\gamma})||_{H^2}^2$. We make this precise in the following H\"older continuity lemma.
\begin{lemma}\label{lemma: holder_continuity_of_U}
    Let $\delta \in( \nu^{2/3-\gamma}, \nu^{1/3-\gamma})$ be given. There exists a constant $C > 0$ depending only on $||\mathcal{W}_{in}||_{H^3},$ $||\Phi||$, and $T$ such that
    \begin{equation}\label{eqn:  continuity_of_U}
        \E\sup_{t \in [0,T]}||U_t - U_{t(\delta)}||_{C^2}^2 \leq C\nu^{2/3}\delta \log(1/\delta).
    \end{equation}
\end{lemma}
\begin{proof}
For simplicity of notation, we set $F_t \coloneqq \mathcal{W}(t/\nu^{1-\gamma})$. This is just temporary notation for this proof. A more natural symbol would be $W_t$, but unfortunately, the symbol $W$ is already used for the Wiener process driving the fast motion. Notice that by Gagliardo-Nirenberg-Sobolev,
\begin{equation}
    ||U_t - U_{t(\delta)}||_{C^2} \leq C ||\partial_y (U_t - U_{t(\delta)})||_{H^2} \leq C ||F_t - F_{t(\delta)}||_{H^2}.
\end{equation}
We will therefore focus on estimating the differences in terms of $F$. Fix $t \in [0,T]$ and let $k \in \{0, \hdots, K\}$ be such that $k\delta = t(\delta)$. Since we have an explicit solution formula, we write
\begin{equation}\label{eqn: mild_form_of_S_heat}
    F_t - F_{k\delta} = (e^{\nu^{\gamma}(t-k\delta)\partial_y^2} - I) F_{k\delta} + \nu^{1/3 + \gamma/2}\int_{k\delta}^t e^{\nu^{\gamma}(t-s)\partial_y^2} \Phi dV_s.
\end{equation}
We apply the $H^2$ norm to \eqref{eqn: mild_form_of_S_heat} and use \eqref{heat_propogator}:
\begin{equation}
    ||F_t - F_{k\delta}||_{H^2} \leq C \nu^{3\gamma/2}(t-k\delta)^{3/2}||F_{k\delta}||_{H^3} + \nu^{1/3 + \gamma/2} \biggl| \biggl| \int_{k\delta}^t e^{\nu^{\gamma}(t-s)\partial_x^2} \Phi dV_s \biggr| \biggr|_{H^2}.
\end{equation}
Hence we find
\begin{equation}\label{eqn: intermediate_holder_est}
    \begin{split}
        \E \sup_{t \in [0, T]}||F_t - F_{t(\delta)}||_{H^2}^2 & \leq C \mathbb{E} \sup_{t \in [0,T]}\left(\nu^{3\gamma}(t-t(\delta))^{3}||F_{t(\delta)}||_{H^3}^2\right) \\
        &\quad+ \nu^{2/3 + \gamma} \mathbb{E} \sup_{t \in [0,T]} \biggl| \biggl| \int_{t(\delta)}^t e^{\nu^{\gamma}(t-s)\partial_x^2} \Phi dV_s \biggr| \biggr|_{H^2}^2.
    \end{split}
\end{equation}
The first term in \eqref{eqn: intermediate_holder_est} can be bounded by standard estimates as
\begin{equation}\label{eqn: final_Fk_est}
    C \mathbb{E} \sup_{t \in [0,T]}\left(\nu^{3\gamma}(t-t(\delta))^{3}||F_{t(\delta)}||_{H^3}^2\right) \leq \nu^{3\gamma}C\times(c_0^2 + \nu^{2/3+\gamma} T ||\Phi||^2) \delta^3 < C(||\Phi||_{H^3},T) \nu^{2/3}\delta \log(1/\delta),
\end{equation}
where we note that the $T$ dependence is like $T \nu^{2/3+\gamma}$. For the second term in \eqref{eqn: intermediate_holder_est} we split the $\sup_{t \in [0,T]}$ into $\max_{k \in \{0, \hdots, K\}} \sup_{t \in [k\delta, (k+1)\delta]}$. Let $G_k = \sup_{t \in [k\delta, (k+1)\delta]} \biggl| \biggl| \int_{k\delta}^t e^{\nu^{\gamma}(t-s)\partial_x^2} \Phi dV_s \biggr| \biggr|_{H^2}^2$. Note that for each $k$, $G_k$ is distributed like
\begin{equation}
    \sup_{t \in [0, \delta]} \biggl| \biggl| \int_{0}^t e^{\nu^{\gamma}(t-s)\partial_x^2} \Phi dV_s \biggr| \biggr|_{H^2}^2.
\end{equation}
By Burkholder-Davis-Gundy, we find
\begin{equation}
\begin{split}
    \mathbb{E} G_k = \mathbb{E}\sup_{t \in [0, \delta]} \biggl| \biggl| \int_{0}^t e^{\nu^{\gamma}(t-s)\partial_x^2} \Phi dV_s \biggr| \biggr|_{H^2}^2 \leq C \int_0^\delta || e^{\nu^{\gamma}s\partial_x^2} \Phi||_{H^2}^2 ds \leq C \nu^{-\gamma}||\Phi||_{H^3}^2 \delta.
    \end{split}
\end{equation}
Thus by standard Gaussian maximal inqualities.
\begin{equation}\label{eqn: final_Gk_estimate}
    \mathbb{E} \max_{k \in \{0, \hdots, K\}} G_k \leq C(||\Phi||_{H^3}) \nu^{-\gamma} \delta \log(K) \leq C(||\Phi||_{H^3}) \nu^{-\gamma} \delta \log((1+T)/\delta).
\end{equation}
Combining \eqref{eqn: final_Gk_estimate} with \eqref{eqn: final_Fk_est} and \eqref{eqn: intermediate_holder_est}, we arrive at
\begin{equation}
    \E \sup_{t \in [0, T]}||F_t - F_{t(\delta)}||_{H^2}^2 \leq C(||\Phi||, T) \nu^{2/3} \delta \log(1/\delta)
\end{equation}
as desired.
\end{proof}

\subsection{The Fast-Slow Energy}\label{subsection: energy_of_fast_slow}
In addition to the energies defined in Subsection \ref{outline}, we also require the following dissipations. For $x$-Fourier frequency $k$, let $\nabla_k \coloneqq (ik, \partial_y)$ and $\Delta_k \coloneqq -k^2 + \partial_y^{2}$. Then set
\begin{equation}
    \begin{split}
        D_k(f)&\coloneqq D_{k,\mathfrak{g}
        } + c_\mathfrak{t} D_{k,\mathfrak{t}} + c_\mathfrak{a} D_{k,\mathfrak{a}} +  c_\mathfrak{t} c_{\mathfrak{a}} D_{k, \mathfrak{t} \mathfrak{a}} + c_\mathfrak{b} D_{\mathfrak{b}, k}\\
        &\coloneqq \nu^\gamma|| \nabla_k f_k||_2^2 + c_\mathfrak{t}\nu^{-1+\gamma}|| k \nabla_k \Delta_k^{-1} f_k||_2^2 + c_{\mathfrak{a}} \nu^{2/3 + \gamma}|k|^{-2/3}|| \nabla_k \partial_y f_k||_2^2\\
        & \quad + c_\mathfrak{t} c_{\mathfrak{a}} \nu^{-1/3+ \gamma}|k|^{-2/3}|| k \nabla_k \Delta_k^{-1} \partial_y f_k||_2^2 +  c_\mathfrak{b} \nu^{-2/3 + \gamma} |k|^{2/3} ||f_k||_2^2,\\
        \mathcal{D}_{\neq}(f) &\coloneqq \sum_{k \neq 0} |k|^{2m} D_k(f),\\
        \mathcal{D}_0(f) &\coloneqq \nu^{\gamma}||\partial_y f_0||_2^2 + \nu^{2/3 + \gamma} ||\partial_y^2 f_0||_2^2,\\
        \mathcal{D}(f) &\coloneqq \mathcal{D}_0(f) + \mathcal{D}_{\neq}(f).
    \end{split}
\end{equation}
We pay special attention to $\mathcal{D}_\mathfrak{t} (f) \coloneqq \sum_{k \neq 0} |k|^{2m} D_{k,\mathfrak{t}}(f),$ which in \cite{bedrossian2023stability} gives the inviscid damping estimates. We note that compared to other components of the dissipation, $\mathcal{D}_\mathfrak{t}$ is capable of absorbing more inverse powers of $\nu$, and so it is the most beneficial of the dissipation terms. We also single out $\mathcal{D}_\mathfrak{b} (f) \coloneqq \sum_{k \neq 0} |k|^{2m} D_{k,\mathfrak{b}}(f),$ which stems from the enhanced dissipation. We now review the linear estimates:
\begin{lemma}\label{linear_deterministic}
    Let $U_t$ be as defined in \eqref{eqn: rescaling}. For fixed $\nu \in (0,1)$, define the time-dependent linear operator $\mathbb{L}_t \coloneqq \nu^{\gamma} \Delta - \nu^{-1+\gamma} U_t \partial_x + \nu^{-1+\gamma} U_t'' \partial_x \Delta^{-1}$. Then there choice of constants $c_\alpha, c_\beta, c_\tau > 0$, $\delta_0 >0$, and a constant $\delta_* = \delta_*(c_\alpha, c_\beta, c_\tau)$, all independent of $\nu$, such that the following holds. Suppose $||\mathcal{W}_{in}||_{H^3} \leq \delta_0$, and let $f\in H^1(\T \times [-1,1]) \cap \mathcal{H}_{\neq}$ be initial data for the initial value problem
    \begin{equation}
        \begin{cases}
            \partial_t g = \mathbb{L}_t g,\\
            g(0) = f.
        \end{cases}
    \end{equation}
    Then $\mathbb{P}$-almost surely, the following estimate holds for every $t < \sigma$:
    \begin{equation}
       \frac{d}{dt} \Ee_{\neq}(g) + 8\delta_* \D_{\neq}(g) + 8\delta_* \nu^{2/3-\gamma} \Ee_{\neq}(g) \leq 0.
    \end{equation}
\end{lemma}
This will give us the good decay properties from the linear semigroup in Section \ref{fast_with_frozen_section} which give rise to the invariant measure of the frozen system. 
A significant portion of \cite{bedrossian2023stability} is dedicated to proving estimates on the nonlinearity in order to produce a bootstrap estimate. By carefully examining the proofs contained in Section 5 of \cite{bedrossian2023stability}, one can extract the following information:
\begin{lemma}\label{nonlinear_deterministic}
    There exists a constant $C_*=C_*(m) > 0$ independent of $\nu$ such that for $X \in \mathcal{H}_{0},\; Y \in \mathcal{H}_{\neq}$:
    \begin{equation}
        \begin{split}
            & |\mathrm{Re}\langle b_0(Y), X \rangle_{\mathcal{H}_{0}}| = |\mathrm{Re}\langle B_0(Y), \partial_y X \rangle_{\mathcal{H}_{0}}| \leq \nu^{-\gamma/2} C_*^{1/2} \Ee_0^{1/2}( B_0(Y)) \D_0^{1/2}(X), \\
            &|\mathrm{Re}\langle b_{m}(X,Y), Y\rangle_{\mathcal{H}_{\neq}}| \leq C_*\nu^{-\gamma + 1/2} \Ee_{0}^{1/2}(X) \D_{\neq}(Y),\\
            &|\mathrm{Re}\langle b_{\neq}(Y),Y \rangle_{\mathcal{H}_{\neq}}| \leq C_* \nu^{-\gamma + 1/2} \Ee_{\neq}^{1/2}(Y) \D_{\neq}(Y).
        \end{split}
    \end{equation}
    Furthermore, we have for any $q \in [0,1/2]$:
    \begin{equation}
        \Ee_0^{1/2}(B_0(Y)) \leq C_*^{1/2}\nu^{q(1-\gamma/2)} \Ee_{\neq}^{1-q}(Y) \D_{\neq}^q(Y).
    \end{equation}
\end{lemma}
We will view $m \in (2/3,1)$ as fixed and allow all constants to depend on it implicitly. Hence we will not remark further on any $m$-dependence. In the context of the bootstrap estimate from \cite{bedrossian2023stability}, if one assumes sufficient smallness of the initial data, then the linear estimates of Proposition \ref{linear_deterministic} are sufficient to control the nonlinear terms arising in \eqref{nonlinear_deterministic}. In the deterministic case (i.e $\Phi =\Psi \equiv 0$), one could take $\alpha = \beta = \gamma= 0$, and then for $\Ee(X_0^\nu,Y_0^\nu) \leq \frac{\delta_*}{4C_*}$, the energy $\Ee(X_t^\nu, Y_t^\nu)$ remains small for all time. In original variables, this corresponds to the vorticity $\omega$ being of size $O(\nu^{1/2})$. 

In the stochastic setting, one expects that due to the forcing, potentially rare events can force the system out of equilibrium, even when starting from small initial data. This is a critical problem in considering \eqref{fast_slow_syst} as a fast-slow system. The fast-slow character of \eqref{fast_slow_syst} is revealed by the differing decay rate of the linear semigroups acting on $X$ and $Y$. However, with the nonlinearity present, we only expect the linearity to dominate within a certain basin of attraction, outside of which the nonlinearity is the main effect. In this regime, we expect the enhanced dissipation in \eqref{original_couette} to no longer be present, destroying the fast-slow structure. Hence, we introduce the following stopping times.

Recall the parameter $\alpha'$. In the explicit case of $\gamma = 0$ in Theorem \ref{main_theorem}, we set
\begin{equation}\label{def_of_alpha_prime}
    \alpha' \coloneqq \frac{\theta}{36}\min\{ \frac{a}{1+a}\frac{1}{6}, \beta, \alpha\}.
\end{equation}
In general, we will let $\alpha' \ll 2/3 - \gamma$ be a sufficiently small parameter as determined in the proof.  Furthermore, we introduce the constant $c_*$ as
\begin{equation}\label{def_of_c_star}
    c_* \coloneqq \frac{4 \delta_*^2}{C_*^2}.
\end{equation}
Let us define the constant $\nu_*$, which will serve as a smallness requirement on $\nu$:
\begin{equation}\label{def_of_nu_star}
    \nu_* = \min\{\nu_1, \nu_2, \nu_3\},
\end{equation}
where
    \begin{equation}
        \nu_1 = \min \left\{ 1/2,\left(\frac{c_* \delta_*}{64||\Psi||^2} \right)^{\frac{2}{\alpha'}} , \left(\frac{c_*}{16C_0'(1+||\Psi||^2 4\delta_*^{-1})}\right)^{\frac{2}{\alpha'}}, \left(\frac{c_*}{64}\right)^{\frac{2}{\alpha'}}\right\},
    \end{equation}
    where $C_0'=C_0'(||\Psi||) >0$ is a constant depending only on $||\Psi||$, $\delta_*$, and $C_*$, determined in the proof of Proposition \ref{energy_bounds_for_original},
    $\nu_2$ solves
    \begin{equation}\label{use_of_nu_2}
        \nu_2^{-2/3} \exp(-\nu_2^{-\alpha'/2}) = \nu_2^{2/3},
    \end{equation}
    and $\nu_3$ is given by
    \begin{equation}
        \nu_3 = \min\left\{\left(\frac{\delta_*}{2\tilde{C} ||\Psi||^2}\right)^{\frac{1}{\beta}}, \left(\frac{c_*}{8 C ||\Psi||^4}\right)^{\frac{1}{2\alpha'}}\right\},
    \end{equation}
    where $C, \tilde{C}>0$ are constants depending only on $\delta_*$, $C_*$, $c_*$, and $m$ determined in the proofs of Propositions \ref{prop: long_time_well_posed_ness} and \ref{key_prop}, respectively. This definition of $\nu_*$ will suffice for the majority of Propositions. In the final proof of Theorem \ref{main_theorem}, we will use $\nu_0$ defined by
    \begin{equation}\label{eqn: def_of_nu_0}
        \nu_0 = \min\{\nu_*, \nu_4\},
    \end{equation}
    where
    $\nu_4$ solves
    \begin{equation}
        \nu_4^{-2/3} = \tilde{c}||\Phi||_{H^3}^2 T^2\ln\left(20 \nu_4^{-2/3}\right),
    \end{equation}
    $\tilde{c} > 0$ is a constant depending only on $c_0 = c_0(m)$, and the constants $c_1$ and $c_2$ from \eqref{eqn: tail_est_on_sigma}. Concerning the relationship between $\nu_0$ and $T$, we note that the condition on $\nu_4$ gives roughly $\nu_0^{1/3 -} T \lesssim 1$.
We define
\begin{subequations}
    \begin{equation}
    \tau_Y \coloneqq \inf\{ t \in [0,T] \; \;  | \; \; \mathcal{E}(Y_t^\nu) \geq c_* \nu^{-2\alpha'} \},
\end{equation}
\begin{equation}
    \tau_X \coloneqq \inf\{ t \in [0,T] \; \;  | \; \; \mathcal{E}_0(X_t^\nu) \geq c_* \nu^{-2\beta} \},
\end{equation}
\end{subequations}
and with similar definitions define $\tau_{\tilde{Y}}, \tau_{\tilde{X}}, \tau_{\hat{Y}}$, and $\tau_{\hat{X}}$.
We introduce the local martingales
    \begin{equation}
    \begin{split}
               M_t \coloneqq \int_0^t \langle \Psi dW_s, Y_s^\nu \rangle, \;\;\; \tilde{M}_t \coloneqq \int_0^t \langle \Psi dW_s, \tilde{Y}_s^\nu \rangle, \;\;\; \hat{M}_t \coloneqq \int_0^t \langle \Psi dW_s, \hat{Y}_s^\nu \rangle.
    \end{split}
    \end{equation}
Then let
\begin{equation}
\begin{split}
        \tau_M \coloneqq \inf \{ t \in [0,T] \; | \; \nu^{-1/3 + \gamma/2}M_{t \wedge \tau_Y} - \frac{\delta_*}{||\Psi||^2}\frac{1}{2}\langle M \rangle_{t \wedge \tau_Y} \geq \frac{||\Psi||^2}{\delta_*} \nu^{-\alpha'/2}\},
\end{split}
\end{equation}
where $\langle M \rangle_t$ denotes the quadratic variation process of $M_t$. We define $\tau_{\tilde{M}}$ and $\tau_{\hat{M}}$ similarly. We are now prepared to prove the following control of the energy:

\begin{proposition}\label{energy_bounds_for_original}
   Suppose that $(Y_0^\nu, X_0^\nu)$ are initial data for \eqref{fast_slow_syst}. Then there exists a constant $C_0 >0$ such that if $\Ee_0(X_0^\nu) < C_0\nu^{-\alpha'}$ and $\Ee_{\neq}(Y_0^\nu) < C_0\nu^{-\alpha'/2}$,  then the following estimates hold for all $\nu < \nu_*$:
   \begin{equation}
       \begin{split}
       &\sup_{t \in [0,T\wedge \tau_{M}\wedge \sigma]} \Ee_{\neq}(Y_t^\nu) \leq \Ee_{\neq}(Y_0^\nu) + 4||\Psi||^2(1 + \delta_*^{-1}\nu^{- \alpha'/2}) \leq \frac{1}{2}c_*\nu^{-\alpha'},\\
       &\delta_*\int_0^{T\wedge \tau_M \wedge \sigma} \D_{\neq}(Y_s^\nu) ds \leq \Ee_{\neq}(Y_0^\nu) + ||\Psi||^2 (T\nu^{-2/3+ \gamma} + 4\delta_*^{-1}\nu^{-\alpha'/2}) \leq C_1(||\Psi||,T)(\nu^{-\alpha'}+ \nu^{-2/3+\gamma}),\\
           &\sup_{t \in [0,T \wedge \tau_M \wedge \sigma]}\Ee_{0}(X_{t}^\nu)  \leq \Ee_{0}(X_0^\nu) + C_2(||\Psi||) \nu^{-3\alpha'/2} \leq \frac{1}{2}c_* \nu^{-2\alpha'},
       \end{split}
   \end{equation}
    where $C_1(||\Psi||, T)$ represents a constant depending only on $||\Psi||$ and $T$ and $C_2(||\Psi||)$ is a constant depending on $||\Psi||$. Furthermore, $\tau_X \wedge \tau_Y \wedge \sigma \geq \tau_M \wedge \sigma$. Additionally we have the bound
    \begin{equation}
        \Pb(\tau_M \geq T) \geq 1-\exp(-\nu^{-\alpha'/2}).
    \end{equation}
\end{proposition}
\begin{proof}
    By It\^o's formula,
    \begin{equation}
    \begin{split}
                \mathcal{E}_{\neq}(Y_t^\nu) &= \Ee_{\neq}(Y_0^\nu)  + 2\nu^{\gamma}\int_0^t \mathrm{Re}\langle \partial_t Y_s^\nu, Y_s^\nu \rangle_{\mathcal{H}_{\neq}} ds + \nu^{-2/3 + \gamma} ||\Psi||^2 \int_0^t   ds + 2\nu^{-1/3+\gamma/2}M_t.
    \end{split}
    \end{equation}
    Applying Lemmas \ref{linear_deterministic} and \ref{nonlinear_deterministic}, we find for $t < \sigma$
\begin{equation}\label{initial_fast_est}
        \begin{split}
        \Ee_{\neq}(Y_t^\nu) + &8 \delta_* \int_0^t \D_{\neq}(Y_s^\nu) + \nu^{-2/3 + \gamma} \Ee_{\neq}(Y_s^\nu)ds  \leq \Ee_{\neq}(Y_0^\nu)  \\
                &\quad + C_* \int_0^t( \nu^{\alpha} \Ee_{\neq}^{1/2}(Y_s^\nu) + \nu^{\beta} \Ee_0^{1/2}(X_s^\nu))\D_{\neq}(Y_s^\nu)ds + \nu^{-2/3 + \gamma} ||\Psi||^2 \int_0^t ds + 2\nu^{-1/3+\gamma/2}M_t.
        \end{split}
    \end{equation}
    Then for times $t$ such that $t \leq \tau_Y \wedge \tau_X \wedge \sigma$, we have by the definition of $c_*$ \eqref{def_of_c_star} that $\Ee_{\neq}^{1/2}(Y_s^\nu) \leq 2 \delta_* C_*^{-1} \nu^{-\alpha}$ and $\Ee_{0}^{1/2}(X_s^\nu) \leq 2 \delta_* C_*^{-1} \nu^{-\beta}$. Then continuing from \eqref{initial_fast_est},
    \begin{equation}
        \begin{split}
        \Ee_{\neq}(Y_t^\nu) + 8 \delta_* \int_0^t \D_{\neq}(Y_s^\nu) + \nu^{-2/3 + \gamma} \Ee_{\neq}(Y_s^\nu)ds  &\leq \Ee_{\neq}(Y_0^\nu) +  4 \delta_* \int_0^t\D_{\neq}(Y_s^\nu)ds \\ &\quad + \nu^{-2/3 + \gamma} ||\Psi||^2 \int_0^t ds +2\nu^{-1/3+\gamma/2}M_t,
        \end{split}
    \end{equation}
    and so
\begin{equation}\label{bound_on_dissip}
        \begin{split}
            \Ee_{\neq}(Y_{t}^\nu) + &2 \delta_* \int_0^{t} \D_{\neq}(Y_s) + \nu^{-2/3+\gamma} \Ee_{\neq}(Y_s)ds  \leq \Ee_{\neq}(Y_0^\nu)  + \nu^{-2/3+ \gamma} ||\Psi||^2 t + 2\nu^{-1/3-\gamma/2}M_t.
        \end{split}
    \end{equation}
    Denote $A = \{ \tau_M \leq T\}$, and note that 
    \begin{equation}\label{alt_form_A}
        A = \left\{ \sup_{0\leq t \leq T}( \nu^{-1/3 + \gamma/2} M_{t \wedge \tau_Y} - \frac{\delta_*}{||\Psi||^2} \frac{1}{2} \nu^{-2/3 +\gamma} 
            \langle M \rangle_{t \wedge \tau_Y}) \geq \frac{||\Psi||^2}{\delta_*}\nu^{-\alpha'/2}\right\}.
    \end{equation}
    Since $M_{t \wedge \tau_Y}$ satisfies Novikov's condition (i.e $\E \exp( \frac{1}{2} \langle M\rangle_{T \wedge \tau_Y}) < \infty$), we have by the exponential submartingale inequality
    \begin{equation}
    \Pb(A) \leq \exp(-\nu^{-\alpha'/2}).
    \end{equation}
    For times $t < \tau_M \wedge \tau_Y \wedge \sigma$, we have by the alternative formulation of $A$ \eqref{alt_form_A} and \eqref{bound_on_dissip},
    \begin{equation}
        \begin{split}
            \Ee_{\neq}(Y_{t}^\nu)& + 2 \delta_* \int_0^{t} \D_{\neq}(Y_s) + \nu^{-2/3+\gamma} \Ee_{\neq}(Y_s)ds\\&  \leq \Ee_{\neq}(Y_0^\nu)  + \nu^{-2/3+ \gamma} ||\Psi||^2 t+ 2(\nu^{-1/3 + \gamma/2} M_t - \frac{\delta_*}{||\Psi||^2} \frac{1}{2} \nu^{-2/3 +\gamma} 
            \langle M \rangle_t)) + \frac{\delta_*}{||\Psi||^2} \nu^{-2/3 + \gamma} \langle M \rangle_t\\
            &\leq \Ee_{\neq}(Y_0^\nu)  + \nu^{-2/3+ \gamma} ||\Psi||^2 t + 2 ||\Psi||^2 \delta_*^{-1} \nu^{-\alpha'/2} + \delta_* \nu^{-2/3 + \gamma} \int_0^t \Ee_{\neq}(Y_s) ds.
        \end{split}
    \end{equation}
    Hence for $t < \tau_M \wedge \tau_Y \wedge \sigma$:
\begin{equation}\label{post_martingale_bound}
        \Ee_{\neq}(Y_{t}^\nu) + \delta_* \int_0^{t} \D_{\neq}(Y_s) + \nu^{-2/3+\gamma} \Ee_{\neq}(Y_s)ds \leq \Ee_{\neq}(Y_0^\nu)  + \nu^{-2/3+ \gamma} ||\Psi||^2 t + 4 ||\Psi||^2 \delta_*^{-1} \nu^{-\alpha'/2}.
    \end{equation}
    Notice that by Gr\"onwall's inequality, \eqref{post_martingale_bound} implies that for the energy $\Ee_{\neq}(Y_t)$ we have for $t < \tau_M \wedge \tau_Y \wedge \sigma$:
\begin{equation}\label{tight_energy_bound}
         \Ee_{\neq}(Y_{t}^\nu) \leq  e^{-\delta_* t /\nu^{2/3 -\gamma}}\Ee_{\neq}(Y_0^\nu) + 4 \delta_*^{-1}||\Psi||^2(1 + \nu^{- \alpha'/2}).
    \end{equation}
    Meanwhile for the slow energy we have from Lemma \ref{nonlinear_deterministic},
    \begin{equation}
        \begin{split}
            \Ee_0(X_{t}^\nu) + 8\delta_* \int_0^t\Ee_0(X_s^\nu) + \D_0(X_s^\nu)ds &\leq \Ee_0(X_0^\nu) + C_* \nu^{2\alpha - \beta} \int_0^t \Ee_{\neq}^{1/2}(Y_s) \D_{\neq}^{1/2}(Y_s) \D_0^{1/2}(X_s) ds,
        \end{split}
    \end{equation}
    which implies by Young's product inequality
    \begin{equation}\label{intermediate_slow_bound}
        \begin{split}
            \Ee_0(X_t^\nu) + 4\delta_* \int_0^t\Ee_0(X_s^\nu) + \D_0(X_s^\nu)ds &\leq \Ee_0(X_0^\nu) + \frac{C_*^2}{4 \delta_*} \nu^{4\alpha - 2 \beta}\sup_{s \in [0,t]}\Ee_{\neq}(Y_s^\nu)\int_0^t  \D_{\neq}(Y_s) ds.
        \end{split}
    \end{equation}
    Continuing with $t < \tau_M \wedge \tau_Y \wedge \sigma$, we apply \eqref{tight_energy_bound} to \eqref{intermediate_slow_bound} and arrive at
    \begin{equation}\label{penultimate_slow_bound}
        \begin{split}
            \Ee_0(X_t^\nu)+ 4 \delta_*\int_0^t \mathcal{E}(X_s^\nu)ds &\leq \Ee_0(X_0^\nu) + C \nu^{2/3-\gamma}\biggl(\Ee_{\neq}(Y_0^\nu)+ 4 \delta_*^{-1}||\Psi||^2(1 + \nu^{- \alpha'/2})\biggr)\\&\quad\quad\quad\quad\quad \biggl(\Ee_{\neq}(Y_0^\nu)  + \nu^{-2/3+ \gamma} ||\Psi||^2 t + 4||\Psi||^2 \delta_*^{-1} \nu^{-\alpha'/2}\biggr).
        \end{split}
    \end{equation}
    Hence by Gr\"onwall's inequality,
    \begin{equation}\label{final_slow_bound}
        \begin{split}
            \Ee_0(X_t^\nu) & \leq e^{-4\delta_* t}\biggl(\Ee_0(X_0^\nu) + C\nu^{2/3-\gamma}\biggl(\Ee_{\neq}(Y_0^\nu)+ 4 \delta_*^{-1}||\Psi||^2(1 + \nu^{- \alpha'/2})\biggr)\biggl(\mathcal{E}_{\neq}(Y_0^\nu ) + 4 ||\Psi||^2 \delta_*^{-1} \nu^{-\alpha'/2}\biggr)\biggr)\\
            &\quad\quad +C\frac{||\Psi||^2}{4\delta_*}\biggl(\Ee_{\neq}(Y_0^\nu)+ 4 \delta_*^{-1}||\Psi||^2(1 + \nu^{- \alpha'/2})\biggr).
        \end{split}
    \end{equation}
    Now we determine the constant $C_0 >0$ from the statement of the proposition. Let $C_0 = \frac{c_*}{16}.$ Then by the definition of $\nu_*$,
    $$C_0\nu^{-\alpha'/2} + 4\delta_*^{-1}||\Psi||^2(1+\nu^{-\alpha'/2}) < \nu^{-\alpha'}c_*/4,$$ 
   completing the estimates on $\Ee_{\neq}(Y_t^\nu)$ and the dissipation. Returning to \eqref{final_slow_bound}, under the assumption $\Ee_{\neq}(Y_0^\nu) \leq C_0 \nu^{-\alpha'/2}$, there exists a constant $C_0' = C_0'(||\Psi||)$ such that
    \begin{equation}\label{final_slow_bound_2}
        \begin{split}
            \Ee_0(X_t^\nu) 
            &\leq \Ee_0(X_0^\nu) + C_0'\times \biggr(1+ \frac{||\Psi||^2}{4 \delta_*}\nu^{-\alpha'/2}\biggl) \nu^{-\alpha'}.
        \end{split}
    \end{equation}
    Note that $C_0'$ depends only on $||\Psi||$, $\delta_*$, $C_*$, and universal constants. Importantly, it is independent of $\nu$ and $T$. Now under the assumption $\Ee_0(X_0^\nu) \leq C_0 \nu^{-\alpha'}$, we find
    $$\Ee_0(X_t^\nu) \leq C_0 \nu^{-\alpha'}+C_0'\times \biggr(1+ \frac{||\Psi||^2}{4 \delta_*}\nu^{-\alpha'/2}\biggl) \nu^{-\alpha'} \leq \frac{c_*}{2} \nu^{-2\alpha'}$$
    for all $\nu \in (0, \nu_*)$.
    We see from \eqref{tight_energy_bound}, our choice of $C_0$, \eqref{final_slow_bound_2}, and the definition of $\alpha'$ \eqref{def_of_alpha_prime}, that $t$ must hit $\tau_M \wedge \sigma$ before $\tau_Y \wedge \tau_X \wedge \sigma$. This completes the proof.
\end{proof}
Similar statements hold for the pseudo-linearized processes and the auxiliary processes. In particular, we have the following two propositions.
\begin{proposition}\label{energy_bounds_for_linear}
Suppose that $(Y_0^\nu, X_0^\nu)$ are initial data for \eqref{fast_slow_lin}. Then there exists a constant $C_0 > 0$ such that if $\Ee_0(X_0^\nu) < C_0\nu^{-\alpha'}$ and $\Ee_{\neq}(Y_0^\nu) < C_0\nu^{-\alpha'/2}$,  then the following estimates hold for all $\nu < \nu_*$:
   \begin{equation}
       \begin{split}
       &\sup_{t \in [0,T\wedge \tau_{\tilde{M}} \wedge \sigma]} \Ee_{\neq}(\tilde{Y}_t^\nu) \leq \Ee_{\neq}(Y_0^\nu) + 4||\Psi||^2(1 + \delta_*^{-1}\nu^{- \alpha'/2}) \leq \frac{1}{2}c_*\nu^{-\alpha'},\\
       &\delta_*\int_0^{T\wedge \tau_{\tilde{M}} \wedge \sigma} \D_{\neq}(\tilde{Y}_s^\nu) ds \leq \Ee_{\neq}(Y_0^\nu) + ||\Psi||^2 (T\nu^{-2/3+ \gamma} + 4\delta_*^{-1}\nu^{-\alpha'/2}) \leq C_1(||\Psi||, T)(\nu^{-\alpha'}+ \nu^{-2/3+\gamma}),\\
           &\sup_{t \in [0,T \wedge \tau_{\tilde{M}} \wedge \sigma]}\Ee_{0}(\tilde{X}_{t}^\nu)  \leq \Ee_{0}(X_0^\nu) + C_2(||\Psi||) \nu^{-3\alpha'/2} \leq \frac{1}{2}c_* \nu^{-2\alpha'},
       \end{split}
   \end{equation}
    where $C_1(||\Psi||, T)$ represents a constant depending only on $||\Psi||$ and $T$ and $C_2(||\Psi||)$ is a constant depending on $||\Psi||$. Furthermore, $\tau_{\tilde{X}} \wedge \tau_{\tilde{Y}} \wedge \sigma \geq \tau_{\tilde{M}} \wedge \sigma$. Additionally we have the bound
    \begin{equation}
        \Pb(\tau_{\tilde{M}} \geq T) \geq 1-\exp(-\nu^{-\alpha'/2}).
    \end{equation}
\end{proposition}
\begin{proposition}\label{energy_bounds_for_auxiliary}
Suppose that $(Y_0^\nu, X_0^\nu)$ are initial data for \eqref{fast_aux_process} and \eqref{slow_aux_process} respectively. Then there exists a constant $C_0 > 0$ such that if $\Ee_0(X_0^\nu) < C_0\nu^{-\alpha'}$ and $\Ee_{\neq}(Y_0^\nu) < C_0\nu^{-\alpha'/2}$,  then the following estimates hold for all $\nu < \nu_*$:
   \begin{equation}
       \begin{split}
       &\sup_{t \in [0,T\wedge \tau_{\hat{M}} \wedge\tau_{\tilde{M}} \wedge \sigma ]} \Ee_{\neq}(\hat{Y}_t^\nu) \leq \Ee_{\neq}(Y_0^\nu) + 4||\Psi||^2(1 + \delta_*^{-1}\nu^{- \alpha'/2}) \leq \frac{1}{2}c_*\nu^{-\alpha'},\\
       &\delta_*\int_0^{T\wedge  \tau_{\hat{M}} \wedge \tau_{\tilde{M}} \wedge \sigma} \D_{\neq}(\hat{Y}_s^\nu) ds \leq \Ee_{\neq}(Y_0^\nu) + ||\Psi||^2 (T\nu^{-2/3+ \gamma} + 4\delta_*^{-1}\nu^{-\alpha'/2}) \leq C_1(||\Psi||, T)(\nu^{-\alpha'}+ \nu^{-2/3+\gamma}),\\
           &\sup_{t \in [0,T \wedge  \tau_{\hat{M}} \wedge\tau_{\tilde{M}} \wedge \sigma ]}\Ee_{0}(\hat{X}_{t}^\nu)  \leq \Ee_{0}(X_0^\nu) + C_2(||\Psi||) \nu^{-3\alpha'/2} \leq \frac{1}{2}c_* \nu^{-2\alpha'}.
       \end{split}
   \end{equation}
    where $C_1(||\Psi||, T)$ represents a constant depending only on $||\Psi||$ and $T$ and $C_2(||\Psi||)$ is a constant depending on $||\Psi||$. Furthermore, $\tau_{\hat{X}} \wedge \tau_{\hat{Y}} \wedge \sigma >  \tau_{\hat{M}} \wedge\tau_{\tilde{M}} \wedge \sigma$.  Additionally we have the bound
    \begin{equation}
      \Pb( \tau_{\hat{M}} \wedge\tau_{\tilde{M}} \geq T) \geq 1 - 2\exp(-\nu^{-\alpha'/2}).
    \end{equation}
\end{proposition}

In addition to the spatial regularity bounds of the previous propositions, we will require H\"older-in-time estimates on $\tilde{X}_t^\nu$ in Section \ref{section_of_propositions}. We will do so by trading for regularity in space. We now estimate this spatial H\"older regularity.
\begin{proposition}\label{proposition: holder_in_space}
    Let $(\tilde{X}_t^\nu, \tilde{Y}_t^\nu)$ solve \eqref{fast_slow_lin} with initial data satisfying $\Ee_0(X_0^\nu) < C_0\nu^{-\alpha'}$ and $ \Ee_{\neq}(Y_0^\nu) < C_0\nu^{-\alpha'/2}$. Let $a \in (0,1)$ be given. For any $\rho \in (0, a/2)$, there exists a constant $C = C(T,a, \rho)>0$ such that for all $\nu < \nu_*$:
\begin{equation}\label{holder_low_reg}
    \sup_{t \in [0,T\wedge \tau \wedge \sigma]}|| |\partial_y|^a \tilde{X}_t^\nu||_{L^2} \leq |||\partial_y|^{\alpha} \tilde{X}_0^\nu||_{L^2} + C \nu^{ -a/6 - \rho/3 -\alpha'(3/2-a - \rho)},
\end{equation}
\begin{equation}\label{holder_high_reg}
    \sup_{t \in [0,T\wedge\tau \wedge \sigma]} ||\partial_y |\partial_y|^a\tilde{X}_t^\nu||_{L^2} \leq ||\partial_y |\partial_y|^a \tilde{X}_0^\nu||_{L^2} + C \nu^{-1/3-a/6 - a\gamma/2 - \rho/3 -\alpha'(3/2-a - \rho)}.
\end{equation}
\end{proposition}
\begin{proof}
    We write by the mild formulation
    \begin{equation}
        |\partial_y|^a \tilde{X}_t^\nu = |\partial_y|^{a}e^{\nu^{\gamma}t\partial_y^2}  \tilde{X}_0^{\nu} - \nu^{\gamma/2 - 1/6} \int_0^t |\partial_y|^{a}e^{\nu^{\gamma}(t-s)\partial_y^2} b_0(\tilde{Y}_s^\nu)ds.
    \end{equation}
    We note that $b_0(\tilde{Y}_s^\nu)$ vanishes on the boundary of $[-1,1]$, i.e it is in $L^2_0$. Then using boundedness of the heat semigroup and \eqref{heat_propogator_2},
    \begin{equation}
        \begin{split}
            |||\partial_y|^{a}\tilde{X}_t^{\nu}||_{L^2} & \leq ||e^{\nu^{\gamma}t\partial_y^2} |\partial_y|^{a} \tilde{X}_0^{\nu}||_{L^2} + \nu^{-1/6+\gamma/2}\int_0^t |||\partial_y|^{a}e^{\nu^{\gamma} t\partial_y^2}  b_0(\tilde{Y}_s^\nu)||_{L^2} ds\\
            &\leq ||\partial_y^{a} \tilde{X}_0^\nu||_{L^2} + \nu^{-1/6 -a\gamma/2} C\int_0^t (t-s)^{-1/2-a/2} ||B_0(\tilde{Y}_s^\nu)||_{L^2} ds.
        \end{split}
    \end{equation}
    This holds since $|||\partial_y|^{-1} b_0(\tilde{Y}_s^\nu)||_{L^2} = |||\partial_y|^{-1} \partial_y B_0(\tilde{Y}_s^\nu)||_{L^2} = ||B_0(\tilde{Y}_s^\nu)||_{L^2}$. We remark that although we have written $|\partial_y|^{-1} \partial_y B_0$, and we typically avoid the expression $|\partial_y|^{-1} \partial_y$ applied to elements of $\mathcal{H}_0$, it is legal here since both $B_0$ and $b_0 = \partial_y B_0$ vanish on the boundary of $[-1,1]$. Hence $\partial_y B_0$ admits an expansion in the sine basis to which we can apply $|\partial_y|^{-1}$, and the equality in norm holds. Working with the explicit structure of $B_0(\tilde{Y}_s^\nu)$, we have by Lemma \ref{nonlinear_deterministic} for any $q \in [0, 1/2]$,
    \begin{equation}
        \begin{split}
            ||B_0(\tilde{Y}_s^\nu)||_{L^2} \leq C \Ee_{\neq}(\tilde{Y}_s^\nu)^{1/2 + q} \nu^{q} \D_{\neq}^{q}(\tilde{Y}_s^\nu).
        \end{split}
    \end{equation}
    Setting $q = 1/2 - a/2 - \rho$ for $0<\rho < a/2 $, we find by H\"older's inequality and Proposition \ref{energy_bounds_for_linear},
    \begin{equation}\label{eqn: holders_ineq_forhold_cont}
        \begin{split}
            \int_0^t (t-s)^{-1/2-a/2} ||B_0(\tilde{Y}_s^\nu)||_{L^2} ds&\leq C \nu^{q(1-\gamma)} \int_0^t (t-s)^{-1/2-a/2}\Ee_{\neq}(\tilde{Y}_s^\nu)^{1/2 + q}  \D_{\neq}^{q}(\tilde{Y}_s^\nu)ds\\
            &\leq C\nu^{q(1-\gamma)-\alpha'(1/2+q)} \left(\int_0^t (t-s)^{\frac{-1/2-a/2}{1-q}}ds\right)^{1-q}\left(\int_0^t \D_{\neq}(\tilde{Y}_s^\nu) ds \right)^{q}\\
            &\leq C\nu^{q/3-\alpha'(1/2+q)},
        \end{split}
    \end{equation}
    so long as $t < \tau \wedge \sigma$. Hence
    \begin{equation}
        \begin{split}
        |||\partial_y|^{a}\tilde{X}_t^{\nu}||_{L^2} &\leq|||\partial_y|^{a}\tilde{X}_0^{\nu}||_{L^2} + C\nu^{ -a/6 - a\gamma/2 - \rho/3 -\alpha'(3/2-a - \rho)}.
        \end{split}
    \end{equation}
    Having established \eqref{holder_low_reg}, we turn to \eqref{holder_high_reg}. We have by \eqref{heat_propogator_2},
    \begin{equation}
        \begin{split}
            ||\partial_y |\partial_y|^a\tilde{X}_t^\nu||_{L^2} & \leq ||\partial_y |\partial_y|^a \tilde{X}_0^\nu||_{L^2} + \nu^{-1/6 - a\gamma/2}  \int_0^t || \partial_y |\partial_y|^a e^{\nu^{\gamma} (t-s)\partial_y^2} b_0(\tilde{Y}_s)||_{L^2} ds\\
            &\leq ||\partial_y |\partial_y|^a \tilde{X}_0^\nu||_{L^2} + \int_0^t (t-s)^{-1/2-a/2}||b_0(\tilde{Y}_s^\nu)||_{L^2} ds.
        \end{split}
    \end{equation}
    By the work done in Section 5 of \cite{bedrossian2023stability}, one can show that for any $q \in [0,1/2]$:
    \begin{equation}
        \begin{split}
            ||b_0(\tilde{Y}_s^\nu)||_{L^2} \leq C(\nu^{q(2/3-\gamma)} +\nu^{-1/3 +q(1-\gamma)})\Ee_{\neq}^{1/2 + q}(\tilde{Y}_s^\nu) \D_{\neq}^{1/2 -q}(\tilde{Y}_s^\nu).
        \end{split}
    \end{equation}
    Now setting $q = 1/2 - a/2 + \rho$, we have $\nu^{q(2/3-\gamma)} +\nu^{-1/3 +q(1-\gamma)} \leq C \nu^{-1/3+q(1-\gamma)}$. Then we use H\"older's inequality and Proposition \ref{energy_bounds_for_linear} as in \eqref{eqn: holders_ineq_forhold_cont}, we obtain
    \begin{equation}
        \begin{split}
            ||\partial_y |\partial_y|^a\tilde{X}_t^\nu||_{L^2} & \leq ||\partial_y |\partial_y|^a \tilde{X}_0^\nu||_{L^2} + C(T, a, \rho) \nu^{-1/3-a/6 - a\gamma/2 -\alpha'(3/2-a-\rho)},
        \end{split}
    \end{equation}
    as desired.
\end{proof}
In Section \ref{section_of_propositions}, we will see that to estimate the H\"older continuity of the pseudo-linearized process, it is not enough to control the martingales on the interval $[0,T]$. Instead, we will need H\"older continuity estimates on intervals of size $\delta$, where $\delta$ is the step-size discussed previously for our Khasminskii discretization. For each $k \in \{0,\hdots, K\}$ we define
\begin{equation}
    \tilde{M}_t^{(k)} \coloneqq \tilde{M}_t - \tilde{M}_{k\delta} = \int_{k\delta}^t \langle\Psi dW_s, \tilde{Y}_s \rangle.
\end{equation}
Taking the place of $\tau_{\tilde{M}}$, we define
\begin{equation}
    \begin{split}
        \tau^{(k)}_{\tilde{M}} \coloneqq \inf \{k\delta \leq t \leq (k+1)\delta \wedge T\; | \;  \nu^{-1/3 + \gamma/2} \tilde{M}_{t \wedge \tau_{\tilde{Y}}}^{(k)} - \frac{\delta_*}{||\Psi||^2} \frac{1}{2} \nu^{-2/3+\gamma} \langle \tilde{M}^{(k)} \rangle_{t \wedge \tau_{\tilde{Y}}} =  \frac{||\Psi||^2}{\delta_*} \nu^{-\alpha'/2}\}.
    \end{split}
\end{equation}
We adopt the convention that the infimum of the empty set is $ \infty$. Then we define
\begin{equation}\label{definition_of_tau}
    \tau \coloneqq \tau_M  \wedge \tau_{\tilde{M}}\wedge \tau_{\hat{M}}\wedge \min\left\{ \tau^{(k)}_{\tilde{M}} \; | \; k \in \{0,\hdots, K\} \right\}.
\end{equation}
We will show in Proposition \ref{continuity_prop} that
\begin{equation}
    \Pb(\tau > T) \geq 1 - (T/\delta+6) \exp(-\nu^{-\alpha'/2}).
\end{equation}

We have seen how the previously defined stopping times give us a pseudo-bootstrap control over the relevant energies with high probability. On the rare event $\{\tau < T\}$, we use a more basic $L^2$ based estimate.
\begin{proposition}\label{proposition: crude_bound}
    Suppose that $(X_0^\nu, Y_0^\nu)$ are $L^2$ initial data for \eqref{fast_slow_syst}. For all $p\geq 1$, there exists a constant $C = C(p) > 0$ such that the following holds:
    \begin{equation}
        \begin{split}
             \E \sup_{t \in [0,T]}||X_t^\nu||_{L^2}^{2p}\leq& C\biggl(||X_0^\nu||_{L^2}^{2p} + \nu^{2p(\alpha-\beta)}||Y_0^\nu||_{L^2}^{2p}+\nu^{-p(1+2\beta)}||\mathcal{W}_{in}||_{L^2}^{2p}\\
             &\quad+\nu^{p(-2/3 - 2\alpha -2\beta+ \gamma)}||\Psi||^{2p}T^p + \nu^{p(1/3-2\beta+\gamma)}||\Phi||_{H^3}^{2p}T^p\biggr).
        \end{split}
    \end{equation}
\end{proposition}

\begin{proof}
    Consider the original Couette system \eqref{original_couette}, with solution $\omega$ corresponding to the initial data $\omega(0) = (\nu^{1/2+\beta}X_0, \nu^{1/2+\alpha} Y_0)$. Let us define $\tilde{\omega} \coloneqq \mathcal{W} + \omega$ and observe that $\tilde{\omega}$ satisfies 
    \begin{equation}\label{eqn: partial_couette_system}
    \begin{cases}
                d \tilde{\omega} = \nu \Delta \tilde{\omega} -y\partial_x \tilde{\omega}  - \tilde{u} \cdot \nabla \tilde{\omega}dt + \nu^{5/6} \Phi dV_t + \nu^{2/3 + \alpha}\Psi dW_t,\\
        \tilde{u} =  \nabla^{\perp} \Delta^{-1} \tilde{\omega} = (\partial_y \Delta^{-1} \tilde{\omega}, -\partial_x \Delta^{-1} \tilde{\omega}), \; \; \tilde{\omega}(t,x,y=\pm1) = 0,\\
        \tilde{\omega}(t = 0,x,y) = \tilde{\omega}_{in}(x,y) = \mathcal{W}_{in}(x,y)+\omega_{in}(x,y).
    \end{cases}    
\end{equation}
That is, $\tilde{\omega}$ expresses the vorticity as a perturbation around the true Couette flow. By It\^o's formula,
\begin{equation}\label{eqn: couette_energy}
        \begin{split}
            ||\tilde{\omega}||^2_{L^2} &= ||\tilde{\omega}_{in}||_{L^2}^2 + \int_0^t 2\nu\mathrm{Re}\langle \Delta\tilde{\omega}, \tilde{\omega}\rangle_{L^2} - 2\mathrm{Re}\langle y  \partial_x \tilde{\omega}, \tilde{\omega}\rangle_{L^2}ds\\
            &\quad- \int_0^t2\mathrm{Re}\langle u \cdot \nabla \tilde{\omega}, \tilde{\omega} \rangle_{L^2} ds + \nu^{4/3+ 2\alpha}||\Psi||_{L^2 \to L^2}^2t + \nu^{5/3} ||\Phi||_{L^2 \to L^2}^2 t\\
            &\quad+\nu^{2/3+\alpha}\int_0^t 2\mathrm{Re}\langle \Psi dW_s, \tilde{\omega}\rangle_{L^2}+\nu^{5/6}\int_0^t 2\mathrm{Re}\langle \Phi dV_s, \tilde{\omega}\rangle_{L^2}.
        \end{split}
    \end{equation}
    By the divergence free condition, $\mathrm{Re}\langle u \cdot \nabla \tilde{\omega}, \tilde{\omega} \rangle_{L^2} = 0$. Meanwhile, integration by parts gives $\mathrm{Re}\langle y  \partial_x \omega, \omega\rangle_{L^2} = 0$. Thus by integration by parts and the Dirichlet boundary conditions,
    \begin{equation}
    \begin{split}
                 ||\tilde{\omega}||^2_{L^2} + 2\nu \int_0^t||\nabla \tilde{\omega}||^2_{L^2}  ds&= ||\tilde{\omega}_{in}||_{L^2}^2 + \left(\nu^{4/3+ 2\alpha}||\Psi||_{L^2 \to L^2}^2+\nu^{5/3}||\Phi||_{L^2 \to L^2}\right)t\\
                 &\quad+ \int_0^t 2\mathrm{Re}\langle \nu^{2/3+\alpha}\Psi dW_s + \nu^{5/6} dV_s, \tilde{\omega}\rangle_{L^2}.
    \end{split}
    \end{equation}
    Then for $p \geq 1$, It\^o's formula gives
    \begin{equation}\label{eqn: couette_energy_2}
    \begin{split}
              ||\tilde{\omega}||^{2p}_{L^2} + &\nu \int_0^t 2p||\tilde{\omega}||_{L^2}^{2(p-1)}||\nabla \tilde{\omega}||_{L^2}^2 ds = ||\tilde{\omega}_{in}||_{L^2}^{2p} + p\left(\nu^{4/3+ 2\alpha}  ||\Psi||_{L^2 \to L^2}^2 + \nu^{5/6}||\Phi||_{L^2}\right) \int_0^t ||\tilde{\omega}||_{L^2}^{2(p-1)} ds\\
              &\quad +2p(p-1)\int_0^t||\tilde{\omega}||_{L^2}^{2(p-2)}\left(\nu^{4/3+2\alpha}||\Psi^*\tilde{\omega}||_{L^2}^2 +\nu^{5/3}||\Phi^*\tilde{\omega}||_{L^2}^2\right)ds\\
              &\quad+ \int_0^t 2p ||\tilde{\omega}||_{L^2}^{2(p-1)}\mathrm{Re}\langle \nu^{2/3+\alpha}\Psi dW_s + \nu^{5/6} \Phi dV_s, \tilde{\omega}\rangle_{L^2}.
    \end{split}
    \end{equation}
    Next, we take the supremum of \eqref{eqn: couette_energy_2} until time $T/\nu^{1-\gamma}$. We then apply expectation and the Burkholder-Davis-Gundy inequality:
    \begin{equation}
        \begin{split}
            \E \sup_{t \in [0,T/\nu^{1-\gamma}]}||\tilde{\omega}(t)||_{L^2}^{2p}&\leq ||\tilde{\omega}_{in}||_{L^2}^{2p} + C(p)\left(\nu^{1/3 + 2\alpha + \gamma}||\Psi||^2 + \nu^{2/3+\gamma}||\Phi||_{H^3}^2\right)T\E\sup_{t\in [0,T/\nu^{1-\gamma}]}||\tilde{\omega}||_{L^2}^{2(p-1)} \\
            &\quad+2p\E\sup_{t \in [T/\nu^{1-\gamma}]}\int_0^t ||\tilde{\omega}||_{L^2}^{2(p-1)}\mathrm{Re}\langle \nu^{2/3+\alpha}\Psi dW_s + \nu^{5/3} \Phi dV_s, \tilde{\omega}\rangle_{L^2}\\
            &\leq ||\tilde{\omega}_{in}||^{2p}_{L^2} + \left(C(p)(\nu^{1/3 - 2\alpha + \gamma}||\Psi||^2 +\nu^{2/3+\gamma}||\Phi||_{H^3}^2)T\right)^{p} + \frac{1}{4}\E\sup_{t \in [T/\nu^{1-\gamma}]}||\tilde{\omega}||_{L^2}^{2p}\\
            &\quad+ C(p)\left(\nu^{2/3+\alpha}||\Psi|| + \nu^{5/6}||\Phi||_{H^3}\right)\E\left(\int_0^{T/\nu^{1-\gamma}} ||\tilde{\omega}||_{L^2}^{2(p-1)}ds\right)^{1/2}\\
            &\leq ||\tilde{\omega}_{in}||^{2p} + \left(C(p)(\nu^{1/3 - 2\alpha + \gamma}||\Psi||^2 +\nu^{2/3+\gamma}||\Phi||_{H^3}^2)T\right)^{p} + \frac{1}{2}\E\sup_{t\in [0,T/\nu]} ||\tilde{\omega}||_{L^2}^{2p}.
        \end{split}
    \end{equation}
    This yields
    \begin{equation}\label{eqn: tilde_estimate}
         \E \sup_{t \in [0,T/\nu^{1-\gamma}]}||\tilde{\omega}(t)||_{L^2}^{2p} \leq C||\tilde{\omega}_{in}||^{2p} + C\left(\nu^{p/3 - 2p\alpha + p\gamma}||\Psi||^{2p} + \nu^{2p/3+p\gamma}||\Phi||_{H^3}^{2p}\right)T^p.
    \end{equation}
     Next, we recall $\tilde{\omega} = \mathcal{W} + \omega = \mathcal{W}+ \omega_{\neq} + \omega_{0}$. Hence by the triangle inequality
     \begin{equation}\label{eqn: partial_size_est}
         \begin{split}
             \E \sup_{t \in [0,T/\nu^{1-\gamma}]}||\omega_0(t)||_{L^2}^{2p} &\leq \E \sup_{t \in [0,T/\nu^{1-\gamma}]}||\omega(t)||_{L^2}^{2p}\\
             &\leq C\left(\E \sup_{t \in [0,T/\nu^{1-\gamma}]}||\tilde{\omega}(t)||_{L^2}^{2p} + \E \sup_{t \in [0,T/\nu^{1-\gamma}]}||\mathcal{W}(t)||_{L^2}^{2p}\right)
         \end{split}
     \end{equation}
     Applying \eqref{eqn: tilde_estimate} to \eqref{eqn: partial_size_est}, as was as standard results for stochastic heat equations yields
     \begin{equation}
         \begin{split}
              \E \sup_{t \in [0,T/\nu^{1-\gamma}]}||\omega_0(t)||_{L^2}^{2p}&\leq C\left(||\omega_{in}||_{L^2}^{2p} + ||\mathcal{W}_{in}||_{L^2}^{2p}+\nu^{p/3 - 2p\alpha + p\gamma}||\Psi||^{2p}T^p + \nu^{2p/3+p\gamma}||\Phi||_{H^3}^{2p}T^p\right)
         \end{split}
     \end{equation}
     Re-scaling to $X_t^\nu = \nu^{1/2+\beta} \omega_0(t)$ give
     \begin{equation}
        \begin{split}
             \E \sup_{t \in [0,T]}||X_t^\nu||_{L^2}^{2p}\leq& C\biggl(||X_0^\nu||_{L^2}^{2p} + \nu^{2p(\alpha-\beta)}||Y_0^\nu||_{L^2}^{2p}+\nu^{-p(1+2\beta)}||\mathcal{W}_{in}||_{L^2}^{2p}\\
             &\quad+\nu^{p(-2/3 - 2\alpha -2\beta+ \gamma)}||\Psi||^{2p}T^p + \nu^{p(1/3-2\beta+\gamma)}||\Phi||_{H^3}^{2p}T^p\biggr).
        \end{split}
    \end{equation}
     as desired.
\end{proof}

\subsection{New Dissipation Estimates}
The linear and nonlinear estimates of Lemmas \ref{linear_deterministic} and \ref{nonlinear_deterministic} as proved in \cite{bedrossian2023stability} will be key tools in the ultimate proof of Theorem \ref{main_theorem}. Because we are working with the time-dependent (and stochastic) flow $U$ rather than the simple Couette flow $y$, we will be required to estimate terms involving the time differences of $U$. We have seen in \ref{subsection: random_shear_flow} how to obtain continuity bounds on $U$. Below we will state the necessary modifications to Lemma \ref{linear_deterministic} to account for the time differences. The proofs consist of simple applications of integration by parts and the usage of the commutator relation $||[ \mathfrak{J}_k, \partial_y ]||_{L^2 \to L^2} \lesssim |k|$ proven in \cite{bedrossian2023stability}. Recall that we have previously defined $\D_{\mathfrak{t}}$ and $\D_{\mathfrak{b}}$. For $* \in \{ \mathfrak{g}, \mathfrak{a}, \mathfrak{ta}\}$, we also define 
$$\D_{*} = \sum_{k \neq 0} |k|^{2m} D_{k,*}(f_k).$$
We now give the following two lemmas, the proofs of which are omitted.
\begin{lemma}\label{lemma: difference_of_U}
    Let $f,g \in \mathcal{H}$. and let $U$, $V \in H^2([-1,1])$. Then there exists a constant $C$ depending only on $c_{\mathfrak{a}}$, $c_{\mathfrak{b}}$, and $c_{\mathfrak{t}}$ such that:
   \begin{subequations}
       \begin{equation}
       \begin{split}
                      \sum_{k\neq0}|k|^{2m} \mathrm{Re}&\langle (I + c_{\mathfrak{t}} \mathfrak{J}_k)(U - V) ik f_k, g_k \rangle_{L^2([-1,1])} \\
                      &\leq C \nu^{1/2 -\gamma}||U-V||_{L^\infty} \D^{1/4}_{\mathfrak{b}}(f) \D^{1/4}_{\mathfrak{g}}(f)\D_{\mathfrak{b}}^{1/2}(g),
       \end{split}
       \end{equation}
       \begin{equation}
       \begin{split}
                      \sum_{k\neq0}|k|^{2m-2/3} \nu^{2/3} \mathrm{Re}&\langle (I + c_{\mathfrak{t}} \mathfrak{J}_k)\partial_y\left((U - V) ik f_k\right), \partial_yg_k \rangle_{L^2([-1,1])}\\& \leq C\nu^{1/2 -\gamma}||U_t-U_s||_{L^\infty} \D^{1/4}_{\mathfrak{b}}(f) \D^{1/4}_{\mathfrak{g}}(f)\D_{\mathfrak{a}}^{1/2}(g),
       \end{split}
       \end{equation}
       \begin{equation}
       \begin{split}
                      \sum_{k\neq0}|k|^{2m-4/3} \nu^{1/3} &\mathrm{Re}\langle ik\left((U - V) ik f_k\right), \partial_y g_k \rangle_{L^2([-1,1])} \\
                      &\leq C\nu^{1/2 -\gamma}||U - V||_{L^\infty} \D^{1/4}_{\mathfrak{b}}(f) \D^{1/4}_{\mathfrak{g}}(f)\D_{\mathfrak{g}}^{1/2}(g).
       \end{split}
       \end{equation}
   \end{subequations}
\end{lemma}
\begin{lemma}\label{lemma: difference_of_U_derivs}
    Let $f,g \in \mathcal{H}$. and let $U, V \in H^2([-1,1])$. Then there exists a constant $C$ depending only on $c_{\mathfrak{a}}$, $c_{\mathfrak{b}}$, and $c_{\mathfrak{t}}$ such that:
   \begin{subequations}
       \begin{equation}
       \begin{split}
                      \sum_{k\neq0}|k|^{2m} \mathrm{Re}&\langle (I + c_{\mathfrak{t}} \mathfrak{J}_k)(U'' - V'') ik \Delta_k^{-1} f_k, g_k \rangle_{L^2([-1,1])} \\
                      &\leq C \nu^{5/6 - \gamma}||U'' - V''||_{L^\infty} \D^{1/2}_{\mathfrak{t}}(f)\D_{\mathfrak{b}}^{1/2}(g),
       \end{split}
       \end{equation}
       \begin{equation}
       \begin{split}
                      \sum_{k\neq0}|k|^{2m-2/3} \nu^{2/3} \mathrm{Re}&\langle (I + c_{\mathfrak{t}} \mathfrak{J}_k)\partial_y\left((U'' - V'') ik \Delta_k^{-1}f\right), \partial_yg_k \rangle_{L^2([-1,1])}\\& \leq C\nu^{5/6 -\gamma}||U'' - V''||_{L^\infty} \D^{1/2}_{\mathfrak{t}}(f) \D_{\mathfrak{a}}^{1/2}(g),
       \end{split}
       \end{equation}
       \begin{equation}
       \begin{split}
                      \sum_{k\neq0}|k|^{2m-4/3} \nu^{1/3} &\mathrm{Re}\langle ik\left((U'' - V'') ik \Delta_k^{-1}f\right), \partial_y g_k \rangle_{L^2([-1,1])} \\
                      &\leq C\nu^{5/6 -\gamma}||U'' - V''||_{L^\infty} \D^{1/2}_{\mathfrak{t}}(f) \D_{\mathfrak{a}}^{1/2}(g).
       \end{split}
       \end{equation}
   \end{subequations}
\end{lemma}

\section{Analysis of the Fast Process with Frozen Slow Component}\label{fast_with_frozen_section}
Recall the frozen system \eqref{frozen_system} obtained from the pesudo-linearized fast process with frozen slow component. We fix the initial data to be arbitrary $Z_0^{U,X,Y} = Y \in \mathcal{H}_{\neq}$. We also fix the frozen slow component $X \in \mathcal{H}_0$ and the frozen velocity field $U$ such that $||U - y||_{H^4} \leq \delta_0$.

Note that the drift portion of \eqref{frozen_system} is actually linear in the process $Z_t^{U,X,Y}$. Indeed, let $S_t^\nu(U, X)$ be the semi-group with generator $$L_t^\nu(X) \coloneqq \nu^{2/3}\Delta - \nu^{-1/3} U \partial_x + \nu^{-1/3}U'' \partial_x \Delta^{-1} - \nu^{\beta + 1/6} b_m(X,\cdot).$$ Then by adapting Lemmas \ref{linear_deterministic} and  \ref{nonlinear_deterministic} to the slow time scale (and simplifying the proof since $U$ is not time-dependent here), we have the following:
\begin{proposition}\label{semigroup_estimate}
    Suppose $X \in \mathcal{H}_{0}$ satisfies $\Ee_0(X) \leq  c_* \nu^{-2\beta}$ and $U$ satisfies $||U-y||_{H^4} \leq \delta_0$. Then for any $Y \in \mathcal{H}_{\neq}$,
    $$||S_t^\nu(U,X) Y||_{\mathcal{H}_{\neq}}^2 \leq \exp(-\delta_* t) || Y||_{\mathcal{H}_{\neq}}^2.$$
\end{proposition}
By the semigroup estimate and \eqref{noise_assumptions}, the SPDE \eqref{frozen_system} is well-posed. In fact, $Z_t^{U,X,Y}$ is an Ornstein-Uhlenbeck process in the case of fixed $X$ and $U$. Furthermore, we have the following estimate:
\begin{proposition}\label{bounds_on_frozen}
    For any $X \in \mathcal{H}_{0}$ satisfying $\Ee_0(X) < c_* \nu^{-2\beta}$ and $U$ satisfying $||U - y||_{H^4} \leq \delta_0$, we have
    \begin{equation}\label{basic_estimates_for_frozen}
        \begin{split}
            &\delta_*\int_0^t \nu^{2/3-\gamma} \D_{\neq}(Z_s^{U,X,Y}) 
 + \Ee_{\neq}(Z_s^{U,X,Y})ds  \leq \Ee_{\neq}(Y) + ||\Psi||^2 t + 2\int_0^t \langle \Psi d\bar{W}_s, Z_s^{U,X,Y}\rangle,\\
            &\Ee_{\neq}(Z_s^{U,X,Y}) \leq e^{-\delta_* t} \Ee_{\neq}(Y) + ||\Psi||^2(1-e^{-\delta_* t}) + 2 \int_0^t e^{-\delta_*(t-s)} \langle \Psi d\bar{W}_s, Z_s^{U,X,Y}\rangle.
        \end{split}
    \end{equation}
    Furthermore, for any $p \geq 1$,
    \begin{equation}\label{pth_moment_est}
        \begin{split}
            \E[\Ee_{\neq}^p(Z_t^{U,X,Y})] \leq e^{-\delta_* pt} \Ee_{\neq}(Y) + C(p)||\Psi||^{2p}.
        \end{split}
    \end{equation}
\end{proposition}
\begin{proof}
    By It\^o's formula applied to $Z_t^{U,X,Y} \mapsto\Ee_{\neq}(Z_t^{U,X,Y})$
    \begin{equation}\label{ito_frozen_eqn}
        \begin{split}
            \mathcal{E}_{\neq}(Z_t^{U,X,Y}) &= \Ee_{\neq}(Y)  + 2\int_0^t \mathrm{Re}\langle \partial_t Z_s^{U,X,Y}, Z_s^{U,X,Y} \rangle_H ds \\
            &\quad+ \nu^{-2/3 + \gamma} ||\Psi||^2 \int_0^t   ds + 2\nu^{-1/3+\gamma/2}\int_0^t \langle \Psi d\bar{W}_s, Z_x^{X,Y}\rangle.
        \end{split}
    \end{equation}
    Then the estimates \eqref{basic_estimates_for_frozen} follow in a similar manner as the proof of Proposition \ref{energy_bounds_for_original} via Lemmas \ref{linear_deterministic} and \ref{nonlinear_deterministic}. Note also that \eqref{pth_moment_est} in the case of $p=1$ is obtained by applying expectation to \eqref{basic_estimates_for_frozen}. We therefore concentrate on the case $p > 1$. From \eqref{ito_frozen_eqn}, we apply It\^o's formula to $\mathcal{E}_{\neq}(Z_t^{U,X,Y}) \mapsto \mathcal{E}_{\neq}^p(Z_t^{U,X,Y})$ and take expectation to reach
    \begin{equation}
        \begin{split}
            \E[\mathcal{E}_{\neq}^p(Z_t^{U,X,Y})] &= \Ee_{\neq}^p(Y)  + 2 p \int_0^t \E\left[\Ee^{p-1}(Z_s^{U,X,Y})\mathrm{Re}\langle \partial_t Z_t^{U,X,Y}, Z_t^{U,X,Y} \rangle_H ds\right]\\
            &\quad+ p \int_0^t\E[\Ee^{p-1}(Z_s^{U,X,Y})||\Psi||^2 ]ds+2p(p-1)\int_0^t \E[\Ee_{\neq}^{p-2}(Z_s^{U,X,Y}) \Ee_{\neq}(\Psi^* Z_s^{U,X,Y})] ds.
        \end{split}
    \end{equation}
    By Lemmas \ref{linear_deterministic}, \ref{nonlinear_deterministic}, and $\Ee_0(X) < c_* \nu^{-2\beta}$,
    \begin{equation}\label{frozen_decomp}
        \begin{split}
        \E[\mathcal{E}_{\neq}^p(Z_t^{U,X,Y})] +& 2p\delta_*\int_0^t \E[\Ee_{\neq}^{p}(Z_s^{U,X,Y})] + \E[\Ee_{\neq}^{p-1}(Z_s^{U,X,Y}) \D_{\neq}(Z_s^{U,X,Y})]ds \leq \Ee_{\neq}^p(Y)\\
        &\quad + p(2p-1) ||\Psi||^2 \int_0^t \E[\Ee_{\neq}^{p-1}(Z_s^{U,X,Y})] ds.
        \end{split}
    \end{equation}
    Note that by Young's product inequality,
    \begin{equation}\label{youngs_decomp}
         p(2p-1) ||\Psi||^2 \int_0^t \E[\Ee_{\neq}^{p-1}(Z_s^{U,X,Y})] ds \leq \frac{1}{p}(p(2p-1)\delta_*^{1-p})^{p}\int_0^t||\Psi||^{2p}ds +\delta_* \frac{p-1}{p}\int_0^t \E[\Ee_{\neq}^{p}(Z_s^{U,X,Y})] ds.
    \end{equation}
    Together, \eqref{frozen_decomp} and \eqref{youngs_decomp} imply
    \begin{equation}
        \begin{split}
            \E[\mathcal{E}_{\neq}^p(Z_t^{U,X,Y})] +& p\delta_*\int_0^t \E[\Ee_{\neq}^{p}(Z_s^{U,X,Y})] + \E[\Ee_{\neq}^{p-1}(Z_s^{U,X,Y}) \D_{\neq}(Z_s^{U,X,Y})]ds \leq \Ee_{\neq}^p(Y) + C(p) ||\Psi||^{2p} t,
        \end{split}
    \end{equation}
    which gives the remainder of proposition after an application of Gr\"onwall's inequality.
\end{proof}
 By Proposition \ref{bounds_on_frozen}, we see that $Z_t^{U,X,Y}$ is bounded in second moment,  and by Proposition \ref{semigroup_estimate}, the only invariant measure for $Z' = L_\nu^t(X)Z$ is $\delta_{0}$. Then by \eqref{noise_assumptions}, it is classical that there exists a unique invariant measure $\mu_\nu^{U,X}$ on $\mathcal{H}_{\neq}$ for $P_t^{U,X}$, the transition semi-group associated with \eqref{frozen_system} \cite{da1996ergodicity}. We also remark that $\mu_\nu^{U,X}$ is ergodic. Furthermore $\mu_\nu^{U,X}$ is Gaussian, i.e $\mu_\nu^{U,X} \sim \mathcal{N}(0,Q_\nu)$, where the covariance operator $Q_\nu$ acts via
\begin{equation}
    Q_\nu \varphi = \int_0^\infty S_t^\nu(U,X) \Psi \Psi^* S_t^{\nu, *}(U,X) \varphi dt.
\end{equation}
Next, we give a simple estimate of the integral of $\Ee_{\neq}(Y)$ over $\mathcal{H}_{\neq}$ with respect to the invariant measure.
\begin{proposition}\label{proposition: integral_against_inv_measure}
    Suppose that $X \in \mathcal{H}_{0}$ satisfies $\Ee_0(X) <  c_* \nu^{-2\beta}$ and $U$ satisfies $||U-y||_{H^4} \leq \delta_0$. Then for any $p \geq 1$ there exists a constant $C = C(p) > 0$ such that:
    \begin{equation}
        \int_{\mathcal{H}_{\neq}} \Ee_{\neq}^p(Y) \mu_\nu^{U,X}(dY) \leq C(p) ||\Psi||^{2p}.
    \end{equation}
    In the special case of $p = 1$, this constant is $1$.
\end{proposition}
\begin{proof}
    For any $t \geq 0$ we have by definition of $\mu_\nu^{U,X}$ as an invariant measure for the semigroup $P_t^{U,X}$, 
    \begin{equation}
        \begin{split}
            \int_{\mathcal{H}_{\neq}} \Ee_{\neq}^p(Y) \mu_\nu^{U,X} (dY) = \int_{\mathcal{H}_{\neq}} P_t^{U,X}\Ee_{\neq}^p(Y) \mu_\nu^{U,X} (dY) = \int_{\mathcal{H}_{\neq}} \E[\Ee_{\neq}^p(Z_t^{U,X,Y})] \mu_\nu^{U,X}(dY).
        \end{split}
    \end{equation}
    By Proposition \ref{bounds_on_frozen},
    \begin{equation}
        \begin{split}
            \int_{\mathcal{H}_{\neq}} \Ee_{\neq}(Y) \mu_\nu^{U,X} (dY) \leq \int_{\mathcal{H}_{\neq}} e^{-\delta_* t} \Ee_{\neq}(Y) + C(p)||\Psi||^{2p} \mu_\nu^{U, X}(dY).
        \end{split}
    \end{equation}
    Letting $t$ be such that $e^{-\delta_* t p} \leq 1/2$ yields the desired result. In the special case of $p =1$, we have the explicit formula $C(1) = 1 - e^{-\delta_* t}$ and so we only need to collect like terms.
\end{proof}
Averaging the slow system over $\mu_\nu^{U,X}$ is what will lead to the averaged system \eqref{averaged_syst}. We now establish bounds on the size of $\bar{b}_0(U,X)$.

\begin{proposition}\label{proposition: ergodic_to_bounded}
    Let $X \in \mathcal{H}_{0}$ satisfy $\Ee_0(X) <  c_* \nu^{-2\beta}$ and suppose $U$ satisfies $||U-y||_{H^4} \leq \delta_0$. Then we have the following boundedness property:
    \begin{equation}
        \begin{split}
            \nu^{-1/6}|| \bar{B}_0(U,X)||_{\mathcal{H}_{0}} \leq C_* \delta_*^{-1} ||\Psi||^2.
        \end{split}
    \end{equation}
\end{proposition}
\begin{proof}
    By comparing the definitions of $\mathcal{H}_{0}$ and $\mathcal{H}_{\neq}$ and applying the triangle inequality followed by the methods of the proof of Lemma \ref{nonlinear_deterministic}, it is not difficult to see that
    \begin{equation}
        \nu^{-1/6}||\bar{B}_0(U,X)||_{\mathcal{H}_{0}} \leq C_* \int_{\mathcal{H}_{\neq}} \Ee_{\neq}^{1/2}(Y)\nu^{1/3 - \gamma/2}\D_{\neq}^{1/2}(Y)  \mu_\nu^{U,X}(dY).
    \end{equation}
    A more careful analysis of the methods in \cite{bedrossian2023stability} reveals the tighter estimate
    \begin{equation}\label{tighter_bound_on_nonlin}
        \nu^{-1/6}||\bar{B}_0(U,X)||_{\mathcal{H}_{0}} \leq C_* \int_{\mathcal{H}_{\neq}} \Ee_{\neq}^{1/2}(Y)\left(\nu^{1/3 - \gamma/2}\D_{\mathfrak{t}}^{1/2}(Y) + \nu^{1/2-\gamma/2} \D_{\mathfrak{b}}(Y) \right) \mu_\nu^{U,X}(dY).
    \end{equation}
    See Appendix \ref{appendix: b} for a proof.
    The integrand on the right-hand side of \eqref{tighter_bound_on_nonlin} is in $L^1(\mathcal{H}_{\neq}, \mu_\nu^{U,X})$ by elliptic regularity and Proposition \ref{proposition: integral_against_inv_measure}. Then since $\mu_\nu^{U,X}$ is ergodic, we have $\Pb$-a.s
    \begin{equation}\label{b_ergodic_eqn}
        \begin{split}
           \int_{\mathcal{H}_{\neq}} \Ee_{\neq}^{1/2}(Y)&\left(\nu^{1/3 - \gamma/2}\D_{\mathfrak{t}}^{1/2}(Y) + \nu^{1/2-\gamma/2} \D_{\mathfrak{b}}^{1/2}(Y) \right)  \mu_\nu^{U,X}(dY) \\
           &= \lim_{t \to \infty} \frac{1}{t} \int_0^t \Ee_{\neq}^{1/2}(Z_s^{U,X,Y})\left(\nu^{1/3 - \gamma/2}\D_{\mathfrak{t}}^{1/2}(Z_s^{U,X,Y}) + \nu^{1/3-\gamma/2} \D_{\mathfrak{b}}^{1/2}(Z_s^{U,X,Y}) \right)   ds.
        \end{split}
    \end{equation}
     Note that the left-hand side of \eqref{tighter_bound_on_nonlin} is a fixed deterministic number for fixed $X$ and $U$. Hence we can apply expectation and get
    \begin{equation}\label{b_ergodic_eqn_2}
        \begin{split}
           \nu^{-1/6}||\bar{B}_0(U,X)||_{\mathcal{H}_{0}} \leq C_*\E\lim_{t \to \infty} \frac{1}{t} \int_0^t \Ee_{\neq}^{1/2}(Z_s^{U,X,Y}) \left(\nu^{1/3 - \gamma/2}\D_{\mathfrak{t}}^{1/2}(Z_s^{U,X,Y}) + \nu^{1/3-\gamma/2} \D_{\mathfrak{b}}^{1/2}(Z_s^{U,X,Y}) \right)  ds.
        \end{split}
    \end{equation}
    Then by Fatou's Lemma and Cauchy-Schwarz: 
    \begin{equation}
        \begin{split}
    \nu^{-1/6 }||\bar{B}_0(U,X)||_{\mathcal{H}_{0}} & \leq C_*\lim_{t \to \infty} \frac{1}{t}\left(\int_0^t \E[\Ee_{\neq}(Z_s^{U,X,Y})] ds\right)^{1/2}\nu^{1/3 - \gamma/2}\left(\int_0^t\E[\D_{\neq}(Z_s^{U,X,Y})]^{1/2}  ds\right)^{1/2}.
        \end{split}
    \end{equation}
    Applying Proposition \ref{bounds_on_frozen} gives
    \begin{equation}
    \begin{split}
                \nu^{-1/6}||\bar{B}_0(U,X)||_{\mathcal{H}_{0}} &\leq C_* \delta_*^{-1} \lim_{t \to \infty} \frac{1}{t} (\Ee_{\neq}(Y) + ||\Psi||^2 t)^{1/2}(\Ee_{\neq}(Y) + ||\Psi||^2 t)^{1/2}=C_*\delta_*^{-1} ||\Psi||^2.
    \end{split}
    \end{equation}
\end{proof}

Until now, we have focused on the case of fixed $X$ and $U$. It will also be necessary to estimate the difference between processes $Z_t^{U,X,Y}$ and $Z_t^{V, X',Y'}$. We remark that this frozen velocity $V$ is not to be confused with the Weiner process $\Phi dV_t$. The symbol $V$ is merely a convenient short-hand. In this case $X \neq X'$ and $U \neq V$, and so we cannot rely solely on the linear nature of \eqref{frozen_system}. Instead, we have the following.
\begin{proposition}\label{second_bounds_on_frozen}
    Let $X$ and $X'$ in $\mathcal{H}_{0}$ satisfy $\Ee_0(X) < c_* \nu^{-2\beta}$ and $\Ee_0(X') < c_*\nu^{-2\beta}$ and let $U$ and $V$ satisfy $||U-y||_{H^4} \leq \delta_0$ and $||V-y||_{H^4} \leq \delta_0$. Let $Y, Y' \in \mathcal{H}_{\neq}$. Then there exists a constant $C = C(m) > 0$ such that:
    \begin{equation}
        \begin{split}
            \E\int_0^t& \nu^{2/3 - \gamma} \D_{\neq}(Z_s^{U,X,Y} - Z_s^{V,X',Y'})  + \Ee_{\neq}(Z_s^{U,X,Y} - Z_s^{V,X',Y'})ds \leq C\Ee_{\neq}(Y-Y')\\
            &\quad+ C\biggl(\nu^{2\beta}\Ee_0(X-X') + \nu^{-1}||U - V||_{L^\infty}^2 + \nu^{-1/3}||U - V||_{C^2}^2\biggr)\biggl(\Ee_{\neq}(Y) + \Ee_{\neq}(Y') + ||\Psi||^2 t\biggr),
        \end{split}
    \end{equation}
    \begin{equation}
        \begin{split}
        &\E(\Ee_{\neq}(Z_t^{U,X,Y} -Z_t^{V,X',Y'})) \leq e^{-\delta_* t}\Ee_{\neq}(Y-Y')\\
        &\quad\quad+  C\biggl(\nu^{2\beta}\Ee_0(X-X') + \nu^{-1}||U - V||_{L^\infty}^2 + \nu^{-1/3}||U - V||_{C^2}^2\biggr)\biggl( \Ee_{\neq}(Y) + \Ee_{\neq}(Y') + ||\Psi||^2\biggr).
        \end{split}
    \end{equation}
\end{proposition}
\begin{proof}
We define the symbol $\rho_t \coloneqq Z_t^{U,X,Y} - Z_t^{V,X',Y'}$. By Ito's formula and applying expectation,
\begin{equation}
    \begin{split}
        \E (\Ee_{\neq}(\rho_t)) &= \Ee_{\neq}(Y-Y') + 2  \E\int_0^t \mathrm{Re}\langle \partial_s \rho_s, \rho_s \rangle ds\\
        &= \Ee_{\neq}(Y-Y') + 2  \E\int_0^t \mathrm{Re}\langle \nu^{2/3} \Delta \rho_s, \rho_s \rangle ds - 2  \E\int_0^t \mathrm{Re}\langle \nu^{-1/3}\partial_x (U Z_s^{U,X,Y} -V Z_s^{V,X',Y'}), \rho_s \rangle ds\\
        &\quad +  2  \E\int_0^t \mathrm{Re}\langle \nu^{-1/3} (U'' \partial_x \Delta^{-1} Z_s^{U,X,Y} -V'' \partial_x\Delta^{-1} Z_s^{V,X',Y'}), \rho_s \rangle ds\\
        &\quad - 2 \nu^{\beta + 1/6} \E \int_0^t \mathrm{Re} \langle b_m(X, Z_s^{U,X,Y}) - b_m(X', Z_s^{V,X',Y'}), \rho_s \rangle ds.
    \end{split}
\end{equation}
We decompose the terms involving $U$ and $V$ as
\begin{equation}
    \begin{split}
        \partial_x (U Z_t^{U,X,Y} -V Z_t^{V,X',Y'}) &= (U - V) \partial_x Z_t^{U,X,Y} + V \partial_x \rho_t,
    \end{split}
\end{equation}
and
\begin{equation}
    \begin{split}
        U'' \partial_x \Delta^{-1} Z_t^{U,X,Y} -V'' \partial_x\Delta^{-1} Z_t^{V,X',Y'} &= (U'' - V'') \partial_x \Delta^{-1} Z_t^{U,X,Y} + V'' \partial_x \Delta^{-1} \rho_t.
    \end{split}
\end{equation}
We additionally decompose $b_m(X,Z_t^{U,X,Y}) - b_m(X',Z_t^{V,X',Y'})$:
\begin{equation}\label{decomp_of_b_m}
\begin{split}
    b_m(X,Z_t^{U,X,Y}) - b_m(X',Z_t^{V,X',Y'}) &= (\partial_y \Delta^{-1} X) (\partial_x Z_t^{U,X,Y}) - (\partial_x \Delta^{-1} Z_t^{U,X,Y})(\partial_y X)\\
    & \quad - (\partial_y \Delta^{-1} X') (\partial_x Z_t^{V,X',Y'}) + (\partial_x \Delta^{-1} Z_t^{V,X',Y'})(\partial_y X')\\
    &= (\partial_y \Delta^{-1} (X - X')) \partial_x Z_t^{U,X,Y} + \partial_x(Z_t^{U,X,Y} - Z_t^{V,X',Y'}) \partial_y \Delta^{-1} X'\\
    &\quad - \partial_y  (X - X') \partial_x \Delta^{-1} Z_t^{U,X,Y} - \partial_x \Delta^{-1}(Z_t^{U,X,Y} - Z_t^{V,X',Y'}) \partial_y  X'.
\end{split}
\end{equation}
    Then by Lemmas \ref{linear_deterministic} and \ref{nonlinear_deterministic}, together with the previous decompositions,
    \begin{equation}\label{eqn: nonlin_terms_frozen}
    \begin{split}
        \E (\Ee_{\neq}(\rho_t)) + &8 \delta_* \int_0^t \nu^{2/3-\gamma} \D_{\neq}(\rho_s) + \Ee_{\neq}(\rho_s)ds \leq \Ee_{\neq}(Y-Y') + 2 \nu^{-1/3} \E\int_0^t \mathrm{Re}\langle (U - V) \partial_x Z_s^{U,X,Y}, \rho_s \rangle ds\\
        &\quad+2 \nu^{-1/3} \E\int_0^t \mathrm{Re}\langle (U'' - V'') \partial_x \Delta^{-1} Z_s^{U,X,Y}, \rho_s \rangle ds\\
        &\quad+  C_* \nu^{\beta + 2/3-\gamma}\E\int_0^t \Ee_0^{1/2}(X-X') \D_{\neq}^{1/2}(Z_s^{U,X,Y} - Z_s^{V,X',Y'})\\
        &\quad\quad\quad\quad\quad\quad\quad\quad\quad \times(\D_{\neq}^{1/2}(Z_s^{U,X,Y}) + \D_{\neq}^{1/2} (Z_s^{V,X',Y'}))ds\\
            &\quad + C_* \nu^{\beta + 2/3-\gamma} \E\int_0^t (\Ee_0^{1/2}(X) + \Ee_0^{1/2}(X'))\D_{\neq}(Z_s^{U,X,Y} - Z_s^{V,X',Y'})  ds.
    \end{split}
\end{equation}
We treat the $U$ and $V$ terms first. By Lemma \ref{lemma: difference_of_U}, the definition of $\langle \cdot, \cdot \rangle_{\mathcal{H}_{\neq}}$, and Young's inequality:
\begin{equation}\label{no_derivs_on_U}
    \begin{split}
        \nu^{-1/3} \E\int_0^t& \mathrm{Re}\langle (U - V) \partial_x Z_s^{U,X,Y}, \rho_s \rangle ds \\
        &\leq C \nu^{-1/2} ||U -V||_{L^\infty} \left(\E \int_0^t \nu^{2/3-\gamma} \D_{\neq}(Z_s^{U,X,Y}) ds \right)^{1/2}\left(\E \int_0^t \nu^{2/3-\gamma}\D_{\neq}(\rho_s) ds \right)^{1/2}\\
        &\leq C \nu^{-1}||U -V||_{L^\infty}^2 \left(\E \int_0^t \nu^{2/3-\gamma} \D_{\neq}(Z_s^{U,X,Y}) ds \right) + \delta_*\left(\E \int_0^t \nu^{2/3-\gamma}\D_{\neq}(\rho_s) ds \right).
    \end{split}
\end{equation}
    The $U''$ and $V''$ terms are controlled in a similar manner via Lemma \ref{lemma: difference_of_U_derivs}
\begin{equation}\label{two_derivs_on_U}
    \begin{split}
         \nu^{-1/3} \E\int_0^t &\mathrm{Re}\langle (U'' - V'') \partial_x \Delta^{-1} Z_s^{U,X,Y}, \rho_s \rangle ds \\
         &\leq  C \nu^{-1/3}||U -V||_{C^2}^2 \left(\E \int_0^t \nu^{2/3-\gamma} \D_{\neq}(Z_s^{U,X,Y}) ds \right) + \delta_*\left(\E \int_0^t \nu^{2/3-\gamma}\D_{\neq}(\rho_s) ds \right).
    \end{split}
\end{equation}
Thus we have by Proposition \ref{bounds_on_frozen}, \eqref{no_derivs_on_U}, and \eqref{two_derivs_on_U} applied to \eqref{eqn: nonlin_terms_frozen}:
\begin{equation}
    \begin{split}
    \E (\Ee_{\neq}(\rho_t)) + &6 \delta_* \int_0^t \nu^{2/3-\gamma} \D_{\neq}(\rho_s) + \Ee_{\neq}(\rho_s)ds \leq  \Ee_{\neq}(Y-Y')\\ &+ C\biggl(\nu^{-1}||U - V||_{L^\infty}^2 + \nu^{-1/3}||U- V||_{C^2}^2\biggr)\left(\Ee_{\neq}(Y) + \Ee_{\neq}(Y') + ||\Psi||^2 t\right)\\
    &+  C_* \nu^{\beta + 2/3-\gamma}\E\int_0^t \Ee_0^{1/2}(X-X') \D_{\neq}^{1/2}(Z_s^{U,X,Y} - Z_s^{V,X',Y'}) (\D_{\neq}^{1/2}(Z_s^{U,X,Y}) + \D_{\neq}^{1/2} (Z_s^{V,X',Y'}))ds\\
            &+ C_* \nu^{\beta + 2/3-\gamma} \E\int_0^t (\Ee_0^{1/2}(X) + \Ee_0^{1/2}(X'))\D_{\neq}(Z_s^{U,X,Y} - Z_s^{V,X',Y'})  ds.
    \end{split}
\end{equation}
    We then apply the smallness assumption on $X$ and $X'$, together with Young's product inequality, and we find
    \begin{equation}
        \begin{split}
            \E(\Ee_{\neq}&(Z_t^{U,X,Y} - Z_t^{V,X',Y'})) + 4\delta_*\E\int_0^t \nu^{2/3 - \gamma} \D_{\neq}(Z_s^{U,X,Y} - Z_s^{V,X',Y'})  + \Ee_{\neq}(Z_s^{U,X,Y} - Z_s^{V,X',Y'})ds\\
            &\leq \Ee_{\neq}(Y-Y')+ C \nu^{2\beta}\Ee_0(X-X')\int_0^t  \nu^{2/3-\gamma}\E(\D_{\neq}(Z_s^{U,X,Y}) + \D_{\neq}(Z_s^{V,X',Y'}))ds\\
            &\quad+ C\biggl(\nu^{-1}||U - V||_{L^\infty}^2 +  \nu^{-1/3}||U- V||_{C^3}^2\biggr)\left(\Ee_{\neq}(Y) + \Ee_{\neq}(Y') + ||\Psi||^2 t\right).
        \end{split}
    \end{equation}
    By Proposition \ref{bounds_on_frozen},
    \begin{equation}
        \begin{split}
            \E(\Ee_{\neq}&(Z_t^{U,X,Y} - Z_t^{V,X',Y'})) + 4\delta_*\E\int_0^t \nu^{2/3 - \gamma} \D_{\neq}(Z_s^{U,X,Y} - Z_s^{V,X',Y'})  + \Ee_{\neq}(Z_s^{U,X,Y} - Z_s^{V,X',Y'})ds\\
            &\leq \Ee_{\neq}(Y-Y') + C \biggl(\nu^{2\beta}\Ee_0(X-X') + \nu^{-1}||U - V||_{L^\infty}^2 + \nu^{-1/3}||U - V||_{C^3}^2\biggr)\\
            &\quad\quad\quad\quad\quad\quad\quad\quad\quad\times( \Ee_{\neq}(Y) + \Ee_{\neq}(Y') + ||\Psi||^2 t).
        \end{split}
    \end{equation}
    Gr\"onwall's inequality completes the proof.
\end{proof}
Having established a counterpart to Proposition \ref{bounds_on_frozen} in the case of distinct $X$ and $U$, we now do the same for Proposition \ref{proposition: ergodic_to_bounded}. One could make more general estimates on the difference between $\mu_\nu^{U,X}$ and $\mu_\nu^{V, X'}$ in various norms. However, the following will be sufficient for our purposes.
\begin{proposition}\label{proposition: ergodic_to_Lipschitz}
    Let $X$ and $X'$ in $\mathcal{H}_{0}$ satisfy $\Ee_0(X) < c_* \nu^{-2\beta}$ and $\Ee_0(X') < c_*\nu^{-2\beta}$ and let $U$ and $V$ satisfy $||U-y||_{H^4} \leq \delta_0$ and $||V-y||_{H^4} \leq \delta_0$. Then there exists a constant $C =C(m)>0$ such that
    \begin{equation}
        \nu^{-1/6}||\bar{B}_0(U, X) - \bar{B}_0(V,X')||_{\mathcal{H}_{0}} \leq C\biggl(\nu^{2\beta}\Ee_0(X-X') + \nu^{-1}||U - V||_{L^\infty}^2 +   \nu^{-1/3}||U- V||_{C^2}^2\biggr)^{1/2}||\Psi||^2.
    \end{equation}
\end{proposition}
\begin{proof}
    By an ergodic argument similar to Proposition \ref{proposition: ergodic_to_bounded} and a decomposition of $B_0$ similar to \eqref{decomp_of_b_m},
    \begin{equation}
        \begin{split}
            \nu^{-1/6 }&||\bar{B}_0(U, X)-\bar{B}_0(V, X')||_{\mathcal{H}_{0}}\\
            & \leq C_* \lim_{t \to 0} \frac{1}{t} \int_0^t \E[\Ee_{\neq}(Z_s^{U,X,Y} -(Z_s^{V,X',Y})]^{1/2} \nu^{1/3-\gamma/2}\E[\D_{\neq}(Z_s^{U,X,Y})]^{1/2} ds\\
            &\quad \quad \quad\quad\quad\quad +\E[\Ee_{\neq}(Z_s^{V,X',Y})]^{1/2}\nu^{1/3-\gamma/2}\E[\D_{\neq}(Z_s^{U,X,Y}-Z_s^{V,X',Y})]^{1/2}\\
            &\leq C_* \lim_{t \to 0} \frac{1}{t}\left(\int_0^t \E[\Ee_{\neq}(Z_s^{U,X,Y}-Z_s^{V,X',Y})]ds\right)^{1/2}\left(\int_0^t \nu^{2/3-\gamma} \E[\D_{\neq}(Z_s^{U,X,Y})] ds\right)^{1/2}\\
            &\quad + C_* \lim_{t \to 0} \frac{1}{t}\left(\int_0^t\nu^{2/3-\gamma} \E[\D_{\neq}(Z_s^{U,X,Y}-Z_s^{V,X',Y})]ds\right)^{1/2}\left(\int_0^t \E[\Ee_{\neq}(Z_s^{V,X',Y})] ds\right)^{1/2}.
        \end{split}
    \end{equation}
    Then we apply Propositions \ref{bounds_on_frozen} and \ref{second_bounds_on_frozen} to find
    \begin{equation}
        \begin{split}
            \nu^{-1/6 }&||\bar{B}_0(U,X)-\bar{B}_0(V,X')||_{\mathcal{H}_{0}}\\
            &\leq C \lim_{t \to \infty}\frac{1}{t}\biggl(\nu^{2\beta}\Ee_0(X-X') + \nu^{-1/3}||U - V||_{L^\infty} + ||U - V||_{C^3}^2\biggr)^{1/2}\biggl( \Ee_{\neq}(Y) + \Ee_{\neq}(Y') + ||\Psi||^2 t\biggr)^{1/2}\\
            &\quad\quad\quad\quad \times \left(\left(\Ee_{\neq}(Y) + ||\Psi||^2t\right)^{1/2} + \left(\Ee_{\neq}(Y') + ||\Psi||^2t\right)^{1/2}\right)\\
            &\leq C\biggl(\nu^{2\beta}\Ee_0(X-X') + \nu^{-1}||U - V||_{L^\infty}^2 +  \nu^{-1/3}||U- V||_{C^2}\biggr)^{1/2}||\Psi||^2.
        \end{split}
    \end{equation}
    as desired.
\end{proof}

\section{Analysis of the Averaged Equation}\label{averaged_section}

Our final discussion before we begin the proof of Theorem \ref{main_theorem} will center on the averaged process $\bar{X}_t$. We begin by showing that \eqref{averaged_syst} is well-posed for short time. We will then show that this can be extended for long times.
\begin{theorem}\label{theorem: averaged_well_posed}
    Fix $\nu \in (0,1)$ and let $T$ be such that $T < \min\{\frac{c_* \delta_*^{2}}{16 C_*^2} ||\Psi||^{-4} \nu^{-2\beta}, \frac{\delta_*^4}{4C_*^4} ||\Psi||^{-4} \nu^{-2\beta}\}$. Fix $X_0^\nu \in \mathcal{H}_0$ and suppose that $\Ee_0(X_0^\nu) < \frac{1}{2}c_*\nu^{-2\beta}$. Then the averaged system \eqref{averaged_syst} has a unique solution $\bar{X}_t^\nu$ in $C([0,T\wedge \sigma];\mathcal{H}_0)$ with initial data $X_0^\nu$.
\end{theorem}
\begin{proof}
    Define the complete metric space $\mathcal{G} = \{f \in C([0,T\wedge \sigma] ; \mathcal{H}_0) : f(0) = X_0^\nu, ||f||_{L^\infty([0,T\wedge \sigma]),\mathcal{H}_0} \leq  c_*^{1/2}/\nu^{\beta} \}$. Define the mapping $F : \mathcal{G}  \to \mathcal{G} $ via
    $$f \mapsto e^{\nu^{\gamma}t \partial_y^2}f(0)+ \nu^{\gamma/2 - 1/6}\int_0^t  e^{\nu^{\gamma}(t-s) \partial_y^2} \bar{b}_0(U_s,f(s)) ds.$$
        First, we show that $F$ is well-defined. Clearly, $F(f(0)) = f(0) = X_0^\nu$. Also, it easy to check that $\bar{b}_0$ vanishes on the boundary of $[-1,1]$, so it does lie in the domain of the heat operator. Next, for $t \in [0,T\wedge \sigma]$, $F(f)(t)$ satisfies
    \begin{equation}
        \begin{split}
            ||F(f)(t)||_{\mathcal{H}_0} & \leq ||e^{\nu^\gamma t \partial_y^2}f(0)||_{\mathcal{H}_0} + \nu^{\gamma/2 - 1/6}\int_0^t || e^{\nu^\gamma (t-s) \partial_y^2} \bar{b}_0(U_s,f(s)) ds||_{\mathcal{H}_0}.
        \end{split}
    \end{equation}
    Then by the property \eqref{heat_propogator_2} of the heat propagator and Proposition \ref{proposition: ergodic_to_bounded},
    \begin{equation}
        \begin{split}
            ||F(f)(t)||_{\mathcal{H}_0} & \leq ||f(0)||_{\mathcal{H}_0} + \nu^{ - 1/6}\int_0^t (t-s)^{-1/2} || \bar{B}_0(U_s,f(s))||_{\mathcal{H}_0} ds\\
            &\leq  ||f(0)||_{\mathcal{H}_0} + 2 C_* \delta_*^{-1} T^{1/2} ||\Psi||^2.
        \end{split}
    \end{equation}
    Thus $\sup_{t \in [0,T\wedge\sigma]} ||F(f)(t)||_{\mathcal{H}_0} <c_*^{1/2}/\nu^{\beta}$ so long as $T < \frac{c_* \delta_*^{2}}{16 C_*^2} ||\Psi||^{-4} \nu^{-2 \beta}$. Our next step is to address the continuity of $F(f)(t)$. Let $r,t \in [0,T\wedge \sigma]$, and suppose without loss of generality that $r < t$. Then by \eqref{heat_propogator_2}, and Proposition \ref{proposition: ergodic_to_bounded},
    \begin{equation}
        \begin{split}
            ||F(f)(t) - F(f)(r)||_{\mathcal{H}_0} &
            \leq ||(e^{(t-r)\nu^{\gamma}\partial_y^2}-I)F(f)(r)||_{\Hh_0}+\nu^{\gamma/2-1/6} ||\int_r^t e^{\nu^\gamma(t-s)\partial_y^2} \bar{b}_0(U_s,f(s)) ds||_{\Hh_0}\\
            &\leq ||(e^{(t-r)\nu^{\gamma}\partial_y^2}-I)F(f)(r)||_{\Hh_0}+\nu^{-1/6} C\int_r^t (t-s)^{-1/2} ||\bar{B}_0(U_s,f(s))||_{\Hh_0} ds\\
            &\leq ||(e^{(t-r)\nu^{\gamma}\partial_y^2}-I)F(f)(r)||_{\Hh_0} +C ||\Psi||^2(t-r)^{1/2}.
        \end{split}
    \end{equation}
    Now $||(e^{(t-r)\nu^{\gamma}\partial_y^2}-I)F(f)(r)||_{\Hh_0} \to 0$ as $t\to r$, with $r$ fixed, giving continuity. For quantitative continuity estimates, see Proposition \ref{prop: avg_holder_cont}. We have now establish that $F$ does map $\mathcal{G}$ into $\mathcal{G}$. Next, we show that $F$ is a contraction. Let $f, g \in \mathcal{G}$. By Proposition \ref{proposition: ergodic_to_Lipschitz},
    \begin{equation}
        \begin{split}
            ||F(f)(t) - F(g)(t)||_{\mathcal{H}_0} & \leq \nu^{\gamma/2 - 1/6} \int_0^t ||e^{\nu^{\gamma}(t-s)\partial_y^2} (\bar{b}_0(U_s,f(s)) - \bar{b}_0(U_s,g(s)))||_{\Hh_0}ds\\
            &\leq \nu^{-1/6} \frac{1}{2} \int_0^t (t-s)^{-1/2}||\bar{B}_0(U_s,f(s)) - \bar{B}_0(U_s,g(s))||_{\Hh_0}ds\\
            &\leq \frac{C_*^2}{2 \delta_*^2}\nu^{\beta}||\Psi||^2 T^{1/2}\sup_{s \in [0,T]}||f(s) - g(s)||_{\mathcal{H}_0} .
        \end{split}
    \end{equation}
    Thus $F$ is a contraction so long as $T < \frac{\delta_*^4}{4C_*^4} ||\Psi||^{-4} \nu^{-2\beta}.$ By the contraction mapping principle on $\mathcal{G} $, we conclude the theorem.
\end{proof}

We now show long-time existence of the averaged process.
\begin{proposition}\label{prop: long_time_well_posed_ness}
     Let $T > 0$ be given and let $\nu \in (0,\nu_*)$ be fixed. Fix $X_0^\nu \in \mathcal{H}_0$. There exists a constant $C_0 = C_0(||\Psi||)$ such that if $\Ee_0(X_0^\nu) < C_0\nu^{-\alpha'}$, then the averaged system \eqref{averaged_syst} has a unique solution $\bar{X}_t^\nu$ in $C([0,T\wedge \sigma];\mathcal{H}_0)$ with initial data $X_0^\nu$. Furthermore, this solution satisfies
     \begin{equation}
         \mathcal{E}(\bar{X}_t^\nu) \leq \frac{1}{2} c_* \nu^{-2\alpha'}
     \end{equation}
     for all times $t \in [0,T\wedge \sigma]$.
\end{proposition}
\begin{proof}
    By Theorem \ref{theorem: averaged_well_posed}, $\bar{X}_t^\nu$ exists for a potentially short time $\tilde{T} \wedge \sigma \leq T \wedge \sigma$. Furthermore, we observe that so long as $\mathcal{E}_0(\bar{X}_t^\nu) < \frac{1}{2}c_* \nu^{-2\beta}$, we may extend the solution until times $(t+\tilde{T}) \wedge \sigma$. Hence we must simply show the existence of $C_0$ such that if $\Ee_0(X_0^\nu) < C_0\nu^{-\alpha'}$, then $\Ee_0(X_t^\nu) < \frac{1}{2}c_*\nu^{-2\alpha'} < \frac{1}{2}c_* \nu^{-2\beta}$ for all times of existence.
    
    Instead of working with the mild form as in Theorem \ref{theorem: averaged_well_posed}, we apply the chain rule to $\mathcal{E}_0(\bar{X}_t^\nu)$. Then by integration by parts and Cauchy-Schwarz, we have for $t \leq \sigma$ such that $\bar{X}_t^\nu$ exists,
    \begin{equation}
        \mathcal{E}_0(\bar{X}_t^\nu) + \frac{1}{2}\int_0^t \mathcal{D}_0(\bar{X}_s^\nu) ds \leq \mathcal{E}_0(\bar{X}_0^\nu) + C \nu^{-1/3} \int_0^t || \bar{B}_0(U_s, \bar{X}_s)||_{\mathcal{H}_0} ^2 ds,
    \end{equation}
    where $C$ depends only on the regularity parameter $m$. Then by Proposition \ref{proposition: ergodic_to_bounded}, we have
    \begin{equation}
         \mathcal{E}_0(\bar{X}_t^\nu) + \frac{1}{2}\int_0^t \mathcal{D}_0(\bar{X}_s^\nu) ds \leq \mathcal{E}_0(\bar{X}_0^\nu) + C \times ||\Psi||^4 t.
    \end{equation}
    Then by the Poincar\'e inequality, $\frac{1}{2}\int_0^t \mathcal{D}_0(\bar{X}_s^\nu) ds \geq 4 \delta_* \int_0^t \mathcal{E}(\bar{X}_s^\nu) ds$. Hence by Gr\"onwall's inequality,
    \begin{equation}
         \mathcal{E}_0(\bar{X}_t^\nu) \leq e^{-\delta_*t}\mathcal{E}_0(\bar{X}_t^\nu) + C \times ||\Psi||^4,
    \end{equation}
    for some constant $C$ depending only on $m$. Let us take $C_0$ to be the constant $C_0$ from Proposition \ref{energy_bounds_for_original}. Then so long as
    \begin{equation}\label{eqn: well_posed_condition}
        C ||\Psi||^4 < \frac{c_*}{4} \nu^{-2\alpha'},
    \end{equation}
    we obtain
    \begin{equation}
\mathcal{E}_0(\bar{X}_t^\nu) \leq \frac{c_*}{2}\nu^{-2\alpha'}
    \end{equation}
    as desired. Note that \eqref{eqn: well_posed_condition} holds by the definition of $\nu_*$.
\end{proof}
We remark that by the choices of $C_0$ and $\nu_*$, if $\Ee_0(X_0^\nu) \leq C_0\nu^{-2\alpha'}$, then the proof of Theorem \ref{theorem: averaged_well_posed} and Proposition \ref{prop: long_time_well_posed_ness} yields $\Ee_0(\bar{X}_t^\nu) \leq \frac{1}{2}c_* \nu^{-2\alpha'}$ for $t < \sigma$ and $\nu < \nu_*$.  Recall that as part of the proof of Theorem \ref{theorem: averaged_well_posed}, we have shown $$||\nu^{\gamma/2 - 1/6} \int_0^t e^{\nu^{\gamma}(t-s)\partial_y^2}\bar{b}_0(U_s,\bar{X}_s^\nu) ds||_{\mathcal{H}_0} \lesssim T^{1/2}||\Psi||^2.$$ 
This agrees with the $O(T^{1/2})$ growth suggested by the heuristics in Subsection \ref{subsection: theorems}. Recall that this heuristic was obtained by considering $\mu_\nu^{U,X}$ in the special case of $U = y$, $X = 0$, and by ignoring the effects of the boundary, so it is not a proof of optimality, but it is strongly suggestive. Thus, even when $\gamma >0$ and for $\nu$ very small, we still expect the averaged system to have nontrivial behavior.

Note that the proof of Theorem \ref{main_theorem} will remain valid if $2\alpha - \beta+ \gamma/2 > 1/3$, rather than $2 \alpha - \beta + \gamma/2 = 1/3$, which we have considered up until now \eqref{scaling_constraint}. A key difference however, is that the size of the nonlinear term in $\bar{X}_t^\nu$ is now bounded by $C\nu^{2\alpha - \beta + \gamma/2 - 1/3}||\Psi||^2$, which vanishes as $\nu \to 0$ if $2\alpha - \beta + \gamma/2 > 1/3$. So in this regime, we would require $\gamma = 0$ in order to have meaningful evolution of $\bar{X}_t^\nu$. We therefore would simply recover $2\alpha - \beta > 1/3$, giving a version of the main theorem with strictly weaker noise and without the noise strength v.s. time scale trade-off.
Lastly, we give a H\"older continuity estimate for the averaged process $\bar{X}_t^\nu$.
\begin{proposition}\label{prop: avg_holder_cont}
    Suppose that $\Ee_0(X_0^\nu) < C_0\nu^{-\alpha'}$, $\nu^{a/3}\Ee_0(|\partial_y|^a X_0^\nu) < C_0\nu^{-\alpha'}$, and $\gamma < 1/3$. Then there exists a constant $C = C(||\Psi||, T)>0$ such that
    \begin{equation}\label{eqn: avg_holder_in_space}
        \sup_{t \in [0,T \wedge \sigma]}|||\partial_y|^a \bar{X}_t^\nu||_{\Hh_0}^2 \leq  C\nu^{-a/3 - \alpha'},
    \end{equation}\label{eqn: avg_holder_in_time}
    \begin{equation}
        \sup_{t \in [0,T \wedge \sigma]}||\bar{X}_t^\nu - \bar{X}_{t(\delta)}^\nu||_{\Hh_0}^2 \leq C\delta^{a/2}\nu^{\gamma/2-a/6 - \alpha'/2}.
    \end{equation}
\end{proposition}
\begin{proof}
    We start with \eqref{eqn: avg_holder_in_space}. By \eqref{heat_propogator_2} and Proposition \ref{proposition: ergodic_to_bounded} we have for $t < \sigma$,
    \begin{equation}
        \begin{split}
            || |\partial_y|^a \bar{X}_t^\nu||_{\Hh_0} & \leq |||\partial_y|^ae^{t \nu^{\gamma}\partial_y^2} \bar{X}_0^\nu||_{\Hh_0} + \nu^{\gamma/2-1/6}\int_0^t ||  |\partial_y|^a e^{(t-s)\nu^{\gamma}\partial_y^2}\bar{b}_0(U_s^\nu,\bar{X}_s^\nu)||_{\Hh_0} ds\\
            &\leq || |\partial_y|^a \bar{X}_0^\nu||_{\Hh_0} + \nu^{-1/6 - a\gamma/2} C \int_0^t (t-s)^{-1/2 - a/2}||\bar{B}_0(U_s^\nu,\bar{X}_s^\nu)||_{\Hh_0} ds\\
            &\leq |||\partial_y|^a X_0^\nu||_{\Hh_0} + \nu^{-a\gamma/2 } ||\Psi||^2 C(T).
        \end{split}
    \end{equation}
We remark just as in the case of $b_0$ and $B_0$ in Proposition \ref{proposition: holder_in_space}, both $\bar{b}_0$ and $\bar{B}_0$ vanish on the boundary of $[-1,1]$, and so it is valid to use \eqref{heat_propogator_2}. For \eqref{eqn: avg_holder_in_time}, we use $t < \sigma$, \eqref{heat_propogator} \eqref{heat_propogator_2}, Proposition \ref{proposition: ergodic_to_bounded}, \eqref{eqn: avg_holder_in_space} and the assumption $|||\partial_y|^a X_0^\nu||_{\Hh_0} \leq C_0 \nu^{-a/6-\alpha'}$:
    \begin{equation}
        \begin{split}
            ||\bar{X}_t^\nu - \bar{X}_{t(\delta)}^\nu||_{\Hh_0} & \leq ||(e^{(t-t(\delta))\nu^{\gamma} \partial_y^2} -I) \bar{X}_{t(\delta)}^\nu||_{\Hh_0} + \nu^{\gamma/2-1/6} \int_{t(\delta)}^t ||e^{(t-s)\nu^{\gamma} \partial_y^2} \bar{b}_0(U_s^\nu,\bar{X}_s^\nu)||_{\Hh_0} ds\\
            &\leq C \delta^{a/2} \nu^{\gamma/2}|||\partial_y|^a \bar{X}_{t(\delta)}^\nu||_{\Hh_0} + C \nu^{-1/6} \int_{t(\delta)}^t (t-s)^{-1/2} ||\bar{B}_0(U_s^\nu,\bar{X}_s^\nu)||_{\Hh_0}ds\\
            &\leq C\delta^{a/2}\nu^{\gamma/2-a/6 - \alpha'/2} + C\delta^{1/2} ||\Psi||^2.
        \end{split}
    \end{equation}
\end{proof}

\section{Proof of the Main Theorem}\label{section_of_propositions}

This section contains the proof of Theorem \ref{main_theorem}. Recall that we have fixed $\gamma = 0$ and $\alpha' = \frac{\theta}{36}\min\{ \frac{a}{1+2a} \frac{1}{6}, \beta, \alpha\}$. We start with the proofs of Propositions \ref{approx_with_lin}, \ref{lin_to_aux_prop}, and \ref{key_prop}.

\subsection{Proof of Key Propositions}
We begin with the proof of Proposition \ref{approx_with_lin}, which gives the approximation of the original process $X_t^\nu$ by the pseudo-linearized process $\tilde{X}_t^\nu$.
\begin{proof}[Proof of Proposition \ref{approx_with_lin}]
Let $C_0  > 0$ be as in Propositions \ref{energy_bounds_for_original}, \ref{energy_bounds_for_linear}, and \ref{energy_bounds_for_auxiliary}. Let $\nu_*$ be as in \eqref{def_of_nu_star}. By definition of $b_0$ \eqref{def_of_b_nonlins}, we have
    \begin{equation}\label{split_nonlin}
        \begin{split}
            b_0(Y_t^\nu) - b_0(\tilde{Y_t^\nu}) &= -\partial_y \int_{\T} \partial_x \Delta^{-1}(Y_t^\nu - \tilde{Y}_t^\nu)(x) Y_t^\nu(x) + \partial_x \Delta^{-1} \tilde{Y}_t^\nu(x)(Y_t^\nu - \tilde{Y}_t^{\nu})(x) dx.
        \end{split}
    \end{equation}
    Then by applying the decomposition \eqref{split_nonlin} and the bounds of \eqref{nonlinear_deterministic} to $\Ee_0(X_t^\nu - \tilde{X}_t^\nu)$, we find
    \begin{equation}\label{slow_diff_1}
        \begin{split}
            \Ee_0(X_t^\nu - \tilde{X}_t^\nu) + 8\delta_* \int_0^t \D_0(X_s^\nu - \tilde{X}_s^\nu)ds & \leq \nu^{1/3}C_*\int_0^t \Ee_{\neq}^{1/2}(Y_s^\nu) \D_{\neq}^{1/2}(Y_s^\nu - \tilde{Y}_s^\nu) \D_0^{1/2}(X_s^\nu - \tilde{X}_s^\nu) ds\\
            &\quad+\nu^{1/3}C_*\int_0^t \Ee_{\neq}^{1/2}(Y_s^\nu - \tilde{Y}_s^\nu) \D_{\neq}^{1/2}(\tilde{Y}_s^\nu) \D_0^{1/2}(X_s^\nu - \tilde{X}_s^\nu)ds.
        \end{split}
    \end{equation}
    We denote $\rho_t^\nu \coloneqq Y_t^\nu - \tilde{Y}_t^\nu$. Then by Young's product inequality applied to \eqref{slow_diff_1}    ,\begin{equation}\label{slow_diff}
        \begin{split}
            \Ee_0(X_t^\nu - \tilde{X}_t^\nu) + 2 \delta_* \int_0^t \D_0(X_s^\nu - \tilde{X}_s^\nu)ds &\leq \nu^{2/3}\frac{C_*^2}{2\delta_*} \sup_{s \in [0,t]} \Ee_{\neq}(Y_s) \int_0^t \D_{\neq}(\rho_s^\nu) ds\\
            & \quad + \nu^{2/3} \frac{C_*^2}{2\delta_*} \sup_{s \in [0,t]}\Ee_{\neq}(\rho_s^\nu) \int_0^t \D_{\neq}(\tilde{Y}_s) ds.
        \end{split}
    \end{equation}
    Having obtained \eqref{slow_diff}, we now need to estimate the terms involving $\rho_t^\nu$. We decompose $b_m(Y_t^\nu) - b_m(\tilde{Y}_t^\nu)$ in a similar manner to \eqref{split_nonlin}. It\^o's formula and Lemmas \ref{linear_deterministic} and \ref{nonlinear_deterministic} applied to $\Ee_{\neq}(\rho_t^\nu)$ give for $t < \sigma$:
 \begin{equation}
     \begin{split}
         \Ee_{\neq}(\rho_t^\nu) + 8 \delta_* \int_0^t \nu^{-2/3}\Ee_{\neq}(\rho_s^\nu)  + \D_{\neq}(\rho_s^\nu)ds &\leq C_* \nu^{\beta} \int_0^t \Ee_0^{1/2}(X_s^\nu - \tilde{X}_s^\nu) \D_{\neq}^{1/2}(\rho_t^\nu) \D_{\neq}^{1/2}(\tilde{Y}_s^\nu)ds\\
         & \quad + C_* \nu^{\beta} \int\Ee_0^{1/2}(X_s^\nu) \D_{\neq}(\rho_s^\nu) ds + C_* \nu^{\alpha} \int_0^t \Ee_{\neq}^{1/2}(Y_s^\nu) \D_{\neq}(Y_s^\nu) ds.
     \end{split}
 \end{equation}
  For $t < \tau_M \wedge \tau_{\tilde{M}} \wedge \sigma$, we have by Propositions \ref{energy_bounds_for_original} and \ref{energy_bounds_for_linear}, Young's product inequality, and H\"older's inequality,
    \begin{equation}\label{approx_est_for_Y}
        \begin{split}
            \Ee_{\neq}(\rho_t^\nu)& + 4 \delta_* \int_0^t \nu^{-2/3}\Ee_{\neq}(\rho_s^\nu)  + \D_{\neq}(\rho_s^\nu)ds\\
            &\leq \frac{C_*^2}{2\delta_*} \nu^{2\beta} \sup_{s \in [0,t]} \Ee_0(X_s - \tilde{X}_s^\nu) \int_0^t \D_{\neq}(\tilde{Y}_s)ds + C_*\nu^\alpha \sup_{s \in [0,t]} \Ee_{\neq}^{1/2}(Y_s^\nu) \int_0^t \D_{\neq}(Y_s^\nu)ds\\
            &\leq \frac{C_*^2}{2\delta_*} \nu^{2\beta} \sup_{s \in [0,t]} \Ee_0(X_s - \tilde{X}_s^\nu) \int_0^t \D_{\neq}(\tilde{Y}_s)ds + C_*\nu^{\alpha}(\frac{1}{2}c_* \nu^{-\alpha'})^{1/2}   \int_0^t \D_{\neq}(\tilde{Y}_s)ds.
        \end{split}
    \end{equation}
   Altogether we have for some absolute constant $C$,
    \begin{equation}
    \begin{split}
                \Ee_{\neq}(\rho_t^\nu) + 4 \delta_* \int_0^t \nu^{-2/3}\Ee_{\neq}(\rho_s^\nu)  + \D_{\neq}(\rho_s^\nu)ds & \leq C( \nu^{2\beta} \sup_{s \in [0,t]}\Ee_0(X_s^\nu - \tilde{X}_s^\nu) + \nu^{\alpha - \alpha'/2}) \int_0^t \D_{\neq}(\tilde{Y}_s)ds.
    \end{split}
    \end{equation}
    Then by Proposition \ref{energy_bounds_for_original} and $t < \tau_M \wedge \tau_{\tilde{M}} \wedge \sigma$,
\begin{equation}\label{dissip_of_diff}
        \begin{split}
            \Ee_{\neq}(\rho_t^\nu) +  &\delta_* \int_0^t\nu^{-2/3} \Ee_{\neq}(\rho_s)+ \D_{\neq}(\rho_s^\nu) ds \\
            &\leq C( \nu^{2\beta} \sup_{s \in [0,t]}\Ee_0(X_s^\nu - \tilde{X}_s^\nu) + \nu^{\alpha - \alpha'/2})(\Ee_{\neq}(Y_0^\nu) + ||\Psi||^2 (t\nu^{-2/3} + 4\delta_*^{-1}\nu^{-\alpha'/2})),
        \end{split}
    \end{equation}
    which implies by Gr\"onwall's inequality
    \begin{equation}\label{energy_of_diff}
    \begin{split}
                \Ee_{\neq}(\rho_t^\nu) &\leq  C( \nu^{2\beta} \sup_{s \in [0,t]}\Ee_0(X_s^\nu - \tilde{X}_s^\nu) + \nu^{\alpha - \alpha'/2})(\Ee_0(Y_0^\nu) + ||\Psi||^2(1+ \nu^{-\alpha'/2})).
    \end{split}
    \end{equation}
    Applying \eqref{dissip_of_diff} and \eqref{energy_of_diff} to \eqref{slow_diff} along with the assumptions on the initial conditions, we find for $t < \tau_M \wedge \tau_{\tilde{M}}$:
    \begin{equation}
        \begin{split}
             \Ee_{0}(X_t^\nu - \tilde{X}_t^\nu) & \leq C( ||\Psi||, T) \nu^{2/3}( \nu^{2\beta} \sup_{s \in [0,t]}\Ee_0(X_s^\nu - \tilde{X}_s^\nu) + \nu^{\alpha -  \alpha'/2})\nu^{-\alpha'}(\nu^{-2/3} + \nu^{-\alpha'}).
        \end{split}
    \end{equation}
    Thus,
    \begin{equation}
        \begin{split}
            \sup_{t \in [0,T \wedge \tau]}\Ee_{0}(X_t^\nu - \tilde{X}_t^\nu) & \leq C(||\Psi||,  T) \nu^{\alpha - 2\alpha'}.
        \end{split}
    \end{equation}
    This completes the proof.
\end{proof}
Having approximated the $(X_t^\nu, Y_t^\nu)$ processes with the pseudo-linearized processes $(\tilde{X}_t^\nu, \tilde{Y}_t^\nu)$, it is natural to estimate the H\"older continuity in time of the pesudo-linearized processes. In some proofs of the averaging principle for various equations, such as \cites{gao2021averaging, li2018averaging, liu2024averaging}, one avoids directly estimating the H\"older regularity of the processes, and instead utilizes weaker time-integrated estimates. However, in our case, the non-Lipschitz nature of the non-linearity necessitates supremum in time estimates, which require the full H\"older regularity estimate. Recall the notation $t(\delta) = \lfloor t/\delta \rfloor \delta$ and $K = \lfloor T/\delta \rfloor$.
\begin{proposition}\label{continuity_prop}
Let $C_0$ and $\nu_*$ be determined by Proposition \ref{approx_with_lin}, and suppose that $\nu < \nu^*$ and
$$\Ee_{\neq}(Y_0^\nu) < C_0\nu^{-\alpha'/2}, \;\Ee_0(X_0^\nu) < C_0\nu^{-\alpha'}, \;  \nu^{a/3}\Ee_0(|\partial_y|^a X_0^\nu) < C_0\nu^{-\alpha'}.$$
If $\nu^{1/3} >\delta > \nu^{2/3 - \alpha'}$, then there exists a constant $C = C(||\Psi||, T, a) > 0$ such that
\begin{subequations}
    \begin{equation}\label{holder_continuity}
    \sup_{t \in [0,T \wedge \tau \wedge \sigma]}||\tilde{X}_t^\nu - \tilde{X}_{t(\delta)}^\nu||_{\mathcal{H}_0} \leq C\delta^{a/2} \nu^{-a/6 -2\alpha'},
\end{equation}
\begin{equation}\label{holder_continuity_aux}
    \sup_{t \in [0,T \wedge \tau \wedge \sigma]}||\hat{X}_t^\nu - \hat{X}_{t(\delta)}^\nu||_{\mathcal{H}_0} \leq C\delta^{a/2} \nu^{-a/6 -2\alpha'},
\end{equation}
\end{subequations}
    and $\Pb(\tau > T) \geq 1 - (T \delta^{-1} + 6) \exp(-\nu^{-\alpha'/2})$.
\end{proposition}
\begin{proof}[Proof of Proposition \ref{continuity_prop}]
We focus our attention on the proof of \eqref{holder_continuity}, as the proof of \eqref{holder_continuity_aux} is almost identical. Let $t \in [0,T \wedge \tau \wedge \sigma]$. Then for $k = \lfloor t/\delta\rfloor$, we have by the mild formulation of $\tilde{X}_t^\nu$,
\begin{equation}
    \begin{split}\label{eqn: start_of_continuity_prop}
        ||\tilde{X}_t^\nu - \tilde{X}_{k\delta}^\nu||_{\mathcal{H}_0} \leq ||(e^{(t-k\delta)\partial_y^2} -I) \tilde{X}_{k\delta}^\nu||_{\mathcal{H}_0} + \nu^{-1/6}|| \int_{k\delta}^t e^{(t-s)\partial_y^2} b_0(\tilde{Y}_s^\nu) ds ||_{\mathcal{H}_0}.
    \end{split}
\end{equation}
The first term is estimated by Proposition \ref{proposition: holder_in_space} and \eqref{heat_propogator}. In particular, we apply Proposition \ref{proposition: holder_in_space} with $a$ as assumed in the statement of Theorem \ref{main_theorem} and $\rho = 3\alpha'/4$:
\begin{equation}\label{eqn: second_for_holder}
    \begin{split}
        ||(e^{(t-k\delta)\partial_y^2} -I) \tilde{X}_{k\delta}^\nu||_{\mathcal{H}_0}  &\leq C(a)\delta^{a/2}|||\partial_y|^a \tilde{X}_{k\delta}^\nu||_{\mathcal{H}_0}\leq C(a)\delta^{a/2}\nu^{-a/6 - 2\alpha'}.
    \end{split}
\end{equation}
As preparation for the second term of \eqref{eqn: start_of_continuity_prop}, we estimate $\int_{k\delta}^t \D_{\neq}(\tilde{Y}_s^\nu) ds$ for $t \leq T \wedge \tau \wedge \sigma$. To estimate this, we use the techniques of Proposition \ref{energy_bounds_for_original} to find
    \begin{equation}\label{dissip_in_continuity}
        \begin{split}
            2\delta_*  \int_{k\delta}^t  \D_{\neq}(\tilde{Y}_{s}^\nu) +\nu^{-2/3} \Ee_{\neq}(\tilde{Y}_s^\nu) ds& \leq  \Ee_{\neq}(\tilde{Y}_s^\nu) +  \nu^{-2/3}||\Psi||^2 \delta+ 2\nu^{-1/3}\int_{k\delta}^t\langle \Psi dW_s, \tilde{Y}_s^\nu\rangle.
        \end{split}
    \end{equation}
    Let $A^{(k)}$ be the event
    \begin{equation}
        \begin{split}
            A^{(k)} \coloneqq \{ \omega : \sup_{k\delta \leq t \leq (k+1)\delta \wedge T} \nu^{-1/3} \tilde{M}_{t \wedge \tau_{\tilde{Y}}}^{(k)} - \frac{\delta_*}{||\Psi||^2} \frac{1}{2} \nu^{-2/3} \langle \tilde{M}^{(k)} \rangle_{t \wedge \tau_{\tilde{Y}}} \geq \frac{||\Psi||^2}{\delta_*} \nu^{-\alpha'/2}\}.
        \end{split}
    \end{equation}
    Note that by the exponential supermartingale inequality, $\Pb(A^{(k)}) \leq \exp(-\nu^{-\alpha'/2})$. Additionally, $A^{(k)} = \{\tau_{\tilde{M}}^{(k)} \leq (k+1)\delta \wedge T\}$. Returning to \eqref{dissip_in_continuity}, since we are operating under the assumption $t \leq \tau$, we are in the event $(A^{(k)})^c$ and so, similar to the proof of Proposition \ref{energy_bounds_for_original},
    \begin{equation}\label{second_dissip_bound}
        \begin{split}
            2\delta_*\int_{k\delta}^t  \D_{\neq}(\tilde{Y}_{s}^\nu) + \nu^{-2/3} \Ee_{\neq}(\tilde{Y}_s^\nu)ds &\leq  \Ee_{\neq}(\tilde{Y}_{t-\delta}^\nu) +  \nu^{-2/3}||\Psi||^2 \delta+C ||\Psi||^2 \nu^{-\alpha'/2} + \frac{\delta_*}{2||\Psi||^2}\langle \tilde{M}^{(k)}\rangle_{t}.
        \end{split}
    \end{equation}
    Then we have the estimate
    \begin{equation}
        \begin{split}
            \langle\tilde{M}^{(k)}\rangle_{t\wedge \tau} \leq \int_{k\delta}^{t\wedge \tau} ||\Psi||^2\Ee_{\neq}(\tilde{Y}_s^\nu) ds.
        \end{split}
    \end{equation}
    Hence we may absorb the quadratic variation term in \eqref{second_dissip_bound} into the left-hand side (and use $\delta > \nu^{2/3- \alpha'}$) to find
    \begin{equation}\label{eqn: needed for_time_integrated}
        \begin{split}
             2\delta_*\int_{t-\delta}^t  \D_{\neq}(\tilde{Y}_{s}^\nu)  ds &\leq \Ee_{\neq}(\tilde{Y}_{t-\delta}^\nu) +C ||\Psi||^2(\nu^{-2/3}\delta + \nu^{-\alpha'/2})\leq  C(||\Psi||, T) \delta \nu^{-2/3}.
        \end{split}
    \end{equation}
    We now return to the second term on the right-hand side of \eqref{eqn: start_of_continuity_prop}. By the triangle inequality, \eqref{heat_propogator_2}, and Lemma \ref{nonlinear_deterministic} we have for any $p \in (0,1/2)$:
\begin{equation}
    \begin{split}
        \nu^{-1/6} ||\int_{k\delta}^t e^{(t-s)\partial_y^2}b_0(\tilde{Y}_s^\nu) ds||_{\mathcal{H}_0}
        &\leq C(p)\nu^{-1/6} \int_{k\delta}^t (t-s)^{-1/2}||B_0(\tilde{Y}_s^\nu)||_{\mathcal{H}_0} ds\\
        &\leq C_*^{1/2}\nu^{-1/6 + p} \int_{k\delta}^t (t-s)^{-1/2} \Ee_{\neq}^{1-p}(Y_s^\nu) \D^{p}_{\neq}(\tilde{Y}_s^\nu) ds.
    \end{split}
\end{equation}
Now applying H\"older's inequality, Proposition \ref{energy_bounds_for_linear}, and \eqref{eqn: needed for_time_integrated}:
\begin{equation}
    \begin{split}
       \nu^{-1/6} ||\int_{k\delta}^t e^{(t-s)\partial_y^2}b_0(\tilde{Y}_s^\nu) ds||_{\mathcal{H}_0}&\leq \nu^{-1/6 + p}\left(\int_{k\delta}^t (t-s)^{-\frac{1}{2(1-p)}} \Ee_{\neq}(\tilde{Y}_s^\nu)ds\right)^{1-p} \left(\int_{k\delta}^t \D_{\neq}(\tilde{Y}_s^\nu) ds\right)^{p}\\
        &\leq C \nu^{-1/6+p}\left(\frac{1-p}{1-2p}\right)^{1-p}\frac{1}{2}c_*\nu^{-\alpha'(1-p)} \delta^{1-p - 1/2}(\delta \nu^{-2/3})^{p}\\
        &\leq C \delta^{1/2 }\nu^{-1/6 + p/3 - \alpha'(1-p)}.
    \end{split}
\end{equation}
Taking $p$ fixed, such as $p=1/4$ for simplicity, we have
\begin{equation}\label{penultimate_continuity_bound}
     \nu^{-1/6} ||\int_{k\delta}^t e^{(t-s)\partial_y^2}b_0(\tilde{Y}_s^\nu) ds||_{\mathcal{H}_0} \leq C(a) \delta^{1/2}\nu^{-1/6 - \alpha'}.
\end{equation}
Combining \eqref{penultimate_continuity_bound} and \eqref{eqn: second_for_holder} with \eqref{eqn: start_of_continuity_prop} and using $\delta < \nu^{1/3}$ gives the desired continuity bound. Now recall $\tau = \min\{\tau^{(k)} \; | \; k \in \{0, ..., K\}\} \wedge \tau_M \wedge \tau_{\tilde{M}} \wedge \tau_{\hat{M}}$. Then by the union bound, $\tau > T$ with probability at least $1 - (T\delta^{-1} + 6)\exp(-\nu^{-\alpha'})$. This completes the proof.
\end{proof}
Having estimated the H\"older continuity of $\tilde{X}_t^\nu$ on the intervals $[k\delta, (k+1)\delta]$, we are prepared to consider the difference between $\tilde{X}_t^\nu$ and the auxiliary process $\hat{X}_t^\nu$, as described by Proposition \ref{lin_to_aux_prop}.
\begin{proof}[Proof of Proposition \ref{lin_to_aux_prop}]
Let $C_0$ be as in Proposition \ref{approx_with_lin} and $\nu_*$ as in \ref{def_of_nu_star}. Using the definitions of $\tilde{X}_t^\nu$ and $\hat{X}_t^\nu$, we decompose the nonlinearity $b_0(\tilde{Y}_t^\nu) - b_0(\hat{Y}_t^\nu)$ in a similar manner as in \eqref{split_nonlin}. Then by Lemma \ref{nonlinear_deterministic} and Young's inequality, we have
    \begin{equation}\label{main_est_of_aux_proc}
        \begin{split}
            \Ee_0(\tilde{X}_t^\nu - \hat{X}_t^\nu) & \leq \frac{C_*^2}{2\delta_*} \nu^{2/3} \int_0^t \Ee_{\neq}(\tilde{Y}_s^\nu - \hat{Y}_s^\nu) \D_{\neq}(\tilde{Y}_s^\nu) + \Ee_{\neq}(\hat{Y}_s^\nu) \D_{\neq}(\tilde{Y}_s^\nu - \hat{Y}_s^\nu) ds.
        \end{split}
    \end{equation}
    We now set $\rho_t^\nu = \tilde{Y}_t^\nu - \hat{Y}_t^\nu$ (noting the change in notation from the proof of Proposition \ref{approx_with_lin}). By similar methods as in the proof of Propositions \ref{second_bounds_on_frozen} and \ref{approx_with_lin} (in particular Lemmas \ref{lemma: difference_of_U} and \ref{lemma: difference_of_U_derivs}), we obtain for $t < \tau \wedge \sigma$:
    \begin{equation}\label{diff_between_fast_aux}
        \begin{split}
            \Ee_{\neq}(\rho_t^\nu) + 6 \delta_* \int_0^t \D_{\neq}(\rho_s^\nu)& + \nu^{-2/3} \Ee_{\neq}(\rho_s^\nu) ds\\
            &\leq C\nu^{-1/2}\sup_{s \in [0,t]}||U_s - U_{s(\delta)}||_{L^\infty} \left(\int_0^t \D_{\neq}(\tilde{Y}_s^\nu)ds\right)^{1/2} \left(\int_0^t \D_{\neq}(\rho_s^\nu)ds \right)^{1/2}\\
            &\quad+ C\nu^{-1/6}\sup_{s \in [0,t]}||U_t - U_{t(\delta)}||_{C^2}\left( \int_0^t \D_{\neq}(\tilde{Y}_s^\nu) ds \right)^{1/2} \left(\int_0^t \D_{\neq}(\rho_s^\nu)ds \right)^{1/2}\\
            &\quad+C \nu^{2\beta} \int_0^t \Ee_{0}(\tilde{X}_s^\nu - \tilde{X}_{s(\delta)}^\nu)\D_{\neq}(\tilde{Y}_s^\nu) ds+  C \nu^{2\beta} \int_0^t  \Ee_{0}( \tilde{X}_{s(\delta)}^\nu)\D_{\neq}(\rho_s^\nu) ds.
        \end{split}
    \end{equation}
    Continuing from \eqref{diff_between_fast_aux}, we apply $t < \tau \wedge \sigma$, Young's inequality, and Proposition \ref{energy_bounds_for_linear} to bound $\sup_{s \in [0,t]}\Ee_{0}(\tilde{X}_{s(\delta)}^\nu) < \frac{1}{2}c_*\nu^{-2\alpha'}$, which gives us
    \begin{equation}\label{good_aux_est}
        \begin{split}
            \Ee_{\neq}&(\rho_t^\nu) +  4\delta_* \int_0^t \D_{\neq}(\rho_s^\nu) + \nu^{-2/3} \Ee_{\neq}(\rho_s^\nu) ds \\ &\leq C\nu^{-1}\sup_{s \in [0,t]}||U_s - U_{s(\delta)}||_{L^\infty}^2\int_0^t \D_{\neq}(\hat{Y}_s^\nu)ds+ C\nu^{-1/3}\sup_{s \in [0,t]}||U_s - U_{s(\delta)}||_{C^2}^2\int_0^t \D_{\neq}(\tilde{Y}_s^\nu) ds\\
            &\quad + C \nu^{2\beta} \int_0^t \Ee_{0}(\tilde{X}_s^\nu - \tilde{X}_{s(\delta)}^\nu)\D_{\neq}(\tilde{Y}_s^\nu) ds.
        \end{split}
    \end{equation}
    Note that as compared with \eqref{approx_est_for_Y}, the equation \eqref{good_aux_est} lacks the terms arising from the nonlinearity $b_{\neq}$, which corresponds to pure self-interaction of the fast process (or in other words high frequencies). This enables us to apply Proposition \ref{continuity_prop} and find for $t < \tau \wedge \sigma$,
    \begin{equation}\label{est_on_fast_aux}
        \begin{split}
            \Ee_{\neq}(\rho_t^\nu) + &2 \delta_* \int_0^t \D_{\neq}(\rho_s^\nu) + \nu^{-2/3} \Ee_{\neq}(\rho_s^\nu) ds\\
            & \leq C\left(\nu^{2 \beta} \delta^{a} \nu^{-a/3 - 4\alpha'} + \nu^{-1} \sup_{s \in [0,t]}||U_s - U_{s(\delta)}||_{L^\infty}^2 + \nu^{-1/3}\sup_{s \in [0,t]}||U_s - U_{s(\delta)}||_{C^2}^2\right)\\
            &\quad\quad\quad \times(\Ee_{\neq}(Y_0^\nu) + ||\Psi||^2 (\nu^{-2/3}t + 4 \delta_*^{-1} \nu^{-\alpha'/2}))\\
            &\leq C \left(\nu^{2 \beta} \delta^{a} \nu^{-a/3 - 4\alpha'} + \nu^{-1} \sup_{s \in [0,t]}||U_s - U_{s(\delta)}||_{L^\infty}^2 + \nu^{-1/3}\sup_{s \in [0,t]}||U_s - U_{s(\delta)}||_{C^2}^2\right)\\
            &\quad\quad\quad \times(\nu^{-\alpha'} + ||\Psi||^2 t \nu^{-2/3}),
        \end{split}
    \end{equation}
    and at the level of the energy
    \begin{equation}\label{est_on_fast_aux_energ}
        \begin{split}
            \Ee_{\neq}(\rho_t^\nu) &\leq C\left(\nu^{2 \beta} \delta^{a} \nu^{-a/3 - 2\alpha'} + \nu^{-1} \sup_{s \in [0,t]}||U_s - U_{s(\delta)}||_{L^\infty}^2 + \nu^{-1/3}\sup_{s \in [0,t]}||U_s - U_{s(\delta)}||_{C^2}^2\right)\\&\quad\quad \times( \Ee_{\neq}(Y_0^\nu) + ||\Psi||^2 (1 + 4 \delta_*^{-1} \nu^{-\alpha'/2}))\\
            &\leq C \left(\nu^{2 \beta} \delta^{a} \nu^{-a/3 - 4\alpha'} + \nu^{-1} \sup_{s \in [0,t]}||U_s - U_{s(\delta)}||_{L^\infty}^2 + \nu^{-1/3}\sup_{s \in [0,t]}||U_s - U_{s(\delta)}||_{C^2}^2\right) \nu^{-\alpha'}.
        \end{split}
    \end{equation}
    Applying \eqref{est_on_fast_aux} and \eqref{est_on_fast_aux_energ} to \eqref{main_est_of_aux_proc}, we see
    \begin{equation}
    \begin{split}
                \E &\sup_{t \in [0,T\wedge \tau \wedge \sigma]} \Ee_0(\tilde{X}_t^\nu - \hat{X}_t^\nu)\\
                &\leq C \nu^{2/3} \E \int_0^{T\wedge \tau \wedge \sigma} \Ee_{\neq}(\tilde{Y}_s^\nu - \hat{Y}_s^\nu) \D_{\neq}(\tilde{Y}_s^\nu) + \Ee_{\neq}(\hat{Y}_s^\nu) \D_{\neq}(\tilde{Y}_s^\nu - \hat{Y}_s^\nu) ds\\
                &\leq C(||\Psi||, T)\nu^{-\alpha'}\biggl( \E\left(\sup_{t \in [0,T\wedge \tau \wedge \sigma]} \Ee_{\neq}(\tilde{Y}_s^\nu - \hat{Y}_s^\nu)\right) +  \nu^{2/3} \E \int_0^{T\wedge \tau\wedge \sigma} \D_{\neq}(\tilde{Y}_s^\nu - \hat{Y}_s^\nu) ds \biggr)\\
                &\leq C\nu^{-\alpha'}\biggl(\nu^{2 \beta} \delta^{a} \nu^{-a/3 - 4\alpha'} + \nu^{-1} \E \sup_{t \in [0,T\wedge \tau \wedge \sigma]}||U_t - U_{t(\delta)}||_{L^\infty}^2+ \nu^{-1/3}\E \sup_{t \in [0,T\wedge \tau \wedge \sigma]}||U_t -U_{t(\delta)}||_{C^2}^2\biggr).
    \end{split}
    \end{equation}
    We apply now apply Lemma \ref{lemma: holder_continuity_of_U} as well as $\alpha' < 5\beta$ and $\delta < \nu^{1/3}$ to obtain the desired result. Note that $\nu^{-1/3} \delta \log(1/\delta) \lesssim \delta^{a} \nu^{-a/3}$ for $a \in (0,1)$.
\end{proof}
Now we turn to Proposition \ref{key_prop}, which controls the difference between $\hat{X}_t^\nu$ and $\bar{X}_t^\nu$. We will use the following useful fact concerning the fractional norms defined in \eqref{def_of_fractional_energy}:
\begin{equation}
    ||f||_{\mathcal{H}_0^{-1/2}}\leq||f||_{\mathcal{H}_0} \leq ||f||_{\mathcal{H}_0^{1/2}} \leq \D_0^{1/2}(f).
\end{equation}
\begin{proof}[Proof of Proposition \ref{key_prop}]
Let $C_0$ be as in Proposition \ref{approx_with_lin} and $\nu_*$ as in \eqref{def_of_nu_star}. Through integration by parts, it is easy to show that
\begin{equation}
\begin{split}
    ||\hat{X}_t - \bar{X}_t||_{\Hh_0^{-1/2}}^2 + 2 \mathrm{Re} \int_0^t ||\hat{X}_t - \bar{X}_t^\nu||_{\Hh_0^{1/2}}^2 = -2\nu^{-1/6} \int_0^t \mathrm{Re}\langle  b_0(\hat{X}_s) - \bar{b}_0(U_s,\bar{X}_s) , \hat{X}_t^\nu - \bar{X}_s^\nu \rangle_{\mathcal{H}_0^{-1/2}} ds. 
    \end{split}
\end{equation}
Thus we are are required to control the difference $b_0(\hat{Y}_s^\nu) - \bar{b}_0(\bar{X}_s^\nu)$. It is unclear how to do this directly, and so we insert additional terms which we can control. In so doing, we define the quantities $I_1$, $I_2$, $I_3$, and $I_4$ below,
\begin{equation}\label{initial_approx_split}
        \begin{split}
            \nu^{-1/6} &\biggl|\int_0^t \mathrm{Re}\langle b_0(\hat{X}_s) - b_0(U_s,\bar{X}_s) , \hat{X}_t^\nu - \bar{X}_s^\nu \rangle_{\mathcal{H}_0^{-1/2}} ds| \\
            &= \nu^{- 1/6} |\int_0^t \mathrm{Re}\langle b_0(\hat{Y}_s^\nu) - \bar{b}_0(U_{s(\delta)},\tilde{X}_{s(\delta)}^\nu)+ \bar{b}_0(U_{s(\delta)},\tilde{X}_{s(\delta)}^\nu) - \bar{b}_0(U_{s},\tilde{X}_s^\nu)\\
            &\hphantom{ \nu^{-1/6} \int_0^t \langle}+ \bar{b}_0(U_s, \tilde{X}_{s}^\nu) - \bar{b}_0(U_s,\hat{X}_s^\nu)+ \bar{b}_0(U_s,\hat{X}_{s}^\nu) - \bar{b}_0(U_s, \bar{X}_s^\nu) , \hat{X}_s^\nu - \bar{X}_s^\nu \rangle_{\mathcal{H}_0^{-1/2}} ds\biggr|\\
            &\leq \nu^{ - 1/6} \biggl|\int_0^t \mathrm{Re}\langle b_0(\hat{Y}_s^\nu) - \bar{b}_0(U_{s(\delta)}, \tilde{X}_{s(\delta)}^\nu), \hat{X}_s^\nu - \bar{X}_s^\nu \rangle_{\mathcal{H}_0^{-1/2}} ds\biggr|\\
            &\quad+ \nu^{- 1/6} \biggl|\int_0^t \langle \bar{b}_0(U_{s(\delta)}, \tilde{X}_{s(\delta)}^\nu) - \bar{b}_0(U_s,\tilde{X}_s^\nu), \hat{X}_s^\nu - \bar{X}_s^\nu \rangle_{\mathcal{H}_0^{-1/2}} ds\biggr|\\
            &\quad + \nu^{ - 1/6} \biggl|\int_0^t \langle \bar{b}_0(U_s, \tilde{X}_{s}^\nu) - \bar{b}_0(U_s,\hat{X}_s^\nu)ds, \hat{X}_s^\nu - \bar{X}_s^\nu \rangle_{\mathcal{H}_0^{-1/2}} ds\biggr|\\
            &\quad+\nu^{ - 1/6} \biggl| \int_0^t \langle \bar{b}_0(U_s, \hat{X}_{s}^\nu) - \bar{b}_0(U_s, \bar{X}_s^\nu), \hat{X}_s^\nu - \bar{X}_s^\nu \rangle_{\mathcal{H}_0^{-1/2}} ds\biggr|\\
            &\eqqcolon I_1 + I_2 + I_3 + I_4.
        \end{split}
    \end{equation}
    We will deal with $I_2$, $I_3$, and $I_4$ first, as they are controlled in somewhat similar ways. First we have $I_2$. Using integration by parts and the spectral calculus, Cauchy-Schwarz, and Young's inequality, we find
    \begin{equation}\label{eqn: initial_I2}
        \begin{split}
             I_2 &=\nu^{- 1/6} \biggl|\int_0^t \langle \bar{B}_0(U_{s(\delta)}, \tilde{X}_{s(\delta)}^\nu) - \bar{B}_0(U_s,\tilde{X}_s^\nu), \partial_y |\partial_y|^{-1}\hat{X}_s^\nu - \bar{X}_s^\nu \rangle_{\mathcal{H}_0} ds \biggr|\\
             &\leq \nu^{- 1/3} C \int_0^t ||\bar{B}_0(U_{s(\delta)}, \tilde{X}_{s(\delta)}^\nu) - \bar{B}_0(U_s,\tilde{X}_s^\nu)||_{\Hh_0}^2 + \frac{\delta_*}{2} \int_0^t ||\hat{X}_s^\nu - \bar{X}_s^\nu||_{\mathcal{H}_0}^2 ds\\
             &\leq \ \nu^{- 1/3} C \int_0^t ||\bar{B}_0(U_{s(\delta)}, \tilde{X}_{s(\delta)}^\nu) - \bar{B}_0(U_s,\tilde{X}_s^\nu)||_{\Hh_0}^2 + \frac{\delta_*}{2} \int_0^t ||\hat{X}_s^\nu - \bar{X}_s^\nu||_{\mathcal{H}_0^{1/2}}^2 ds
        \end{split}
    \end{equation}
    The second term of the right-hand side of \eqref{eqn: initial_I2} will be absorbed into the left-hand side of \eqref{initial_approx_split}. For the first term, as long as $t < \tau \wedge \sigma$, we can apply Proposition \ref{proposition: ergodic_to_Lipschitz} and find,
\begin{equation}\label{control_I2_first}
        \begin{split}
             \nu^{- 1/3} & \int_0^t ||(\bar{B}_0(U_{s(\delta)}, \tilde{X}_{s(\delta)}^\nu) - \bar{B}_0(U_s,\tilde{X}_s^\nu)||_{\Hh_0}^2\\
             &\leq C||\Psi||^4 \int_0^t \nu^{2\beta}||\tilde{X}_{s(\delta)}^\nu - \tilde{X}_s^\nu||_{\mathcal{H}_0}^2+ \nu^{-1}||U_{s(\delta)}- U_s||_{L^\infty}^2 + \nu^{-1/3}||U_{s(\delta)}- U_s||_{C^2}^2ds.
        \end{split}
    \end{equation}
    Finally, we apply Proposition \ref{continuity_prop} and see that \eqref{control_I2_first} is bounded by
    \begin{equation}\label{I2_final}
    \begin{split}
                C(||\Psi||, T, a)  \left(\delta^{a} \nu^{-a/3} + \nu^{-1}\sup_{s \in [0,T ]}||U_{s(\delta)}- U_s||_{L^\infty}^2 + \nu^{-1/3} \sup_{s \in [0,T]}||U_{s(\delta)}- U_s||_{C^2}^2\right).
    \end{split}
    \end{equation}
    In a similar fashion, we use Proposition \ref{proposition: ergodic_to_Lipschitz} and $\alpha'< 2 \beta$ to estimate
    \begin{equation}\label{I3_final}
    \begin{split}
                I_3&\leq C (||\Psi||, T)||\Psi||^4 \nu^{2\beta} \int_0^t ||\tilde{X}_{s}^\nu - \hat{X}_s^\nu||_{\mathcal{H}_0}^2 ds + \frac{\delta_*}{2} \int_0^t ||\hat{X}_s^\nu - \bar{X}_s^\nu||_{\mathcal{H}_0}^2 ds\\
                &\leq C ||\Psi||^4 \nu^{2\beta} \sup_{s \in [0,T\wedge \tau \wedge \sigma]} ||\tilde{X}_{s}^\nu - \hat{X}_s^\nu||_{\mathcal{H}_0}^2 + \frac{\delta_*}{2} \int_0^t ||\hat{X}_s^\nu - \bar{X}_s^\nu||_{\Hh_0^{1/2}}^2 ds.
    \end{split}
    \end{equation}
    To address $I_4$, we take more care to keep track of the constants. By Cauchy-Schwarz, Proposition \ref{proposition: ergodic_to_Lipschitz}, and $t < \tau \wedge \sigma$,  we find that there exists a constant $\tilde{C}$ depending only on $\delta_*$, $m$, $c_*$, and $C_*$ such that
    \begin{equation}\label{I4_final}
    \begin{split}
        I_4 &\leq \left(\int_0^t ||\bar{B}_0(U_{s}, \hat{X}_{s}^\nu) - \bar{B}_0(U_s,\bar{X}_s^\nu)||_{\Hh_0}^2 ds\right)^{1/2} \left(\int_0^t ||\hat{X}_s^\nu - \bar{X}_s^\nu||_{\Hh_0}^2 ds\right)^{1/2} \\
       &\leq \tilde{C} ||\Psi||^2 \nu^{\beta} \left(\int_0^t||\hat{X}_s^\nu - \bar{X}_{s}||_{\mathcal{H}_0}^2ds\right)^{1/2} \left(\int_0^t||\hat{X}_s^\nu - \bar{X}_{s}||_{\mathcal{H}_0}^2ds\right)^{1/2}\\
       &\leq \tilde{C} ||\Psi||^2 \nu^{\beta} \int_0^t||\hat{X}_s^\nu - \bar{X}_{s}||_{\mathcal{H}_0^{1/2}}^2ds.
    \end{split}
\end{equation}
Then we note that $\nu < \nu_*$ and $\nu^* \leq \left(\frac{\delta_*}{2\tilde{C} ||\Psi||^2 }\right)^{1/\beta}$ by \eqref{def_of_nu_star}. Collecting the bounds \eqref{I2_final}, \eqref{I3_final}, and \eqref{I4_final}, and applying them to \eqref{initial_approx_split} alongside $\delta_* \leq 1/16$, we find for $t < \tau \wedge \sigma$
\begin{equation}\label{estimate_with_gronwall}
     \begin{split}
     ||\hat{X}_t& - \bar{X}_t||_{\Hh_0^{-1/2}}^2 + 4\delta_* \int_0^t ||\hat{X}_s - \bar{X}_s||^2_{\Hh_0^{1/2}} ds\\
     &\leq C(||\Psi||, T, a)\biggl(\delta^{a} \nu^{-a/3} + \nu^{-1}\sup_{s \in [0,T]}||U_{s(\delta)}- U_s||_{L^\infty}^2 + \nu^{-1/3} \sup_{s \in [0,T]}||U_{s(\delta)}- U_s||_{C^2}^2\\
&\quad\quad\quad\quad\quad\quad\quad\quad+ \nu^{2\beta}\sup_{s \in [0,T\wedge \tau \wedge \sigma]} ||\tilde{X}_{s}^\nu - \hat{X}_s^\nu||_{\mathcal{H}_0}^2+ I_1(t)\biggr).
     \end{split}
 \end{equation}
 Taking supremum and expectation, then using Propositions \ref{lemma: holder_continuity_of_U} and \ref{lin_to_aux_prop}, as well as $\delta < \nu^{1/3}$ and $\alpha' < 2\beta$, we find
 \begin{equation}\label{intermediate_bound}
     \begin{split}
         \E \sup_{t \in [0,T\wedge \tau \wedge \sigma]} ||\hat{X}_t - \bar{X}_t||_{\Hh_0^{-1/2}}^2 \leq C\left(\delta^{a} \nu^{-a/3} + \E \sup_{t \in [0,T\wedge \tau \wedge \sigma]} I_1(t)\right).
     \end{split}
 \end{equation}
 We now turn our attention to the term $I_1$, which is where we take the most advantage of the discretization. We take inspiration from  \cite{liu2023strong}, where the authors used non-semigroup methods to obtain an averaging principle. In particular,  we insert factors of  $\hat{X}_{s(\delta)}^\nu$ and $\bar{X}_{s(\delta)}^\nu$.
 \begin{equation}
     \begin{split}
         I_{1} &=\nu^{ - 1/6} \biggl|\int_0^{t} \mathrm{Re}\langle b_0(\hat{Y}_s^\nu) - \bar{b}_0(U_{s(\delta)}, \tilde{X}_{s(\delta)}^\nu), \hat{X}_s^\nu - \bar{X}_{s}^\nu \rangle_{\Hh_0^{-1/2}} ds \biggr|\\
         &\leq \nu^{ - 1/6} \biggl|\int_0^t \mathrm{Re}\langle b_0(\hat{Y}_s^\nu) - \bar{b}_0(U_{s(\delta)}, \tilde{X}_{s(\delta)}^\nu), \hat{X}_s^\nu - \hat{X}_{s(\delta)}^\nu \rangle_{\Hh_0^{-1/2}} ds \biggr|\\
         &\quad + \nu^{ - 1/6} \biggl|\int_0^t \mathrm{Re}\langle b_0(\hat{Y}_s^\nu) - \bar{b}_0(U_{s(\delta)}, \tilde{X}_{s(\delta)}^\nu), \hat{X}_{s(\delta)}^\nu - \bar{X}_{s(\delta)}^\nu \rangle_{\Hh_0^{-1/2}} ds \biggr|\\
         &\quad+\nu^{ - 1/6}\biggl| \int_0^t \mathrm{Re}\langle b_0(\hat{Y}_s^\nu) - \bar{b}_0(U_{s(\delta)}, \tilde{X}_{s(\delta)}^\nu), \bar{X}_s^\nu - \bar{X}_{s(\delta)}^\nu \rangle_{\Hh_0^{-1/2}} ds\biggr|\\
         &\eqqcolon I_{1}^a + I_{1}^b + I_{1}^c.
     \end{split}
 \end{equation}
To handle $I_1^a$, we compute by integration by parts and the spectral calculus, Cauchy-Schwarz, Propositions \ref{energy_bounds_for_auxiliary} and \ref{proposition: ergodic_to_bounded}, and Proposition \ref{continuity_prop}:
\begin{equation}\label{I_1_a_bound}
    \begin{split}
        \E \sup_{t \in [0, T \wedge \tau \wedge \sigma]} I_{1}^a & \leq C \nu^{-1/6} \E \int_0^{T\wedge \tau \wedge \sigma} ||B_0(\hat{Y}_s^\nu)  - \bar{B}_0(U_{s(\delta)}, \tilde{X}_{s(\delta)}^\nu)||_{\Hh_0} ||\hat{X}_s^\nu - \hat{X}_{s(\delta)}^\nu||_{\Hh_0} ds\\
        &\leq C\nu^{-1/6}\left(\E \int_0^{T\wedge \tau \wedge \sigma} ||B_0(\hat{Y}_s^\nu)||_{\Hh_0}^2 +|| \bar{B}_0(U_{s(\delta)}, \tilde{X}_{s(\delta)}^\nu)||_{\Hh_0}^2 ds\right)^{1/2}\\
        &\quad\quad\quad\quad\quad\left(\E \int_0^{T\wedge \tau \wedge \sigma} ||\hat{X}_s^\nu - \hat{X}_{s(\delta)}^\nu||_{\Hh_0}^2 ds\right)^{1/2}\\
        &\leq C \delta^{a/2}\nu^{-a/6 - 3\alpha'}.
    \end{split}
\end{equation}
Similarly, we use Proposition \ref{prop: avg_holder_cont} to obtain
\begin{equation}\label{I_1_c_bound}
    \E \sup_{t \in [\delta, T \wedge \tau \wedge \sigma]} I_{1}^c \leq C \delta^{a/2}\nu^{-a/6 - 2\alpha'}.
\end{equation}
Continuing from \eqref{intermediate_bound}, we have by \eqref{I_1_a_bound}, \eqref{I_1_c_bound}, and $\delta < \nu^{1/3}$,
 \begin{equation}\label{intermediate_bound_2}
     \begin{split}
         \E \sup_{t \in [0,T\wedge \tau \wedge \sigma]} ||\hat{X}_t - \bar{X}_t||_{\Hh_0^{-1/2}}^2 \leq C\left(\delta^{a/2} \nu^{-a/6 - 3\alpha'} + \E \sup_{t \in [0,T\wedge \tau \wedge \sigma]} I_1^b(t)\right).
     \end{split}
 \end{equation}
 This leaves us to consider $I_1^b$. We introduce the shorthand notation $\bar{\tau} \coloneqq \tau \wedge \sigma.$ We now integrate by parts (and use the spectral definition of the fractional derivative) and split the integral up according to the discretization:
  \begin{equation}\label{I1_first}
     \begin{split}
         \E(\sup_{t \in [0,T\wedge \tau \wedge \sigma ]} I_1^b(t)) & \leq\nu^{-1/6} \E \sup_{t \in [0,T]} 1_{t\leq \bar{\tau}}\left|\int_0^{t}  \langle B_0(\hat{Y}_s^\nu) - \bar{B}_0(U_{s(\delta)}, \tilde{X}_{s(\delta)}^\nu) , \partial_y |\partial_y|^{-1}(\hat{X}_{s(\delta)}^\nu - \bar{X}_{s(\delta)}^\nu) \rangle_{\Hh_0} ds \right|\\
         &\leq \nu^{-1/6}\E \sup_{t \in [0,T]}\sum_{k=0}^{\lfloor t/\delta \rfloor-1} 1_{k\delta \leq \bar{\tau}}\left| \int_{k\delta}^{(k+1)\delta} \langle B_0(\hat{Y}_s^\nu) - \bar{B}_0(U_{k\delta},\tilde{X}_{k \delta}^\nu)  ,  \partial_y |\partial_y|^{-1}(\hat{X}_{k\delta}^\nu - \bar{X}_{k\delta}^\nu )\rangle_{\Hh_0} ds \right|\\
         &\quad+\nu^{-1/6} \E \sup_{t \in [0,T]}1_{t\leq \bar{\tau}}\left|\int_{t(\delta)}^{t} \langle B_0(\hat{Y}_s^\nu) - \bar{B}_0(U_{t(\delta)},\tilde{X}_{t(\delta)}^\nu) , \partial_y |\partial_y|^{-1}(\hat{X}_{t(\delta)}^\nu - \bar{X}_{t(\delta)}^\nu) \rangle_{\Hh_0} ds\right|\\
         &\leq \nu^{-1/6}\E \sum_{k=0}^{K-1} \left| \int_{k\delta}^{(k+1)\delta}1_{k\delta \leq \bar{\tau}} \langle B_0(\hat{Y}_s^\nu) - \bar{B}_0(U_{k\delta},\tilde{X}_{k \delta}^\nu)  ,  \partial_y |\partial_y|^{-1}(\hat{X}_{k\delta}^\nu - \bar{X}_{k\delta}^\nu )\rangle_{\Hh_0} ds \right|\\
         &\quad+\nu^{-1/6} \E \sup_{t \in [0,T]}\left|\int_{t(\delta)}^{t} 1_{s \leq \bar{\tau}}\langle B_0(\hat{Y}_s^\nu) - \bar{B}_0(U_{t(\delta)},\tilde{X}_{t(\delta)}^\nu) ,  \partial_y |\partial_y|^{-1}(\hat{X}_{t(\delta)}^\nu - \bar{X}_{t(\delta)}^\nu) \rangle_{\Hh_0} ds\right|\\
         &\eqqcolon J_1 + J_2.\\
     \end{split}
 \end{equation}
 The $J_2$ term is straightforward to estimate. By a variant of Proposition \ref{energy_bounds_for_auxiliary}, together with Proposition \ref{proposition: ergodic_to_bounded} and the \textit{a priori} estimate on $\sup_{t \in [0, T\wedge \sigma]}||\bar{X}_{t(\delta)}||_{\Hh_0}^2$ implied by Theorem \ref{theorem: averaged_well_posed},
 \begin{equation}\label{J_2_bound}
     \begin{split}
         \nu^{-1/6} \E&  \sup_{t \in [0,T]}\left|\int_{t(\delta)}^t 1_{s \leq \bar{\tau}}\langle B_0(\hat{Y}_s^\nu) - \bar{B}_0(U_{t(\delta)},\tilde{X}_{t(\delta)}^\nu), \partial_y |\partial_y|^{-1}\hat{X}_{t(\delta)}^\nu - \bar{X}_{t(\delta)}^\nu \rangle_{\Hh_0} ds\right|\\
         &\leq \nu^{-1/6} \E \sup_{t \in [0,T]} \int_{t(\delta)}^t 1_{s \leq \bar{\tau}}|| B_0(\hat{Y}_s^\nu) - \bar{B}_0(U_{t(\delta)},\tilde{X}_{t(\delta)}^\nu)||_{\Hh_0} ||\hat{X}_{t(\delta)}^\nu - \bar{X}_{t(\delta)}^\nu ||_{\Hh_0} ds\\&\leq C\nu^{-\alpha'} \nu^{-1/6} \E \max_{0\leq k \leq K} \int_{k\delta}^{(k+1)\delta} 1_{s\leq \bar{\tau}\wedge t}(|| B_0(\hat{Y}_s^\nu)||_{\Hh_0} +|| \bar{B}_0(U_{t(\delta)},\tilde{X}_{t(\delta)}^\nu)||_{\Hh_0})ds\\
         &\leq C(||\Psi||)\nu^{-2\alpha'} \delta.
     \end{split}
 \end{equation}
 We now focus our attention on $J_1$. We begin by using Cauchy-Schwarz and the fact that we have removed the time dependence from the $\hat{X}_{k\delta}^\nu$ and $\bar{X}_{k\delta}^\nu$ terms:
\begin{equation}\label{J1_decomp}
\begin{split}
         J_1 &= \nu^{-1/6}\E \sum_{k=0}^{K-1} \left| \int_{k\delta}^{(k+1)\delta} 1_{k\delta\leq \bar{\tau}}\langle  B_0(\hat{Y}_s^\nu) - \bar{B}_0(U_{k\delta},\tilde{X}_{k \delta}^\nu)  , \partial_y |\partial_y|^{-1} \hat{X}_{k\delta}^\nu - \bar{X}_{k\delta}^\nu \rangle_{\Hh_0} ds \right| \\
         &\leq \nu^{-1/6}  C \delta^{-1}\max_{k \in \{0, \hdots, K-1\}} \biggl( \left(\E \left[1_{k\delta \leq \bar{\tau}}||\hat{X}_{k\delta}^\nu - \bar{X}_{k\delta}^\nu||_{\Hh_0}^2\right]\right)^{1/2} \\
         &\quad\quad\times\left(\E\biggl| \biggl|\int_{k\delta}^{(k+1)\delta}  1_{k\delta \leq \bar{\tau}}B_0(\hat{Y}_s^\nu) - \bar{B}_0(U_{k\delta}, \tilde{X}_{k\delta}^\nu)) ds \biggr|\biggr|_{\Hh_0}^2\right)^{1/2} \biggr).
\end{split}
\end{equation}
By the triangle inequality, Proposition \ref{energy_bounds_for_auxiliary}, and Theorem \ref{theorem: averaged_well_posed}, 
\begin{equation}\label{alpha_prime_J1}
    \left(\E \left[1_{k\delta \leq \bar{\tau}}||\hat{X}_{k\delta}^\nu - \bar{X}_{k\delta}^\nu||_{\Hh_0}^2\right]\right)^{1/2} \leq C \nu^{-\alpha'}.
\end{equation}
For each $k$, we define
\begin{equation}
    \begin{split}
        J_1^{(k)} \coloneqq \E\biggl| \biggl|\int_{k\delta}^{(k+1)\delta} 1_{k\delta \leq \bar{\tau}} B_0(\hat{Y}_s^\nu) - \bar{B}_0(U_{k\delta}, \tilde{X}_{k\delta}^\nu) ds \biggr|\biggr|_{\Hh_0}^2,
    \end{split}
\end{equation}
so that by \eqref{alpha_prime_J1}, \eqref{J1_decomp} becomes
\begin{equation}\label{J1_decomp_2}
    J_1 \leq C \delta^{-1}\nu^{-1/6 - \alpha'}  \max_{k \in \{0,\hdots, K-1\}}(J_1^{(k)})^{1/2}.
\end{equation}
We now fix $k \in \{0,\hdots,K-1\}$ We make the time change $s \mapsto s -k\delta$ and observe that we may write $J^{(k)}$ as
\begin{equation}\label{k_by_k_expectation}
    \begin{split}
        J_1^{(k)} &= \E\biggl| \biggl| \int_{0}^{\delta} 1_{k\delta \leq \bar{\tau}} B_0(\hat{Y}_{s+k\delta}^\nu) - \bar{B}_0(U_{k\delta}, \tilde{X}_{k \delta}^\nu) ds\biggr | \biggr|_{\Hh_0}^2\\
        &= 2\int_0^\delta \int_r^\delta \E \langle  B_0(\hat{Y}_{s+k\delta}^\nu) - \bar{B}_0(U_{k\delta},\tilde{X}_{k \delta}^\nu), 1_{k\delta \leq \bar{\tau}}  B_0(\hat{Y}_{r+k\delta}^\nu) - \bar{B}_0(U_{k\delta}, \tilde{X}_{k \delta}^\nu)\rangle_{\Hh_0} ds dr\\
        &\eqqcolon 2\int_0^\delta \int_r^\delta L^{(k)}(s,r) ds dr.
    \end{split}
\end{equation}
With $L^{(k)}(s,r)$ in mind, we apply the tower property of conditional expectation to see that for $s > r$ (as is the case in \eqref{k_by_k_expectation}):
\begin{equation} \label{eqn: conditional_expectation}
    \begin{split}
        L^{(k)}(s,r) &= \E \biggl[ \langle  B_0(\hat{Y}_{s+k\delta}^\nu) - \bar{B}_0(U_{k\delta},\tilde{X}_{k \delta}^\nu), 1_{k\delta \leq \bar{\tau}}  B_0(\hat{Y}_{r+k\delta}^\nu) - \bar{B}_0(U_{k\delta},\tilde{X}_{k \delta}^\nu)\rangle_{\Hh_0}\biggr]\\
        &=\E \biggl[ \E\biggl[\langle  B_0(\hat{Y}_{s+k\delta}^\nu) - \bar{B}_0(U_{k\delta},\tilde{X}_{k \delta}^\nu), 1_{k\delta  \leq \bar{\tau}}  B_0(\hat{Y}_{k\delta}^\nu) - \bar{B}_0(U_{k\delta},\tilde{X}_{k \delta}^\nu)\rangle_{\Hh_0} \, \,\vert \mathcal{F}_{k\delta} \biggr]\biggr]\\
        &= \E \biggl[ \E\biggl[\langle \E[ B_0(\hat{Y}_{s+k\delta}^\nu) - \bar{B}_0(U_{k\delta},\tilde{X}_{k \delta}^\nu) \, \vert \mathcal{F}_{r+k\delta} ], 1_{k\delta \leq \bar{\tau}}  B_0(\hat{Y}_{r+k\delta}^\nu) - \bar{B}_0(U_{k\delta},\tilde{X}_{k \delta}^\nu)\rangle_{\Hh_0} \, \,\vert \mathcal{F}_{k\delta} \biggr]\biggr].
    \end{split}
\end{equation}
The key point of \eqref{eqn: conditional_expectation} is that the conditional expectation with respect to $\mathcal{F}_{r + k\delta}$ is now inside the $\Hh_0$ inner product. Then, we apply Cauchy-Schwarz and the fact that $\{k\delta \leq \bar{\tau}\}$ is $\mathcal{F}_{k\delta}$ measurable to find 
\begin{equation}\label{cauch_schwarz_for_L}
    \begin{split}
        |L^{(k)}(s,r)| & \leq  \E \biggl[ \E\biggl[ |\langle \E[  B_0(\hat{Y}_{s+k\delta}^\nu) - \bar{B}_0(U_{k\delta},\tilde{X}_{k \delta}^\nu) \, \vert \mathcal{F}_{r+k\delta} ], 1_{k\delta \leq \bar{\tau}}  B_0(\hat{Y}_{r+k\delta}^\nu) - \bar{B}_0(U_{k\delta},\tilde{X}_{k \delta}^\nu)\rangle_{\Hh_0}| \, \,\vert \mathcal{F}_{k\delta} \biggr]\biggr]\\
        &\leq \E \biggl[ 1_{k\delta \leq \bar{\tau}}\E\biggl[ ||\E[   B_0(\hat{Y}_{s+k\delta}^\nu) - \bar{B}_0(U_{k\delta},\tilde{X}_{k \delta}^\nu) \, \vert \mathcal{F}_{r+k\delta} ]||_{\Hh_0}\\
        &\quad\quad\quad\quad\quad\quad\quad\quad \quad\quad\quad||  B_0(\hat{Y}_{r+k\delta}^\nu) - \bar{B}_0(U_{k\delta},\tilde{X}_{k \delta}^\nu)||_{\Hh_0} \, \,\vert \mathcal{F}_{k\delta} \biggr]\biggr].
    \end{split}
\end{equation}
Because the conditional expectation with respect to $\mathcal{F}_{r+k\delta}$ was placed inside of the inner product in \eqref{eqn: conditional_expectation}, the conditional expectation  with respect to $\mathcal{F}_{r+k\delta}$ is now crucially inside of the $\Hh_0$ norm.
We now embark on a technical analysis of the various components of $|L^{(k)}(s,r)|$. With $(U,X,Y) = (U_{k \delta},\tilde{X}_{k\delta}^\nu, \tilde{Y}_{k\delta}^\nu)$ given, we observe, by the Markov property, that given the $\sigma$-algebra $\mathcal{F}_{r+k\delta}$, we may take $Z = Z_{r/\nu^{2/3}}^{U,X,Y}$ as given. Here $r \in [0,\delta)$. Then $\hat{Y}^\nu_{s+k\delta}$ on the time interval $s \in [r, \delta)$ is distributed like $Z_{(s-r)/\nu^{2/3}}^{U,X,Z}$. Thus
\begin{equation}\label{internal_expectation_for_L}
    \begin{split}
    &\E[B_0(\hat{Y}_{s+k\delta}^\nu)\, \vert \mathcal{F}_{r+k\delta}] \bigg|_{(U,X,Y) = (U_{k\delta}, \tilde{X}_{k\delta}^\nu, \tilde{Y}_{k\delta}^\nu)}= \E[B_0(Z_{(s-r)/\nu^{2/3}}^{U,X, Z})\, \vert \mathcal{F}_{r+k\delta}]\bigg|_{Z = Z_{r/\nu^{2/3}}^{U X,Y}} \bigg|_{(U,X,Y) = (U_{k\delta}, \tilde{X}_{k\delta}^\nu, \tilde{Y}_{k\delta}^\nu)}\\
    &\quad= \int_{\mathcal{H}_{\neq}} \E[ B_0(Z_{(s-r)/\nu^{2/3}}^{U,X, Z})]\bigg|_{Z = Z_{r/\nu^{2/3}}^{U X,Y}}  \mu_\nu^{U,X}(dY') \bigg|_{(U,X,Y) = (U_{k\delta}, \tilde{X}_{k\delta}^\nu, \tilde{Y}_{k\delta}^\nu)}.
    \end{split}
\end{equation}
We now use \eqref{internal_expectation_for_L}, the definition of $\bar{B}_0$ and the invariant measure, and the fact that $\bar{B}_0(U_{k\delta}, \tilde{X}_{k\delta}^\nu)$ is $\mathcal{F}_{k\delta}$ and $\mathcal{F}_{r+k\delta}$-measurable:
\begin{equation}\label{internal_expectiation_for_L_two}
    \begin{split}
        &||\E[  B_0(\hat{Y}_{s+k\delta}^\nu) - \bar{B}_0(U_{k\delta}, X)) \, \vert \mathcal{F}_{r+k\delta} ]||_{\Hh_0} ,\bigg \vert_{(U,X,Y) = (U_{k \delta},\tilde{X}_{k\delta}^\nu, \tilde{Y}_{k\delta}^\nu)}\\
        &\quad= || \int_{\mathcal{H}_{\neq}} \E[ B_0(Z_{(s-r)/\nu^{2/3}}^{U, X, Z})]\bigg|_{Z =  Z_{r/\nu^{2/3}}^{U, X,Y}} \mu_\nu^{U,X}(dY')  - \bar{B}_0(U_{k\delta}, X)||_{\Hh_0} \,\bigg \vert_{(U,X,Y) = (U_{k\delta}, \tilde{X}_{k\delta}^\nu, \tilde{Y}_{k\delta}^\nu)}\\
        &\quad= || \int_{\mathcal{H}_{\neq}} \E[ B_0(Z_{(s-r)/\nu^{2/3}}^{U, X, Z})]\bigg|_{Z =  Z_{r/\nu^{2/3}}^{U, X,Y}}\mu_\nu^{U,X}(dY')   - \int_{\mathcal{H}_{\neq}}B_0(Y) \mu_\nu^{U,X}(dY') ||_{\Hh_0} \,\bigg \vert_{(U,X,Y) = (U_{k\delta}, \tilde{X}_{k\delta}^\nu, \tilde{Y}_{k\delta}^\nu)}\\
        &\quad= || \int_{\mathcal{H}_{\neq}} \E[ B_0(Z_{(s-r)/\nu^{2/3}}^{U, X, Z})  - B_0(Z_{(s-r)/\nu^{2/3}}^{U, X,Y'})]\bigg|_{Z =  Z_{r/\nu^{2/3}}^{X,Y}} \mu_\nu^{U,X}(dY') ||_{\Hh_0} \,\bigg \vert_{(U,X,Y) = (U_{k\delta}, \tilde{X}_{k\delta}^\nu, \tilde{Y}_{k\delta}^\nu)}.
    \end{split}
\end{equation}
To \eqref{internal_expectiation_for_L_two}, we apply Minkowski's integral inequality, \eqref{split_nonlin}, and Lemma \ref{nonlinear_deterministic} to find:
\begin{equation}\label{internal_expectiation_for_L_three}
    \begin{split}
        &|| \int_{\mathcal{H}_{\neq}} \E[  B_0(Z_{(s-r)/\nu^{2/3}}^{U, X, Z})  - b_0(Z_{(s-r)/\nu^{2/3}}^{U, X,Y'}))]\bigg|_{Z =  Z_{r/\nu^{2/3}}^{U, X,Y}} \mu_\nu^{U,X}(dY') ||_{\Hh_0}\\
        &\quad \leq C \int_{\mathcal{H}_{\neq}} \E[ \Ee_{\neq}(Z_{(s-r)/\nu^{2/3}}^{U, X, Z} - Z_{(s-r)/\nu^{2/3}}^{U, X, Y'})]^{1/2} \nu^{1/2} \E[\D_{\neq}(Z_{(s-r)/\nu^{2/3}}^{U, X, Y'})]^{1/2} \bigg|_{Z =  Z_{r/\nu^{2/3}}^{U, X,Y}} \\
        &\quad\quad\quad\quad \E[\Ee_{\neq}(Z_{(s-r)/\nu^{2/3}}^{U, X, Z})]^{1/2} \nu^{1/2} \E[\D_{\neq}(Z_{(s-r)/\nu^{2/3}}^{U, X, Z} - Z_{(s-r)/\nu^{2/3}}^{U, X, Y'})]^{1/2}\bigg|_{Z =  Z_{r/\nu^{2/3}}^{U, X,Y}} \mu_\nu^{U,X}(dY')\\
        &\eqqcolon N^{(k)}(s,r),
    \end{split}
\end{equation}
where we understand $(U,X,Y) = (U_{k\delta}, \tilde{X}_{k\delta}^\nu, \tilde{Y}_{k\delta}^\nu)$ as being given. We collect some of the previous work together, applying \eqref{internal_expectiation_for_L_three} and \eqref{cauch_schwarz_for_L} to \eqref{k_by_k_expectation}. We additionally use the Markov property and the fact that with $\mathcal{F}_{k\delta}$ given, $\hat{Y}_{r+k\delta}^\nu$ is distributed like $Z_{r/\nu^{2/3}}^{U,X,Y}$ with $(U,X,Y) = (U_{k\delta}, \tilde{X}_{k\delta}^\nu, \tilde{Y}_{k\delta}^\nu)$ fixed:
\begin{equation}\label{consolidated_est_on_J}
\begin{split}
   J_1^{(k)} &\leq C\int_0^\delta \E\biggl[1_{k\delta \leq \bar{\tau}} \E\biggl[\int_r^\delta N^{(k)}(s,r) ds||  B_0(Z_{r/\nu^{2/3}}^{X,Y}) - \bar{B}_0(U, X))||_{\Hh_0}\biggr] \big\vert_{(U,X,Y) = (U_{k\delta}, \tilde{X}_{k\delta}^\nu, \tilde{Y}_{k\delta}^\nu)}\biggr] dr.
\end{split}
\end{equation}
We see then that the next term to estimate is the time integral of $N^{(k)}(s,r)$ in $s$ from $r$ to $\delta$. We apply H\"older's inequality, and Propositions \ref{bounds_on_frozen} and \ref{second_bounds_on_frozen}. Additionally, we implicitly fix $(X, Y, Z) = (\tilde{X}_{k\delta}^\nu, \tilde{Y}_{k\delta}^\nu, Z_{r/\nu^{2/3}}^{U,X,Y})$ and use that $\Ee_0(\tilde{X}_{k\delta}^{\nu}) < c_* \nu^{-2\alpha'}$ in the event $\{k\delta \leq \bar{\tau}\}$ (which is $\mathcal{F}_{k\delta}$ measurable):
\begin{equation}\label{estimate_on_N}
    \begin{split}
        \int_r^\delta& N^{(k)}(s,r) ds\\ &\leq C \nu^{1/6}\int_{\mathcal{H}_{\neq}} \Ee_{\neq}(Z - Y')^{1/2} \int_r^\delta e^{-\delta_*(s-r)/2\nu^{2/3}} \nu^{1/3} \E[\D_{\neq}(Z_{(s-r)/\nu^{2/3}}^{U. X, Y'})]^{1/2} ds \mu_\nu^{U,X}(dY')\\
        &\quad+ C \nu^{1/6}\int_{\mathcal{H}_{\neq}} (\Ee_{\neq}(Z) + ||\Psi||^2)^{1/2} \int_r^\delta \nu^{1/3}\E[\D_{\neq}(Z_{(s-r)/\nu^{2/3}}^{U, X, Z'} -Z_{(s-r)/\nu^{2/3}}^{U, X, Y'})]^{1/2} ds \mu_\nu^{U,X}(dY')\\
        &\leq C\nu^{1/6} \int_{\mathcal{H}_{\neq}} \Ee_{\neq}(Z - Y')^{1/2} (\delta-r)^{1/2}\\
        &\quad \quad \quad \quad \quad \times\left(\nu^{1/3}\int_r^\delta e^{-\delta_*(s-r)/\nu^{2/3})} \E[\D_{\neq}(Z_{(s-r)/\nu^{2/3}}^{X, Y'})])] ds\right)^{1/2} \mu_\nu^{U,X}(dY')\\
        &\quad + C\nu^{1/6} (\delta - r)^{1/2} (\Ee_{\neq}(Z) + ||\Psi||^2)^{1/2}\\
        &\quad\quad \times\int_{\mathcal{H}_{\neq}}\left(\nu^{1/3}\int_r^\delta \E[\D_{\neq}(Z_{(s-r)/\nu^{2/3}}^{X, Z'} -Z_{(s-r)/\nu^{2/3}}^{X, Y'})] ds\right)^{1/2} \mu_\nu^{U,X}(dY')\\
        &\leq C \nu^{1/2} (\delta-r)^{1/2} \biggl( \int_{\mathcal{H}_{\neq}} \Ee_{\neq}(Z - Y')^{1/2}(\Ee_{\neq}(Y') + ||\Psi||^2)^{1/2}\\
        &\quad\quad\quad\quad \quad\quad\quad\quad\quad\quad+ \Ee_{\neq}(Z-Y')^{1/2}(\Ee_{\neq}(Z) + ||\Psi||^2)^{1/2} \mu_\nu^{U,X}(dY') \biggr).
    \end{split}
\end{equation}
Applying \eqref{estimate_on_N} to \eqref{consolidated_est_on_J}, we find
\begin{equation}
    \begin{split}
         J_1^{(k)}&\leq C\nu^{1/2} \int_0^\delta (\delta-r)^{1/2} \E \biggl[1_{k\delta \leq \bar{\tau}} \E\biggl[\biggl( \int_{\mathcal{H}_{\neq}} \Ee_{\neq}(Z_{r/\nu^{2/3}}^{X,Y} - Y')^{1/2}(\Ee_{\neq}(Y') + ||\Psi||^2)^{1/2}\\
        &\quad \quad\quad\quad\quad\quad + \Ee_{\neq}(Z_{r/\nu^{2/3}}^{X,Y}-Y')^{1/2}(\Ee_{\neq}(Z_{r/\nu^{2/3}}^{X,Y}) + ||\Psi||^2)^{1/2} \mu_\nu^{U,X}(dY') \biggr)\\
        &\quad\quad\quad\quad\quad\quad\quad ||  B_0(Z_{r/\nu^{2/3}}^{X,Y}) - \bar{B}_0(U,X) ||_{\Hh_0}\biggr] \vert_{(U,X,Y) = (U_{k\delta}, \tilde{X}_{k\delta}^\nu, \tilde{Y}_{k\delta}^\nu)} \biggr] dr.
    \end{split}
\end{equation}
Note that with $X = \tilde{X}_{k\delta}^\nu$ and $U = U_{k\delta}$ fixed, and applying $\{k\delta \leq \bar{\tau}\}$, we have by Proposition \ref{proposition: ergodic_to_bounded}, $||\bar{B}_0(U, X)||_{\mathcal{H}_0} \leq C \nu^{1/6} ||\Psi||^2$. Thus by the triangle inequality and Proposition \ref{proposition: integral_against_inv_measure},
\begin{equation}
    \begin{split}
     J_1^{(k)}&\leq C\nu^{1/2} \int_0^\delta (\delta-r)^{1/2} \E \biggl[1_{k\delta \leq \bar{\tau}} \E\biggl[\biggl( \int_{\mathcal{H}_{\neq}} (\Ee_{\neq}(Z_{r/\nu^{2/3}}^{U, X,Y})^{1/2} +\Ee_{\neq}(Y')^{1/2})(\Ee_{\neq}(Y') + ||\Psi||^2)^{1/2}\\
        &\quad \quad\quad\quad\quad\quad + (\Ee_{\neq}(Z_{r/\nu^{2/3}}^{U, X,Y})^{1/2}+\Ee_{\neq}(Y')^{1/2})(\Ee_{\neq}(Z_{r/\nu^{2/3}}^{U, X,Y}) + ||\Psi||^2)^{1/2} \mu_\nu^{U,X}(dY') \biggr)\\
        &\quad\quad\quad\quad\quad\quad\quad \times\biggl(\Ee_{\neq}(Z_{r/\nu^{2/3}}^{U, X,Y})^{1/2}\nu^{1/2}\D_{\neq}^{1/2}(Z_{r/\nu^{2/3}}^{U, X,Y}) + \nu^{1/6}||\Psi||^2\biggr)\biggr] \vert_{(U,X,Y) = (U_{k\delta}, \tilde{X}_{k\delta}^\nu, \tilde{Y}_{k\delta}^\nu)} \biggr] dr.\\
        &\leq C\nu^{1/2}\E\biggl[1_{k\delta\leq \bar{\tau}} \int_0^\delta (\delta-r)^{1/2} \E\biggl[\biggl(\Ee_{\neq}(Z_{r/\nu^{2/3}}^{U, X,Y})^{1/2}||\Psi|| + ||\Psi||^2+ \Ee_{\neq}(Z_{r/\nu^{2/3}}^{U, X,Y}) \biggr)\\
        &\quad\quad\quad\quad\quad\quad\quad\quad\quad \times\biggl(\Ee_{\neq}(Z_{r/\nu^{2/3}}^{U, X,Y})^{1/2}\nu^{1/2}\D_{\neq}^{1/2}(Z_{r/\nu^{2/3}}^{U, X,Y}) + \nu^{1/6}||\Psi||^2\biggr)\biggr] \vert_{(U,X,Y) = (U_{k\delta}, \tilde{X}_{k\delta}^\nu, \tilde{Y}_{k\delta}^\nu)} \biggr] dr.
    \end{split}
\end{equation}
Next, we apply H\"older's inequality, Proposition \ref{bounds_on_frozen}, and the condition $\{k\delta \leq \bar{\tau}\}$:
\begin{equation}
    \begin{split}
    J_1^{(k)} &\leq C(||\Psi||)\nu^{2/3}\E\biggl[1_{k\delta \leq \tau} \biggl(\int_0^\delta (\delta-r)\E[\Ee_{\neq}(Z_{r/\nu^{2/3}}^{U, X,Y}) + \Ee_{\neq}(Z_{r/\nu^{2/3}}^{U, X,Y})^2 + \Ee_{\neq}(Z_{r/\nu^{2/3}}^{U, X,Y})^{3}] \vert_{(U,X,Y) = (U_{k\delta}, \tilde{X}_{k\delta}^\nu, \tilde{Y}_{k\delta}^\nu)} \biggr)^{1/2}\\
    &\quad\quad\quad\quad\quad\quad\quad\quad\quad \times\left(\int_0^\delta \nu^{2/3}\E[\D_{\neq}(Z_{r/\nu^{2/3}}^{U, X,Y})]\vert_{(U,X,Y) = (U_{k\delta}, \tilde{X}_{k\delta}^\nu, \tilde{Y}_{k\delta}^\nu)} + ||\Psi||^4 ds \right)^{1/2}  \biggr]\\
    & \leq C(||\Psi||)\nu^{2/3} (1+ \nu^{-3\alpha'}) \delta (\nu^{2/3} + ||\Psi||^2 \delta + ||\Psi||^4 \delta)^{1/2}.
    \end{split}
\end{equation}
Consolidating and using $\delta > \nu^{2/3 - \alpha'}$, we find
\begin{equation}\label{estimate_on_J1_k}
    J_1^{(k)} \leq C(||\Psi||)\nu^{2/3 - 3\alpha'} \delta^{3/2}.
\end{equation}
Thus applying \eqref{estimate_on_J1_k} to \eqref{J1_decomp_2} gives
\begin{equation}\label{J_1_final}
    \begin{split}
    J_1 \leq C \delta^{-1}\nu^{-1/6} \nu^{-\alpha'} \nu^{1/3 - 3 \alpha'/2} \delta^{3/4}\leq C \nu^{1/6 - 3\alpha'} \delta^{-1/4}.
    \end{split}
\end{equation}
We now return to estimating $\E \sup_{t \in [0,T \wedge \tau \wedge \sigma]}||\hat{X}_t - \bar{X}_t||_{\mathcal{H}_0^{-1/2}}^2$. Continuing from \eqref{intermediate_bound_2}, we apply \eqref{J_2_bound} and \eqref{J_1_final} to obtain
\begin{equation}
    \begin{split}
        \E \sup_{t \in [0,T \wedge \tau\wedge \sigma]}||\hat{X}_t - \bar{X}_t||_{\mathcal{H}_0^{-1/2}}^{1/2} &\leq C(||\Psi||, T, a)\left(\delta^{a/2}\nu^{-a/6-3\alpha'}+ \E\sup_{t \in [0,T \wedge \tau\wedge \sigma]} I_1^b(t)\right)\\
     &\leq
     C(||\Psi||, T, a)\left( \delta^{a/2}\nu^{-a/6-3\alpha'} + \delta^{-1/4} \nu^{1/6 - 3\alpha'}\right).
    \end{split}
\end{equation}
This completes the proof of the proposition.
\end{proof}

\subsection{Explicit Rate of Convergence in Theorem \ref{main_theorem}}\label{subsection: rate_of_convergence}

We now combine our previous work into a proof of Theorem \ref{main_theorem}. In particular, the key will be to show that for an appropriate choice of constants as in the theorem statement, the following expectation bound holds:
\begin{equation}\label{expectation_bound}
    \begin{split}
        \E \left( \sup_{t \in [0,T\wedge \tau \wedge \sigma]}||X_t^\nu - \bar{X}_t^\nu||_{H^{-\theta}}^2\right) \leq C_1\nu^{\theta\frac{a}{1+a} \frac{1}{6}}.
    \end{split}
\end{equation}
We let $c_0 = \frac{1}{2}\delta_0$. Then we let $c_1$ and $c_2$ be as determined by the Gaussian tail bound \eqref{eqn: tail_est_on_sigma}. We will use $\nu_0$ as defined by \eqref{eqn: def_of_nu_0} and take $\nu < \nu_0$. Note that in the statement of Theorem \ref{main_theorem}, we are looking for a constant $C_0$ bounding quantities like $||X_0^\nu||_{\Hh} \leq C_0 \nu^{-\alpha'/2}$. The norm involved is the true $\Hh$ norm, and we recall that $||f||_{\Hh} \approx ||f||_{\Hh_0}$ for $f\in \Hh_0$ and $||f||_{\Hh} \approx ||f||_{\Hh_{\neq}}$ for $f\in \Hh_{\neq}$. We would like to use the constant $C_0$ appearing in the proofs of the propositions throughout this paper, but this constant is involved in bounds like $||X_0^\nu||_{\Hh_0} \leq C_0 \nu^{-\alpha'/2}$. We rename this constant $\hat{C}_0$. For the purposes of the theorem, we let $C_0$ be the minimal constant so that $||f||_{\Hh} \leq C_0  \implies ||f||_{\Hh_0} \leq \hat{C}_0$ for $f \in \Hh_0$ and $||f||_{\Hh} \leq C_0  \implies ||f||_{\Hh_{\neq}} \leq \hat{C}_0$ for $f \in \Hh_{\neq}$. The final constant $C_1$ will be determined in the proof of \eqref{expectation_bound}.

By the triangle inequality,
\begin{equation}\label{decomp_of_quant_est}
    \begin{split}
    \E (\sup_{t \in [0,T \wedge \tau \wedge \sigma]} ||X_t^\nu - \bar{X}_t^\nu||_{H^{-\theta}}^2) & \leq \E (\sup_{t \in [0,T \wedge \tau \wedge \sigma]} ||X_t^\nu - \tilde{X}_t^\nu||_{H^{-\theta}}^2) + \E (\sup_{t \in [0,T \wedge \tau \wedge \sigma]} ||\tilde{X}_t^\nu - \hat{X}_t^\nu||_{H^{-\theta}}^2)\\
    &\quad+ \E (\sup_{t \in [0,T \wedge \tau \wedge \sigma]} ||\hat{X}_t^\nu - \bar{X}_t^\nu||_{H^{-\theta}}^2).
    \end{split}
\end{equation}
The first two terms of \eqref{decomp_of_quant_est} are controlled by Propositions \ref{approx_with_lin} and \ref{lin_to_aux_prop}:
\begin{equation}\label{lin_and_aux_final_decay}
    \begin{split}
        \E (\sup_{t \in [0,T \wedge \tau\wedge \sigma]} ||X_t^\nu - \tilde{X}_t^\nu||_{H^{-\theta}}^2) + \E &(\sup_{t \in [0,T \wedge \tau \wedge \sigma]} ||\tilde{X}_t^\nu - \hat{X}_t^\nu||^2_{H^{-\theta}}) \\&\leq \E (\sup_{t \in [0,T \wedge \tau \wedge \sigma]} ||X_t^\nu - \tilde{X}_t^\nu||_{\mathcal{H}_0}^2) + \E (\sup_{t \in [0,T \wedge \tau \wedge \sigma]} ||\tilde{X}_t^\nu - \hat{X}_t^\nu||_{\mathcal{H}_0}^2)\\
        &\leq C\left(\nu^{\alpha - 2\alpha'} + \delta^{a} \nu^{-a/3}\right).
    \end{split}
\end{equation}
The remaining term in \eqref{decomp_of_quant_est} is controlled by Propositions \ref{energy_bounds_for_auxiliary} and \ref{key_prop}, Theorem \ref{theorem: averaged_well_posed}, and interpolation of Sobolev norms:
\begin{equation}\label{hat_final_decay}
    \begin{split}
        \E (\sup_{t \in [0,T \wedge \tau\wedge \sigma]} ||\hat{X}_t^\nu - \bar{X}_t^\nu||_{H^{-\theta}}^2) &\leq C\E (\sup_{t \in [0,T \wedge \tau\wedge \sigma]} ||\hat{X}_t^\nu - \bar{X}_t^\nu||_{H^{-\theta}}^{4\theta}||\hat{X}_t^\nu - \bar{X}_t^\nu||_{L^2}^{2-4\theta})\\
        &\leq C\E (\sup_{t \in [0,T \wedge \tau\wedge \sigma]} ||\hat{X}_t^\nu - \bar{X}_t^\nu||_{\Hh_0^{-\theta}}^2)^{2\theta} \nu^{-2(1-2\theta)\alpha'}\\
        &\leq C(||\Psi||, T, a)(\delta^{a/2}\nu^{-a/6-3\alpha'} + \delta^{-1/4} \nu^{1/6 - 3\alpha'} )^{2\theta} \nu^{-2(1-2\theta)\alpha'}\\
        &\leq C(\delta^{a/2} \nu^{-a/6} + \delta^{-1/4} \nu^{1/6})^{2\theta} \nu^{-2\alpha'(1+\theta)}
    \end{split}
\end{equation}
We therefore wish to optimize in $\delta$ the expression
\begin{equation}\label{equation_to_optimize}
    \begin{split}
        \delta^{a/2}\nu^{-a/6'} + \delta^{-1/4} \nu^{1/6}.
    \end{split}
\end{equation}
Letting $\delta = \nu^q$, we find the optimal value for $q$ as
$$q_* \coloneqq \frac{2}{3}\frac{1 + a}{1+2a}.$$
Note that $q* > 1/3$ for any $a \in (0,1)$, which verifies $\delta < \nu^{1/3}$. Furthermore, by the choice of $\alpha'$, we verify $\delta > \nu^{2/3-\alpha'}$. Thus by \eqref{lin_and_aux_final_decay} and \eqref{hat_final_decay}, we see that the decay rate is
\begin{equation}
    \max\{\nu^{\alpha-2\alpha'}, \nu^{\frac{a}{3}\frac{1 + a}{1+2a} - \frac{a}{6}}, \nu^{2\theta\left(\frac{a}{3}\frac{1 + a}{1+2a} - \frac{a}{6}\right) - 3\alpha'}\}.
\end{equation}
Note that $ \nu^{2\theta\left(\frac{a}{3}\frac{1 + a}{1+2a} - \frac{a}{6}\right) - 3\alpha'} =  \nu^{\frac{\theta}{3}\frac{a}{1+2a} - 3\alpha'}$
Now we recall that $\alpha' = \frac{\theta}{36}\min\{ \frac{a}{1+2a}\frac{1}{6}, \beta, \alpha\}$. Thus we obtain the decay rate of
\begin{equation}\label{final_exp_error}
     \E \sup_{t \in [0,T \wedge \tau \sigma]} ||X_t^\nu - \bar{X}_t^\nu||_{H^{-\theta}}^2 \leq C(||\Psi||, T, \alpha,  \theta, a) \nu^{\frac{\theta}{6} \frac{a}{1+2a}}.
\end{equation}

We now prove \eqref{explicit_error_rate}. We first split into the two events $\{\sigma \wedge \tau  < T\}$ and $\{ \sigma \wedge \tau
 \geq T\}$:
 \begin{equation}\label{eqn: sigma_split}
    \begin{split}
        \E(\sup_{t \in [0,T]} ||X_t^\nu - \bar{X}_t^\nu||_{H^{-\theta}})&\leq  \E(\sup_{t \in [0,T]} ||X_t^\nu - \bar{X}_t^\nu||_{H^{-\theta}}, 1_{\tau \wedge \sigma\geq T})+ \E(\sup_{t \in [0,T]} ||X_t^\nu - \bar{X}_t^\nu||_{H^{-\theta}} 1_{\tau \wedge \sigma < T})\\
        &\leq \E(\sup_{t \in [0,T \wedge \tau \wedge \sigma]} ||X_t^\nu - \bar{X}_t^\nu||_{H^{-\theta}})\\
        &\quad + \Pb(\sigma\wedge\tau < T)^{1/2}\sqrt{\E(\sup_{t \in [0,T]} ||X_t^\nu||_{H^{-\theta}}^2) + \E(\sup_{t \in [0,T]} ||\bar{X}_t^\nu||_{H^{-\theta}}^2)}.
    \end{split}
\end{equation}
By \eqref{continuity_prop}, $\delta > \nu^{2/3}$, and the usage of \eqref{use_of_nu_2} in the definition of $\nu_*$:
\begin{equation}\label{prob_bound}
    \Pb(\tau < T) \leq C(T) \nu^{-2/3} \exp(-\nu^{-\alpha'/2}) \leq C(T) \nu^{2/3}.
\end{equation}
Similarly, the usage of $\nu_4$ in the definition of $\nu_0$ and \eqref{eqn: tail_est_on_sigma} gives for $\nu < \nu_0$:
\begin{equation}\label{eqn: sigma_est}
    \Pb(\sigma < T) \leq C \exp( - c_2 c_0^2 ||\Phi||^{-2} \nu^{-2/3} T^{-2}) \leq C \nu^{2/3}
\end{equation}
Then by Proposition \ref{proposition: crude_bound}, Theorem \ref{theorem: averaged_well_posed}, \eqref{prob_bound}, the union bound, and \eqref{eqn: sigma_est}:
\begin{equation}
    \begin{split}
        \Pb(\tau \wedge \sigma< T)^{1/2}&\sqrt{\E(\sup_{t \in [0,T]} ||X_t^\nu||_{H^{-\theta}}^2) + \E(\sup_{t \in [0,T]} ||\bar{X}_t^\nu||_{H^{-\theta}}^2)}\\
        & \leq C  \left(\Pb(\tau <T) + \Pb(\sigma< T)\right)^{1/2}\nu^{1/3}\left(\nu^{-2/3 +2(\alpha + \beta)} + \nu^{-2\alpha'}\right)^{1/2}\\
        &\leq C \nu^{\alpha}.
    \end{split}
\end{equation}
Thus we have
 \begin{equation}
    \begin{split}
        \E(\sup_{t \in [0,T]} ||X_t^\nu - \bar{X}_t^\nu||_{H^{-\theta}})
        &\leq \E(\sup_{t \in [0,T \wedge \tau \wedge \sigma]} ||X_t^\nu - \bar{X}_t^\nu||_{H^{-\theta}}) + C \nu^\alpha.
    \end{split}
\end{equation}
Applying \eqref{final_exp_error} and Cauchy-Schwarz yields
\begin{equation}
        \E(\sup_{t \in [0,T]} ||X_t^\nu - \bar{X}_t^\nu||_{H^{-\theta}}) \leq \sqrt{\E(\sup_{t \in [0,T \wedge \tau \wedge \sigma]} ||X_t^\nu - \bar{X}_t^\nu||_{H^{-\theta}}^2)} + C \nu^\alpha \leq C(||\Psi||, T, \alpha, \theta, a) \nu^{\frac{\theta}{12} \frac{a}{1+2a}}
\end{equation}
as desired.

\appendix
\section{Appendix: Well-Posedness}\label{appendix: a}

In this section, we briefly address the (local) well-posedness of \eqref{original_couette}, \eqref{fast_slow_lin}, \eqref{fast_aux_process}, and \eqref{slow_aux_process}. Simultaneously, this will verify the validity of the infinite dimensional It\^o formula used throughout the paper, as well as the mild formulations which we have occasionally used. We will primarily just be sketching the main results.

We introduce the space $\nabla \mathcal{H}$ as the space of $L^2(\mathbb{T} \times [-1,1])$ functions $f$ such that $||\nabla f||_{\Hh} < \infty$. For the purpose of Theorem \ref{well_posed_full}, we re-define $\mathcal{D}$ by multiplying the original definition by $\nu^{1-\gamma}$. We are now ready to begin.
\begin{theorem}\label{well_posed_full}
    Let $\omega(0) \in \mathcal{H}$, $T > 0$, and suppose that \eqref{noise_assumptions} hold. Then there exists a unique $\omega \in L^2(\Omega ; C([0,T] ; L^2)) \cap L^2(\Omega; L^2([0,T] ; H^1))$ such that for any $t \in [0,T]$ and $\phi \in D(-\Delta)$, the following holds $\Pb$-a.s:
    \begin{equation}\begin{split}
                \langle \omega, \phi\rangle_{L^2} &= \langle \omega(0), \phi\rangle_{L^2} + \int_0^t \langle \omega(s), (\nu\Delta-\mathcal{U}\partial_x + \partial_x \Delta^{-1} \mathcal{U}'')\phi\rangle_{L^2} ds\\ &\quad- \int_0^t \langle  \omega(s), (\nabla^\perp \Delta^{-1} \omega(s))\cdot \nabla\phi \rangle_{L^2} ds + \int_0^t \langle \Psi dW_s, \omega(s)\rangle_{L^2}.
    \end{split}
    \end{equation}
    Furthermore, if $||\omega_{ \neq}(0)||_{\mathcal{H}_{\neq}} \leq c_* \nu^{1/2 + \alpha - \alpha'/2}$ and $||\omega_{ 0}(0)||_{\mathcal{H}_0} \leq c_* \nu^{1/2 + \beta - \alpha'}$, then there exists a positive stopping time $\tau$ such that $\omega \in  L^2(\Omega ; C([0,T \wedge \tau] ; \mathcal{H})) \cap L^2(\Omega; L^2([0,T \wedge \tau] ; \nabla \mathcal{H})$.
\end{theorem}
\begin{proof}
    We begin by introducing the following Galerkin approximations. Recall the Fourier basis $\{e_{k,j}\}_{(k,j) \in (\Z \setminus \{0\} \times \Z_+}$ with $e_{k,j}(x,y) = e^{ikx} \sin(\frac{\pi}{2}j(y+1))$. This forms an orthogonal basis of $H_0^2(\T \times [-1,1])$. For $N > 0$, we will let 
$$H_N \coloneqq \mathrm{span}\{ e_{k,j} : |k| \leq N, \, j \leq N\}.$$
Let $P_N$ be the projection to $H_N$. Introduce $\omega_N$ as the solution to the following problem:
For all $\phi \in H_N$,
\begin{equation}\label{fin_dim_syst}
    d\langle \omega_N(t), \phi \rangle_H = d\langle (\nu\Delta-\mathcal{U}\partial_x + \partial_x \Delta^{-1} \mathcal{U}'') \omega_N , \phi \rangle + d\langle P_N(\nabla^{\perp} \Delta^{-1} \omega_N \cdot \nabla \omega_N), \phi \rangle + \langle P_N \Psi dW_s, \phi \rangle.
\end{equation}
with initial data $\omega_N(0) = \sum_{j\leq N, |k| \leq N} \langle\omega_{k,j}(0), e_{k,j} \rangle_H e_{k,j}$. Then as the $e_{k,j}$ diagonalize the Laplacian on $\T \times [-1,1]$, we see that \eqref{fin_dim_syst} is a stochastic differential equation on $\R^M$ for $M$ sufficiently large with locally Lipschitz coefficients. Hence a unique strong solution exists locally in time. Introduce the stopping time $\tau_N^{(K)} \coloneqq \inf\{||\omega_N(t)||_{L^2}^2 \geq K\} \wedge T$. We apply It\^o's formula to $||\omega_N(t \wedge \tau_N^{(N)})||_{L^2}^4$:
\begin{equation}
    \begin{split}
        ||\omega_N(t \wedge \tau_N^{(N)})||_{L^2}^4 &= ||\omega_N(0)||_{L^2}^4 + 4\int_0^{t\wedge\tau_N^{(K)}}||\omega_N(s)||_{L^2}^2\mathrm{Re}\langle \nu\Delta \omega_N(s) - \mathcal{U}\partial_x \omega_N(s)+ \mathcal{U}''\partial_x \Delta^{-1} \omega_N(s), \omega_N(s) \rangle ds\\
        &\quad- 4\int_0^{t \wedge \tau_N^{(K)}} ||\omega_N(s)||_{L^2}^2 \mathrm{Re}\langle P_N(\nabla^{\perp} \Delta^{-1} \omega_N \cdot \nabla \omega_N)(s), \omega_N(s) \rangle ds\\
        &\quad+ 2\nu^{4/3 + 2\alpha} ||P_N\Psi||^2_{HS(L^2)} \int_0^{t \wedge \tau_{N}^{(K)}} ||\omega_N(s)||^2_{L^2}ds + 4\nu^{4/3 + 2\alpha}\int_0^t ||\Psi^* \omega_N(s)||_{L^2}^2 ds\\
        &\quad+ 4 \nu^{2/3+\alpha}\int_0^{t \wedge \tau_N^{(K)}}||\omega_N(s)||_{L^2}^2 \langle P_N \Psi dW_S, \omega_N(s) \rangle.
    \end{split}
\end{equation}
Then by integration by parts, the boundary conditions, and elliptic regularity, we find
\begin{equation}
    \begin{split}
        ||\omega_N(t \wedge \tau_N^{(N)})||_{L^2}^4&+ \nu\int_0^{t \wedge \tau_N^{(K)}} ||\omega_N(s)||^2_{L^2}||\nabla \omega_N(s)||_{L^2}^2ds \leq ||\omega(0)||_{L^2}^4 + C\int_0^{t\wedge\tau_N^{(K)}}||\omega_N(s)||_{L^2}^4||\mathcal{U}'''||_{L^\infty} ds \\
        &+ C\nu^{4/3 + 2\alpha} ||\Psi||^2_{HS(L^2)} \int_0^{t \wedge \tau_{N}^{(K)}} ||\omega_N(s)||_{L^2}^2 ds + 4 \nu^{2/3+\alpha}\int_0^{t \wedge \tau_N^{(K)}}||\omega_N(s)||_{L^2}^2 \langle P_N \Psi dW_S, \omega_N(s) \rangle.
    \end{split}
\end{equation}
Using Young's product inequality,
\begin{equation}
    \begin{split}
        ||\omega_N(t \wedge \tau_N^{(N)})||_{L^2}^4&+ \nu\int_0^{t \wedge \tau_N^{(K)}} ||\omega_N(s)||^2_{L^2}||\nabla \omega_N(s)||_{L^2}^2ds \leq ||\omega(0)||_{L^2}^4 + C\int_0^{t\wedge\tau_N^{(K)}}||\omega_N(s)||_{L^2}^4||\mathcal{U}'''||_{L^\infty} ds\\
        &\quad+ C\nu^{4/3 + 2\alpha} ||\Psi||^4 t \wedge \tau_{N}^{(K)}+ \frac{1}{2} C\nu^{4/3+2\alpha} \int_0^{t \wedge \tau_{N}^{(K)}} ||\omega_N(s)||_{L^2}^4 ds\\
        &\quad+ 4 \nu^{2/3+\alpha}\int_0^{t \wedge \tau_N^{(K)}}||\omega_N(s)||_{L^2}^2 \langle P_N \Psi dW_S, \omega_N(s) \rangle.
    \end{split}
\end{equation}
Define the function $G(t) \coloneqq \exp(-\int_0^t \frac{C}{2}\nu^{4/3+2\alpha}+||\mathcal{U}'''(s)||_{L^\infty} ds)$, and let $F_t$ denote the (local) martingale
$$F_t \coloneqq \int_0^t ||\omega_N(s)||_{L^2}^2\langle P_N \Psi dW_S, \omega_N(s) \rangle,$$
so that $F_{t \wedge \tau_N^{(K)}}$ is a true martingale. Then
\begin{equation}\label{gronwall_type_galerkin}
    \begin{split}
        G(t) &||\omega_N(t \wedge \tau_N^{(K)})||_{L^2}^4  + \nu\int_0^{t \wedge \tau_N^{(K)}} G(s)||\omega_N(s)||_{L^2}^2 ||\nabla \omega_N(s)||_{L^2}^2ds\\
        &\leq ||\omega(0)||_{L^2}^ 4+  C \nu^{4/3 + 2\alpha} ||\Psi||^4\int_0^{t \wedge \tau_N^{(K)}} G(s) ds +C\nu^{2/3+\alpha}\int_0^{t \wedge \tau_N^{(K)}}G(s) dF_s.
    \end{split}
\end{equation}
Now we apply Burkholder-Davis-Gundy and Young's inequality to find
\begin{equation}\label{BDG_Galerkin}
    \begin{split}
        \nu^{2/3+\alpha} \E \sup_{t \in [0,T]} \int_0^{t \wedge \tau_N^{(K)}}G(s) dF_s & \leq \nu^{2/3+\alpha}C \E \left(\int_0^{t \wedge \tau_N^{(K)}} G(s)^2 ||\Psi||^2 ||\omega_N(s)||_{L^2}^4ds\right)^{1/2}\\
         &\leq \nu^{2/3+\alpha}\E \sup_{s \in [0,T \wedge \tau_N^{(K)}]}G(s)^{1/2}||\omega_N(s)||_{L^2}^2\left(\int_0^{t \wedge \tau_N^{(K)}} G(s) ||\Psi||^2 ds\right)^{1/2}\\
         &\leq \frac{1}{2} \E\sup_{s \in [0,t \wedge \tau_N^{(K)}]} G(s)||\omega_N(s)||_{L^2}^4 + C\nu^{4/3+2\alpha} ||\Psi||^2 \E \int_0^{t \wedge \tau_N^{(K)}} G(s) ds. 
    \end{split}
\end{equation}
To \eqref{gronwall_type_galerkin}, we take supremum in time, expectation, and use \eqref{BDG_Galerkin} to find
\begin{equation}\label{eqn_partial_bound_well}
    \begin{split}
        \E \sup_{t \in [0,T\wedge\tau_N^{(K)}]} G(t)||\omega_N(t)||_{L^2}^4\leq C||\omega(0)||_{L^2}^4 + C\nu^{4/3+2\alpha}  \E \int_0^{T \wedge \tau_N^{(K)}}  G(s)||\Psi||^2(1+ ||\Psi||^2) ds.
    \end{split}
\end{equation}
Now 
\begin{equation}
    \begin{split}
       \E \int_0^{T \wedge \tau_N^{(K)}} G(s) ds &\leq \int_0^T \E\left(\exp( \int_0^s ||\mathcal{U}'''(r)||_{L^\infty} ds)\right) \leq \int_0^T\E\left(\exp( \int_0^s ||\mathcal{W}(r)||_{H^3} ds)\right) < \infty,
    \end{split}
\end{equation}
with a similar bound holding for $\E(G(T)^{-1})$. We then use Cauchy-Schwarz to find
\begin{equation}\label{final_L2_Galerkin}
    \begin{split}
         \E \sup_{t \in [0,T\wedge\tau_N^{(K)}]} ||\omega_N(t)||^2&= \E \sup_{t \in [0,T\wedge\tau_N^{(K)}]} G(t)^{1/2} G(t)^{-1/2}||\omega_N(t)||^2\\
         &\leq \E\left(G(T)^{-1}\right) \left(\E\sup_{t \in [0,T\wedge\tau_N^{(K)}]} G(t) ||\omega_N(t)||_{L^2}^4\right)^{1/2}\\
         &\leq C( ||\omega(0)||_{L^2}, ||\Psi||, ||\Phi||, T, \nu) < \infty,
    \end{split}
\end{equation}
where notably the constant in \eqref{final_L2_Galerkin} is independent of $K$ and $N$. We then take $K \to \infty$ and find that $\omega_N \in L^2(\Omega ; C([0,T] ; L^2))$. Now with It\^o's formula applied to $||\omega_N(t)||_{L^2}^2$, it is not hard to show
\begin{equation}
\begin{split}
        \E \int_0^T ||\nabla \omega_N(s)||_{L^2}^2 ds &\leq C ||\omega(0)||_{L^2}^2 + C \int_0^T \E(||\mathcal{U}'''||_{L^\infty} ||\omega_N(s)||_{L^2}^2 ds) + C(\nu, T)||\Psi||^2\\
    &\leq C(||\omega(0)||_{L^2}, ||\Psi||, \nu, T) + C \int_0^T \E ( ||\mathcal{W}||_{H^3}^2) + \E(||\omega_N(s)||_{L^2}^2) ds\\
   & \leq C(||\omega(0)||_{L^2}, ||\Psi||, ||\Phi||, \nu, T) < \infty
\end{split}
\end{equation}
Thus $\{\omega_N\}$ is uniformly bounded in $ L^2(\Omega ; C([0,T] ; L^2)) \cap L^2(\Omega; L^2([0,T] ; H^1))$.
From here, the proof proceeds as in the proof of well-posedness for stochastic Navier Stokes on $\mathbb{T} \times [-1,1]$ with Dirichlet boundary conditions. This is essentially the same as the proof on $\mathbb{T}^2$ due to the vanishing of vorticity on the boundary. The solution $\omega$ is formed as the (unique) limit of the $\omega_N$. In particular, there is a subsequence of the $\omega_N$ converging pointwise almost surely to $\omega$. Furthermore, $\omega$ is a probablistically strong (i.e pathwise) solution and can be shown to satisfy the SPDE in the analytically weak sense.\\

Suppose now the smallness assumption of Theorem \ref{well_posed_full}. Note that the $\omega_N$ which solve \eqref{fin_dim_syst} are in fact locally in $\mathcal{H}$. Let us define
\begin{equation}
\begin{split}
    &\tau_N \coloneqq\inf\{t \in [0,T \wedge \sigma] \; |\; \Ee_{\neq}(\omega_{N,\neq}) \geq c_* \nu^{-2\alpha} \vee \Ee_0(\omega_{N,0}) \geq c_* \nu^{-2\beta}\},\\
    &\tau_{\omega} \coloneqq\inf\{t \in [0,T \wedge \sigma] \; |\; \Ee_{\neq}(\omega_{\neq}) \geq c_* \nu^{-2\alpha} \vee \Ee_0(\omega_0) \geq c_* \nu^{-2\beta}\}.
\end{split}
\end{equation}
By adapting the arguments used in the \textit{a priori} arguments of Proposition \ref{energy_bounds_for_original}, we find for each $N$:
\begin{equation}
    \E \sup_{t \in [0, T \wedge \tau_N]} \Ee(\omega_N(t)) + \delta_* \E \int_0^{T \wedge \tau_N} \D(\omega_N (s)) ds \leq C( \Ee (\omega_{in}) + \nu^{4/3+2\alpha}||\Psi||^2 T).
\end{equation}
Taking the supremum in $N$ and noting that $\tau_{\omega} \leq \liminf_{N\to \infty} \tau_N$, we see that
\begin{equation}
    \E \sup_{t \in [0, T \wedge \tau_\omega]} \Ee(\omega(t)) + \delta_* \E \int_0^{T \wedge \tau_\omega} \D(\omega (s)) ds \leq C( \Ee (\omega_{in}) + \nu^{4/3+2\alpha}||\Psi||^2 T).
\end{equation}
By taking $T$ sufficiently small, we obtain the claim that $\omega$ is locally well-posed in $\mathcal{H}$ for small initial data. In fact one can show that $\nabla^\perp \Delta^{-1} \omega \cdot \nabla \omega \in L^1([0,T \wedge \tau_\omega] ; \mathcal{H}^*)$. Then it follows that $\omega  \in C([0,T \wedge \tau_{\omega}]; \mathcal{H})$ and the infinite dimensional It\^o formula is valid for the functional $\Ee(\cdot)$.

Note that the \textit{a priori} estimates of Proposition \ref{energy_bounds_for_original} are still needed. In the rescaled setting of \eqref{fast_slow_syst}, we have only show local well-posedness in $H$ for very short times, in a $\nu$ dependent manner. Proposition \ref{energy_bounds_for_original} enables us to extend the solution in $\mathcal{H}$ to long times with high probability.
\end{proof}

\begin{theorem}\label{well_posed_pseudo_lin}
For $\tilde{\omega}(0) \in \mathcal{H}$, $T > 0$, and under \eqref{noise_assumptions}, there exists a unique $\tilde{\omega} \in L^2(\Omega ; C([0,T] ; L^2)) \cap L^2(\Omega; L^2([0,T] ; H^1))$ such that for any $t \in [0,T]$ and $\phi \in D(-\Delta)$, the following holds $\Pb$-a.s:
    \begin{equation}\begin{split}
                \langle \tilde{\omega}, \phi\rangle_{L^2} &= \langle \tilde{\omega}(0), \phi\rangle_{L^2} + \int_0^t \langle \tilde{\omega}(s), (\nu\Delta-\mathcal{U}\partial_x + \partial_x \Delta^{-1} \mathcal{U}'')\phi\rangle_{L^2} ds- \int_0^t \langle  \tilde{\omega}(s), (\nabla^\perp \Delta^{-1} \tilde{\omega}(s))\cdot \nabla\phi \rangle_{L^2} ds\\ &\quad+ \int_0^t \langle  P_{\neq}\tilde{\omega}(s), (P_{\neq}(\nabla^\perp \Delta^{-1} \tilde{\omega}(s)))\cdot P_{\neq}\nabla\phi \rangle_{L^2} ds + \int_0^t \langle \Psi dW_s, \tilde{\omega}(s)\rangle_{L^2}.
    \end{split}
    \end{equation}
Furthermore, if $||\tilde\omega_{in, \neq}||_{\mathcal{H}_{\neq}} \leq c_* \nu^{1/2 + \alpha - \alpha'/2}$ and $||\tilde\omega_{in, 0}||_{\mathcal{H}_0} \leq c_* \nu^{1/2 + \beta - \alpha'}$, then there exists a positive stopping time $\tau$ such that $\tilde\omega \in  L^2(\Omega ; C([0,T \wedge \tau] ; \mathcal{H})) \cap L^2(\Omega; L^2([0,T \wedge \tau] ; \nabla \mathcal{H})$. Additionally, one can write $\tilde{\omega}_0 = \int_{\T} \tilde{\omega} dx$ in the mild form:
\begin{equation}
    \tilde{\omega}_0 = e^{\nu t \partial_y^2} \tilde{\omega}_0(0) + \int_0^t e^{\nu(t-s)\partial_y^2}\left(I- P_{\neq}\right)\left(\nabla^\perp \Delta^{-1}\tilde{\omega} \cdot \nabla \tilde{\omega}\right)ds.
\end{equation}
\end{theorem}
\begin{remark}
    It is clear that Theorem \ref{well_posed_pseudo_lin} implies well-posedness of \eqref{fast_slow_lin} and the mild formula \eqref{mild_form_of_pseudo_lin}. The proof is essentially the same as that of Theorem \eqref{well_posed_full}. Note that the projected non-linearity will still vanish in the $L^2$ inner product.
\end{remark}
\begin{theorem}\label{well_posed_fast_aux}
    Suppose \eqref{noise_assumptions} holds and let $(Y_0^\nu, X_0^\nu) \in \mathcal{H}$. Let $(\tilde{Y}_t^\nu, \tilde{X}_t^\nu)$ be the corresponding solution to \eqref{fast_slow_lin}. Suppose that $|| \tilde{X}_t^\nu||_{\Hh_0}^2 \leq c_* \nu^{-2\beta}$ up to an almost surely positive stopping time $\tau$. Then for $Y_0^\nu \in \mathcal{H}_{\neq}$, there exists a unique process $\hat{Y}_t^\nu \in  L^2(\Omega ; C([0,T\wedge \tau \wedge \sigma] ; \mathcal{H}_{\neq})) \cap L^2(\Omega; L^2([0,T \wedge \tau \wedge \sigma] ; \nabla \mathcal{H}_{\neq}))$ such that for $k \in \{0, \hdots, K\}$ and $t \in [k \delta, (k+1)\delta \wedge T \wedge\sigma]$:
    $$\hat{Y}_t^\nu = e^{(t-k\delta)L_{k\delta}}\hat{Y}_{k\delta}+ \nu^{-1/3 + \gamma/2}\int_{k\delta}^t e^{(t-s)L_{k\delta}}\Psi dW_s,$$
    where $L_{k\delta} = \nu^{\gamma}(\Delta + \nu^{-1} U_{k\delta} \partial_x + U_{k\delta}'' \partial_x \Delta^{-1}) - \nu^{\beta + \gamma - 1/2} b_m(\tilde{X}_{k\delta}^\nu, \cdot)$ and $\hat{Y}_0^\nu = Y_0^\nu$.
\end{theorem}
\begin{remark}
    The existence of a solution is clear, since on each time interval $\hat{Y}_t^\nu$ is an Ornstein-Uhlenbeck process and we have the estimate \eqref{semigroup_estimate} on the semi-group generated by the drift.
\end{remark}
\begin{theorem}
    Suppose \eqref{noise_assumptions} holds and let $(Y_0^\nu, X_0^\nu) \in \mathcal{H}$.  Suppose that $|| \tilde{X}_0^\nu||_{\mathcal{H}_0}^2 \leq c_* \nu^{-\alpha'}$ and let $\hat{Y}_t^\nu$ be the corresponding solution to \eqref{fast_aux_process}. Then there exists a unique process $\hat{X}_t^\nu$ which is in $L^2(\Omega ; C([0,T  \wedge \tau_2 \wedge \sigma] ; \mathcal{H}_0)) \cap L^2(\Omega; L^2([0,T\wedge \tau \wedge \sigma] ; \nabla \mathcal{H}_0))$ solving \eqref{slow_aux_process}, up to a positive stopping time $\tau_2$. Furthermore, we may write the solution as
    $$\hat{X}_t^\nu = e^{\nu^\gamma t \partial_y^2 } X_0^\nu - \nu^{\gamma/2-1/6} \int_0^t e^{\nu^{\gamma} (t-s) \partial_y^2} b_0(\hat{Y}_s) ds.$$
\end{theorem}
\begin{proof}
    This result follows from \eqref{well_posed_fast_aux} and the estimates \eqref{nonlinear_deterministic}. The proof is similar to that of Theorem \ref{theorem: averaged_well_posed}, only we are simply looking for local existence up to a stopping time.
\end{proof}

\section{Appendix: Technical Estimate on Nonlinearity}\label{appendix: b}
Here we prove the bound \eqref{tighter_bound_on_nonlin}. Note that \eqref{tighter_bound_on_nonlin} is directly implied by
\begin{equation}\label{eqn: substep}
    \begin{split}
        ||B_0(Y)||_{\mathcal{H}_0} \lesssim \mathcal{E}_{\neq}^{1/2} (Y)\left(\nu^{1/2 - \gamma/2} \D_{\mathfrak{t}}^{1/2}(Y) + \nu^{2/3-\gamma/2}\D_{\mathfrak{b}}^{1/2}(Y)\right),
    \end{split}
\end{equation}
where the implicit constant is independent of $\nu$. We note that \eqref{eqn: substep} is essentially contained in \cite{bedrossian2023stability}. We reproduce it here for clarity. First, we observe that
\begin{equation}
    \begin{split}
        ||B_0(Y)||_{\mathcal{H}_0}  &\lesssim || \sum_{k\neq0} ik \Delta_k^{-1} Y_k Y_{-k}||_{L^2} + c_{\mathfrak{a}} \nu^{1/3}|| \partial_y \sum_{k\neq0} ik \Delta_k^{-1} Y_k Y_{-k}||_{L^2} \eqqcolon T_0 + T_{\mathfrak{a}},
    \end{split}
\end{equation}
where $Y_k$ denotes the projection of $Y$ to the $x$-Fourier frequency $k$ and $\Delta_k = -k^2 + \partial_y^2$. The $T_0$ term is handled by H\"older's inequality and Gagliardo-Nirenberg-Sobolev:
\begin{equation}
    \begin{split}
        T_0 &\lesssim \sum_{k\neq0} || ik \Delta_k^{-1} Y_k||_{L^\infty} ||Y_{-k}||_{L^2}\\
        &\lesssim \left(\sum_{k\neq0} || ik \Delta_k^{-1} Y_k||_{L^\infty}^2\right)^{1/2} \left(\sum_{k\neq0} || Y_k||_{L^2}^2\right)^{1/2}\\
        &\lesssim \mathcal{E}_{\neq}^{1/2}(Y) \left(\sum_{k\neq 0} k||\partial_y \Delta_k^{-1} Y_k||_{L^2}||k \Delta_k^{-1} Y_k||_{L^2}\right)^{1/2} \\
        &\lesssim \nu^{1/2-\gamma/2} \mathcal{E}_{\neq}^{1/2}(Y) \mathcal{D}_{\mathfrak{t}}^{1/2}(Y).
    \end{split}
\end{equation}
For the $T_{\mathfrak{a}}$ term, we apply $\partial_y$ to the sum and estimate using Gagliardo-Nirenberg-Sobolev and elliptic regularity:
\begin{equation}
    \begin{split}
        T_{\mathfrak{a}} &\lesssim \nu^{1/3}\sum_{k\neq 0}\left( || ik \partial_y \Delta_k^{-1} Y_k||_{L^\infty} ||Y_{-k}||_{L^2} + ||ik \Delta_k^{-1} Y_k||_{L^\infty} ||\partial_y Y_{-k}||_{L^2}\right)\\
        &\lesssim \nu^{1/3}\sum_{k\neq 0} |k| ||\partial_y^2 \Delta_k^{-1} Y_k||_{L^2}^{1/2}||\partial_y\Delta_k^{-1} Y_k||_{L^2}^{1/2}  ||Y_{-k}||_{L^2}\\
        &\quad+ \nu^{1/3}\sum_{k\neq 0}
        |k||| \partial_y \Delta_k^{-1} Y_k||_{L^2}^{1/2}|| \Delta_k^{-1} Y_k||_{L^2}^{1/2} ||\partial_y Y_{-k}||_{L^2}\\
        &\lesssim \nu^{1/3}\sum_{k\neq 0} |k|^{1/2} ||Y_k||_{L^2} ||Y_{-k}||_{L^2} + \nu^{1/3} \sum_{k\neq 0} |k|^{5/6} || \nabla_k \Delta_k^{-1} Y_k||_{L^2} || |k|^{-1/3} \partial_y Y_{-k}||_{L^2}\\
        &\lesssim \nu^{2/3-\gamma/2} \mathcal{D}_{\mathfrak{b}}^{1/2}(Y) \Ee_{\neq}^{1/2}(Y) + \nu^{1/2-\gamma/2}\D_{\mathfrak{t}}^{1/2}(Y) \Ee_{\neq}^{1/2}(Y).
    \end{split}
\end{equation}
In the above, we have used the fact that $m > 1/2$. This completes the proof.


\section*{Declarations}

\subsubsection*{Acknowledgments}
This material is based upon work supported by the National Science Foundation Graduate Research Fellowship under Grant No. DGE-2034835 and Grant No. DGE-2444110. Any opinions, findings, and conclusions or recommendations expressed in this material are those of the authors and do not necessarily reflect the views of the National Science Foundation. JB was supported by DMS-2108633.

\subsubsection*{Conflicts of Interest}
The authors have no competing interests to declare that are relevant to the content of this article.

\subsubsection*{Data Availability}
Data sharing not applicable to this article as no datasets were generated or analyzed during the current study.
\\


\bibliographystyle{plain} 
\bibliography{citations} 

\end{document}